\documentclass{amsart}

\usepackage[colorlinks=true, urlcolor=black, citecolor=black, linkcolor=black, hyperfootnotes=true]{hyperref}
\usepackage{amssymb}
\usepackage{cancel}
\usepackage{aliascnt}
\usepackage{mathscinet}
\usepackage{graphicx}
\usepackage{accents}
\usepackage{enumitem}
\usepackage{mathrsfs}
\usepackage{mathtools}
\usepackage{xcolor}

\usepackage{trimclip}
\def\@clipped@vdash{\raise.6ex\hbox{\clipbox{0pt .6ex 0pt .6ex}{$\vdash$}}}
\newcommand*\vDdash{\mathrel{\ooalign{$\vdash$\cr\raise.3ex\hbox{\@clipped@vdash}\cr\raise -.3ex\hbox{\@clipped@vdash}}}}

\newcommand{\sqin}{\mathrel{\vphantom{\sqsubset}\text{\mathsurround=0pt\ooalign{$\sqsubset$\cr$-$\cr}}}}

\numberwithin{equation}{section}

\newtheorem{thm}{Theorem}[section]

\newaliascnt{prp}{thm}
\newtheorem{prp}[prp]{Proposition}
\aliascntresetthe{prp}

\newaliascnt{cor}{thm}
\newtheorem{cor}[cor]{Corollary}
\aliascntresetthe{cor}

\theoremstyle{definition}

\newaliascnt{dfn}{thm}
\newtheorem{dfn}[dfn]{Definition}
\aliascntresetthe{dfn}

\newaliascnt{xpl}{thm}
\newtheorem{xpl}[xpl]{Example}
\aliascntresetthe{xpl}

\newaliascnt{rmk}{thm}
\newtheorem{rmk}[rmk]{Remark}
\aliascntresetthe{rmk}

\author{Tristan Bice}
\author{Wies{\l}aw Kubi\'s}
\email{bice@math.cas.cz}
\email{kubis@math.cas.cz}
\thanks{Tristan Bice is supported by the GA\v{C}R project 22-07833K and RVO: 67985840.\\
Wies{\l}aw Kubi\'s is supported by the GA\v{C}R project EXPRO 20-31529X and RVO: 67985840.\\
Both authors are researchers at the Institute of Mathematics of the Czech Academy of Sciences.}

\title[Vickers Duality for Stably Locally Compact Spaces]{Lattice-Free and Point-Free: Vickers Duality for Subbases of Stably Locally Compact Spaces}

\begin{document}

\begin{abstract}
Inspired by classic work of Wallman and more recent work of Jung-Kegelmann-Moshier and Vickers, we show how to encode general subbases of stably locally compact spaces via certain entailment relations.  We further build this up to a categorical duality encompassing the classic Priestley-Stone duality and its various extensions to stably locally compact spaces by Shirota, De Vries, Hofmann-Lawson (in the stable case), Jung-S\"underhauf, Hansoul-Poussart, Bezhanishvili-Jansana, van Gool and Bice-Starling.
\end{abstract}

\maketitle

\section*{Introduction}

\subsection*{Background}

One of the most famous dualities in mathematics is the classic Stone duality \cite{Stone1936} between Boolean algebras and Stone spaces, i.e. zero-dimensional compact Hausdorff spaces, which has by now found numerous applications in various areas of topology, algebra and analysis.  Equally important are its many and varied extensions developed by a large number of people over the past 80 years.

Stone himself \cite{Stone1938} was the first to extend his duality, namely to distributive lattices and spectral spaces, which Priestley \cite{Priestley1970} later showed how to view as ordered Stone spaces.  Dualities of spectral spaces with more general distributive semilattices have also been examined more recently by Hansoul-Poussart \cite{HansoulPoussart2008} and Bezhanishvili-Jansana \cite{BezhanishviliJansana2011}.  A somewhat different duality with more general compactly based sober spaces was also obtained by Gr\"atzer \cite{Gratzer1978}, which was recently extended to even non-distributive semilattices by Celani-Gonzalez \cite{CelaniGonzalez2020}.

In a somewhat different direction, Shirota \cite{Shirota1952} and de Vries \cite{deVries1962} extended Stone's original duality to (certain bases of) arbitrary compact Hausdorff spaces via compingent lattices, i.e. distributive lattices with an additional transitive relation representing compact containment.  More recently, Jung-S\"underhauf \cite{JungSunderhauf1995} and van Gool \cite{vanGool2012} obtained a similar duality between certain proximity lattices and stably compact spaces\footnote{Compactness of the spaces considered in \cite{JungSunderhauf1995} and \cite{vanGool2012} is implicit from the fact compact saturated sets are closed under finite intersections -- in particular the empty intersection, i.e. the whole space, must be compact.  Confusingly, some older works (e.g. \cite{Johnstone1986}) call these spaces `stably locally compact', while for us this only means that compact saturated subsets are closed under pairwise (equivalently, non-empty finite) intersections -- if the whole space is also compact then we call it `stably compact'.  This is consistent with modern usage as in \cite{vanGool2012} and \cite{Goubault2013}.}, a natural common generalisation of spectral spaces and compact Hausdorff spaces.  Very recently, a further extension to distributive $\vee$-predomains (which Kawai calls `localized strong proximity $\vee$-semilattices' -- see \cite{Kawai2021}) and arbitrary locally compact sober spaces was even obtained in \cite{Bice2021GHLJS}.

These results also have a well-known frame theoretic counterpart, namely the Hofmann-Lawson duality \cite{HofmannLawson1978} between continuous frames and (the entire open set lattices of) locally compact sober spaces.  This fits in nicely with the frame-theoretic approach to general point-free topology, which by now is firmly established (e.g. see \cite{PicadoPultr2012}).  The only problem with frames is their infinitary nature, as the lattices involved are required to be complete and satisfy an infinite version of distributivity.  As a result, non-trivial frames are invariably uncountable and difficult to construct from scratch, as well as being impervious to tools from classic model theory, like those used in Fra\"iss\'e theory.  Indeed, for these and other reasons, logicians have long been interested in more constructive `predicative' analogs of frames that avoid any quantification over infinite subsets.

In contrast to frames, the lattices above representing mere bases of locally compact spaces are much more like the finitary structures one usually encounters in abstract algebra.  But to what extent is the lattice structure really needed here?  Might there be other finitary ways of encoding stably locally compact spaces?  Could we even encode general subbases of such spaces in some finitary, perhaps more relational way?  Could this lead to the unification of various Stone-type dualities and constructions?  Might this yield a natural way of obtaining stable compactifications (see \cite{Lawson1991} and \cite{BezhanishviliHarding2014}) or even more general stabilisations, much as compingent/proximity lattices have been used to obtain Hausdorff compactifications?

To answer these questions, we can take some inspiration from classic work of Wallman \cite{Wallman1938} in the slightly different realm of compact $T_1$ spaces.  Wallman's key idea was to encode these spaces using finite subbasic covers.  Specifically, given any subbasis $S$ of open subsets of some compact $T_1$ space $X$, we can form an abstract simplicial complex consisting of the finite subsets of $S$ that do not cover $X$.  Wallman further showed that any abstract simplicial complex can be represented as a subbasis of such a space, thus yielding a very general kind of Stone-like duality.

For a stably locally compact space $X$, subbasic covers of the entirety of $X$ do not always suffice to completely encode the space.  Instead, we must consider covers, or rather compact covers, of smaller subsets, e.g. formed from finite intersections of the subbasis $S$ in question.  This leads us to consider `cover relations' $\vdash$ on finite subsets $\mathsf{F}S=\{F\subseteq S:|F|<\infty\}$ of the subbasis $S$ defined by
\[F\vdash G\qquad\Leftrightarrow\qquad\bigcap F\Subset\bigcup G\]
(here $\Subset$ is compact containment, i.e. $O\Subset N$ means that every open cover of $N$ has a finite subcover of $O$ or, equivalently, $O\subseteq K\subseteq N$, for some compact $K$).

The questions one then needs to address are the following:
\begin{enumerate}
\item How can we characterise these cover relations abstractly?
\item How can we recover the space $X$ from the abstract cover system $(S,\vdash)$?
\item How can we further axiomatise continuous maps in similar abstract terms?
\end{enumerate}
The present paper is devoted to answering these questions, thus obtaining a categorical duality encompassing the classic Priestley-Stone duality and various extensions to stably locally compact spaces.  As a simple application of our results, we also show that every core compact $T_0$ space $X$ has a minimal stabilisation, a moderate extension of some well known results on stable compactifications.

\subsection*{Related Work}

The cover relations above turn out to be very similar to some entailment relations considered previously in various logical systems, e.g. see \cite{Scott1974}, \cite{JungKegelmannMoshier1999}, \cite{CederquistCoquand2000}, \cite{CoquandZhang2003}, \cite{Vickers2004} and \cite{Kawai2020}.  Indeed, our cover relations are more general than Scott's entailment relations and can also be viewed as encompassing Kawai's strong continuous entailment relations.  On the other hand, our cover relations are slightly more restrictive than Vickers' entailment relations, which are `monotone cut-idempotents' in our terminology and which are essentially connective-free versions of Jung-Kegelmann-Moshier's consequence relations.  Indeed, this resolves a number of issues with general monotone cut-idempotents, without eliminating any isomorphism classes from Vickers' category.  Connections of entailment systems with proximity lattices and stably continuous frames (where the way-below relation respects finite meets) have also been investigated in some detail in this previous work.

One of our goals is to obtain a similar duality with more general arithmetic lattices (frames where the way-below relation respects only pairwise meets), which involves a slightly different interpretation of what entailment systems represent.  But more importantly, what really distinguishes our work is that we are interested in dualities with concrete spaces, not just their lattice or frame theoretic counterparts.  Of course, once a duality with certain lattices or frames is established (e.g. as done by Vickers in \cite{Vickers2004}), one can immediately use one of the dualities already available (e.g. Hofmann-Lawson duality) to obtain another duality with concrete spaces.  The down-side is that this relegates the duality with concrete spaces to a secondary position, ignoring the untapped potential of cover relations to conversely extend and unify previous dualities for stably locally compact spaces.

As a case in point, say we follow Vickers and consider a monotone cut-idempotent $\vdash$ on $\mathsf{F}S$.  The quasi-ideals relative to $\vdash$ then form an arithmetic lattice (see \autoref{MCI} below) whose irreducible elements then form a stably locally compact space.  However, these are certain subsets of $\mathsf{F}S$, while if we want to generalise the dualities already established then our space should rather be constructed from filter-like subsets of $S$.  Another issue is that the principal quasi-ideals do not necessarily generate the frame, at least not in the usual sense via meets and joins.  This means the elements of $S$ do not always get represented as a subbasis of the corresponding stably locally compact space (see \autoref{DyadicXpl} below).

Our slightly more restrictive cover relations resolve these issues.  Indeed, we construct the corresponding stably locally compact space directly from `tight' subsets of $S$ with a topology such that $S$ itself immediately gets represented as a subbasis of the space.  For cover relations arising canonically from proximity lattices or stably continuous frames, these tight subsets are precisely the rounded proximal/prime filters -- previous dualities thus fall out as immediate corollaries.  Moreover, Vickers' original category is essentially preserved, as monotone cut-idempotents are always isomorphic in his category to one of our cover relations (see \autoref{Karoubiisomorphic} below).

\section{Preliminaries}\label{Preliminaries}

Given any sets $S$ and $T$, we view any subset of their product $\vdash\ \subseteq S\times T$ as a relation `from $T$ to $S$'.  Any $Q\subseteq S$ then defines existential and universal `polar' subsets of $T$ which we denote by superscripts and subscripts respectively, i.e.
\begin{align*}
Q^\vdash&=\{t\in T:\exists q\in Q\ (q\vdash t)\}.\\
Q_\vdash&=\{t\in T:\forall q\in Q\ (q\vdash t)\}.
\end{align*}
With respect to inclusion, these are order preserving and reversing respectively, i.e.
\begin{align*}
Q\subseteq R\qquad&\Rightarrow\qquad Q^\vdash\subseteq R^\vdash.\\
Q\subseteq R\qquad&\Rightarrow\qquad R_\vdash\subseteq Q_\vdash.
\end{align*}
Also note these agree for singletons, i.e. $\{q\}^\vdash=\{q\}_\vdash$, in which case we often drop the curly braces and just write $q^\vdash$ or $q_\vdash$.  We also let $\dashv\ \subseteq T\times S$ denote the opposite relation of $\vdash$, i.e. $t\dashv s$ means $s\vdash t$.
We define the composition of $\vdash$ with any other relation $\vDash\ \subseteq R\times S$ existentially in the usual way, i.e. $\vDash\circ\vdash\ \subseteq R\times T$ is given by
\[r\vDash\circ\vdash t\qquad\Leftrightarrow\qquad\exists s\in S\ (r\vDash s\vdash t)\qquad\Leftrightarrow\qquad r^\vDash\cap t^\dashv\neq\emptyset.\]

We denote the finite subsets of $S$ by
\[\mathsf{F}S=\{F\subseteq S:|F|<\infty\}.\]
The canonical universal extensions ${}_\forall\hspace{-3pt}\vdash\ \subseteq\mathsf{F}S\times T$ and $\vdash_\forall\ \subseteq S\times\mathsf{F}T$ are defined by
\begin{align*}
G\mathrel{{}_\forall\hspace{-3pt}\vdash}t\qquad&\Leftrightarrow\qquad\forall g\in G\ (g\vdash t).\\
s\vdash_\forall H\qquad&\Leftrightarrow\qquad\forall h\in H\ (s\vdash h).
\end{align*}
In single instances of these relations, we will usually drop the subscripts, as long as there is no possibility of confusion.  Indeed, it will be convenient to adopt this convention even for infinite subsets, e.g. for any $Q\subseteq S$ and $R\subseteq T$,
\[Q\vdash R\qquad\text{means}\qquad\forall q\in Q\ \forall r\in R\ (q\vdash r).\]
Likewise, $Q\vdash t$ means $q\vdash t$, for all $q\in Q$ (equivalently $Q\subseteq t_\dashv$ or $t\in Q_\vdash$).

Conversely, for any relations $\vdash\ \subseteq S\times\mathsf{F}T$ and $\vDash\ \subseteq\mathsf{F}S\times T$, we denote their singleton restrictions in $S\times T$ by $\vdash_1$ and ${}_1\hspace{-3pt}\vDash$ respectively, i.e.
\begin{align*}
s\vdash_1t\qquad&\Leftrightarrow\qquad s\vdash\{t\}.\\
s\mathrel{{}_1\hspace{-3pt}\vDash}t\qquad&\Leftrightarrow\qquad\{s\}\vDash t.
\end{align*}
Again, however, we will usually drop the subscripts/curly braces in single instances of these relations, e.g. for any $s\in S$ and $t\in T$,
\[s\vdash t\qquad\text{means}\qquad s\vdash\{t\}.\]

Finally, we will also have occasion to consider the canonical existential extension $\vdash_\exists\ \subseteq S\times\mathsf{F}T$ of a relation $\vdash\ \subseteq S\times T$ defined by
\[p\vdash_\exists G\qquad\Leftrightarrow\qquad\exists g\in G\ (p\vdash g).\]
Equivalently, $\vdash_\exists\ =\ \vdash\circ\in_T$, where $\in_T$ is the membership relation on $T\times\mathsf{F}T$.  In particular, any relation $\vdash\ \subseteq S\times\mathsf{F}T$ gives rise to another relation between the same sets $\vdash_{1\exists}\ \subseteq S\times\mathsf{F}T$, i.e. where
\[p\vdash_{1\exists}G\qquad\Leftrightarrow\qquad\exists g\in G\ (p\vdash\{g\}).\]

We call a relation $\vdash\ \subseteq S\times\mathsf{F}T$ \emph{upper} if, for all $s\in S$, $t\in T$ and $G\in\mathsf{F}T$,
\[\tag{Upper}s\vdash G\qquad\Rightarrow\qquad s\vdash G\cup\{t\}.\]
Equivalently, $\vdash$ is upper if $s\vdash G\subseteq H\in\mathsf{F}T$ implies $s\vdash H$, i.e.
\[\tag{Upper$'$}\vdash\circ\subseteq_{\mathsf{F}T}\ \ \subseteq\ \ \vdash,\]
where $\subseteq_{\mathsf{F}T}$ denotes the inclusion relation on $\mathsf{F}T$.  Note that
\[\vdash\text{ is upper}\qquad\Rightarrow\qquad\vdash_{1\exists}\ \subseteq\ \vdash,\]
i.e. $\vdash_{1\exists}$ is a strengthening of $\vdash$, one which will play an important role in our work.  Likewise, we call a relation $\vdash\ \subseteq\mathsf{F}S\times T$ \emph{lower} if, for all $s\in S$, $t\in T$ and $F\in\mathsf{F}S$,
\[\tag{Lower}F\vdash t\qquad\Rightarrow\qquad F\cup\{s\}\vdash t.\]
We call a relation $\vdash\ \subseteq\mathsf{F}S\times\mathsf{F}T$ \emph{monotone} if it is both upper and lower.

We call a relation $\vdash$ on $\mathsf{F}S$ (i.e. $\vdash\ \subseteq\mathsf{F}S\times\mathsf{F}S$) a \emph{cut relation} if Gentzen's cut rule holds, which is just the converse of monotonicity, i.e. for all $F,G\in\mathsf{F}S$ and $s\in S$,
\[\tag{Cut}\label{Cut}\{s\}\cup F\vdash G\quad\text{and}\quad F\vdash G\cup\{s\}\qquad\Rightarrow\qquad F\vdash G.\]
We call $\vdash$ \emph{$1$-reflexive} if ${}_1\hspace{-3pt}\vdash_1$ is reflexive, i.e. for all $s\in S$,
\[\tag{$1$-Reflexive}s\vdash s.\]
Note that if $\vdash$ is monotone and $1$-reflexive then $\vdash$ itself is almost reflexive too, specifically $F\vdash F$, for all non-empty finite $F\subseteq S$.  However, usually reflexivity does not hold for the empty set, i.e. $\emptyset\nvdash\emptyset$.  Indeed, the only monotone relation satisfying $\emptyset\vdash\emptyset$ is $\mathsf{F}S\times\mathsf{F}S$ itself, i.e. where $F\vdash G$, for all finite $F,G\subseteq S$.

Finally, let us denote the singleton subsets of any $T\subseteq S$ by
\[\mathsf{1}T=\{\{t\}:t\in T\}.\]
Note that if $\vdash$ on $\mathsf{F}S$ is both $1$-reflexive and upper then $\in_S\ \subseteq\ \vdash$, i.e. $f\vdash F$, whenever $f\in F\in\mathsf{F}S$.  Put another way, this is saying $\mathsf{1}F\vdash F$ or, equivalently, $F{}_{\forall1}\hspace{-3pt}\vdash F$, for all $F\in\mathsf{F}S$.  Likewise,
\[\in_S\ \subseteq\ \dashv\qquad\Leftrightarrow\qquad\forall F\in\mathsf{F}S\ (F\vdash\mathsf{1}F)\qquad\Leftrightarrow\qquad\forall F\in\mathsf{F}S\ (F\vdash_{1\forall}F)\]
and these equivalent conditions hold whenever $\vdash$ is both $1$-reflexive and lower.

\section{Entailments}

\begin{dfn}
An \emph{entailment} is a monotone cut relation.
\end{dfn}

So a relation $\vdash$ on $\mathsf{F}S$ is an entailment if, for all finite $F,G\subseteq S$ and $s\in S$,
\[\tag{Entailment}\label{Entailment}\{s\}\cup F\vdash G\quad\text{and}\quad F\vdash G\cup\{s\}\qquad\Leftrightarrow\qquad F\vdash G.\]
Note that, unlike some other authors, we do not require entailments to be $1$-reflexive -- we call a $1$-reflexive entailment a \emph{Scott relation}.

\subsection{Examples}\label{Examples}

The canonical example of a Scott relation in the field of logic is where $S$ is a collection of propositions and $F\vdash G$ indicates that at least one of the propositions in $G$ holds whenever all propositions in $F$ hold.  There are also a number of less well known entailments arising in other areas of mathematics, as we now demonstrate.  Indeed, the primary motivating examples for our work come instead from topology and order theory.

\begin{xpl}[Dense Covers]\label{DenseCovers}
When $S$ is a family of open sets of a space $X$, we can take $F\vdash G$ to mean the union of $G$ is dense in the intersection of $F$, i.e.
\[\tag{Dense Cover}\label{Dcl}F\vdash G\qquad\Leftrightarrow\qquad\bigcap F\subseteq\mathrm{cl}(\bigcup G).\]
Indeed, if $\bigcap F\subseteq\mathrm{cl}(\bigcup G)$ then certainly $p\cap\bigcap F\subseteq\mathrm{cl}(\bigcup G)$ and $\bigcap F\subseteq\mathrm{cl}(\bigcup G\cup p)$.  Conversely, say $p\cap\bigcap F\subseteq\mathrm{cl}(\bigcup G)$ and $\bigcap F\subseteq\mathrm{cl}(\bigcup G\cup p)$.  Note that
\[\mathrm{cl}(p)\cap\bigcap F\subseteq\mathrm{cl}(p\cap\bigcap F)\subseteq\mathrm{cl}(\bigcup G)\]
(for any open neighbourhood $O$ of $x\in\mathrm{cl}(p)\cap\bigcap F$, we see that $O\cap\bigcap F$ is again an open neighbourhood of $x$ so $O\cap\bigcap F\cap p\neq\emptyset$, showing that $x\in\mathrm{cl}(p\cap\bigcap F)$).  Thus
\[\bigcap F=\mathrm{cl}(\bigcup G\cup p)\cap\bigcap F\subseteq\mathrm{cl}(\bigcup G)\cup(\mathrm{cl}(p)\cap\bigcap F)\subseteq\mathrm{cl}(\bigcup G),\]
i.e. $F\vdash G$, showing that $\vdash$ is an entailment.  We also immediately see that $\vdash$ is $1$-reflexive and hence a Scott relation.

In particular, we can consider this example when $X$ is discrete, in which case
\[\tag{Containment}\label{Containment}F\vdash G\qquad\Leftrightarrow\qquad\bigcap F\subseteq\bigcup G.\]
In fact, what we really want to do is replace $\subseteq$ here with compact containment $\Subset$ on a core compact space.  This will result in an entailment that is no longer $1$-reflexive but still possesses other important properties \textendash\, see \eqref{SubsetC} below.
\end{xpl}

Much like the example from logic mentioned above, the canonical Scott relation on a distributive lattice $S$ is given by
\[\tag{$\bigwedge$-$\bigvee$-Cover}\label{MeetJoinCover}F\vdash G\qquad\Leftrightarrow\qquad\bigwedge F\leq\bigvee G.\]
This can be extended to semilattices, even those that are not distributive.  Indeed, this is what will allow us to view the semilattice dualities of Hansoul-Poussart \cite{HansoulPoussart2008} and Bezhanishvili-Jansana \cite{BezhanishviliJansana2011} as special cases of our duality.

\begin{xpl}[Semilattices]\label{veeSemilattices}
If $S$ is a $\vee$-semilattice with minimum $0(=\bigvee\emptyset)$ then we have a Scott relation $\vdash$ on $\mathsf{F}S$ defined by
\[\tag{$\bigvee$-Cover}\label{veeDensity}F\vdash G\quad\Leftrightarrow\quad\forall p,q\in S\ \Big(\forall f\in F\ (p\leq f\vee q)\ \Rightarrow\ p\leq\bigvee G\vee q\Big)\]
Indeed, $1$-reflexivity and monotonicity are immediate.  Conversely, say $\{s\}\cup F\vdash G$ and $F\vdash G\cup\{s\}$.  Take $p,q\in P$ and say $p\leq f\vee q$, for all $f\in F$, and hence $p\leq s\vee\bigvee G\vee q$, as $F\vdash G\cup\{s\}$.  Then certainly $p\leq f\vee\bigvee G\vee q$ too, for all $f\in F$, and hence $p\leq\bigvee G\vee\bigvee G\vee q=\bigvee G\vee q$, as $\{s\}\cup F\vdash G$.  This shows that $F\vdash G$ so the cut rule also holds and hence $\vdash$ is a Scott relation.

Note that $\vdash$ reduces to $\leq$ on singletons.  In fact, for all $f\in S$ and $G\in\mathsf{F} S $,
\begin{equation}\label{fDGbigvee}
f\vdash G\qquad\Leftrightarrow\qquad f\leq\bigvee G.
\end{equation}
Indeed, if $f\leq\bigvee G$ then certainly $p\leq f\vee q$ implies $p\leq\bigvee G\vee q$.  Conversely, if $f\vdash G$ then, as $f\leq f\vee\bigvee G$, the definition of $\vdash$ yields $f\leq\bigvee G\vee\bigvee G=\bigvee G$.

When $S$ is distributive, $\vdash$ has a simpler reduction even for $F\in\mathsf{F}S $, namely
\begin{equation}\label{vDistributive}
F\vdash G\qquad\Leftrightarrow\qquad F_\geq\leq\bigvee G,
\end{equation}
(recall from our conventions in \autoref{Preliminaries} that the latter expression unpacks to mean $p\leq\bigvee G$ whenever $p\leq F$, i.e. whenever $p$ is below every $f\in F$).  Indeed, if $F\vdash G$ and $p\leq F$ then certainly $p\leq f\vee\bigvee G$, for all $f\in F$, so $p\leq\bigvee G\vee\bigvee G=\bigvee G$, showing that $F_\geq\leq\bigvee G$.  Conversely, say $F_\geq\leq\bigvee G$ and $p\leq f\vee q$, for all $f\in F$.  Letting $F=\{f_1,\cdots,f_n\}$, distributivity yields $f_1'\leq f_1$ and $q_1\leq q$ with $p=f_1'\vee q_1$.  As $f_1'\leq p\leq f_2\vee q$, distributivity again yields $f_2'\leq f_2$ and $q_2\leq q$ with $f_1'=f_2'\vee q_2$.  Continuing in this way, we get $f'_n\in F_\geq$ with $p\leq f'_n\vee q\leq\bigvee G\vee q$, as $F_\geq\leq\bigvee G$, showing that $F\vdash G$.

When $S$ is a distributive lattice, $F_\geq=(\bigwedge F)_\geq$ and so we immediately see that this characterisation of $\vdash$ is equivalent to \eqref{MeetJoinCover} above (where $\bigwedge\emptyset=1$ if $S$ has a maximum $1$, otherwise $\emptyset\nvdash G$, for all $G\in\mathsf{F}S$).  This even characterises distributive lattices.  To see this, assume \eqref{MeetJoinCover} holds and take any $e,f,q\in S$.  Certainly $e\wedge f\leq e\wedge f$ so \eqref{MeetJoinCover} yields $\{e,f\}\vdash e\wedge f$.  As $p\leq e\vee q$ and $p\leq f\vee q$, for $p=(e\vee q)\wedge(f\vee q)$, it follows that $p\leq(e\wedge f)\vee q\leq p$, i.e. $(e\wedge f)\vee q=(e\vee q)\wedge(f\vee q)$, showing that $S$ is distributive.
\end{xpl}

The above is a `top down' way of defining Scott relations from $\vee$-semilattices.  There is also a different `bottom up' construction for arbitrary posets and even more general transitive relations.  As a result, our duality will also encompass previous work of Bice-Starling -- the Scott relation $\vdash$ below corresponds to $\precapprox$ in \cite[Definition 6.1]{BiceStarling2016} and $\widehat{\mathsf{D}}$ in \cite[\S1.3]{BiceStarling2020HTight}, which are in turn based on Lenz's arrow relation \cite[Definition 5.1]{Lenz2008} and Exel's covers \cite[Definition 15.3]{Exel2008}.

\begin{xpl}[Ordered Sets]\label{OrderedSets}
Given a relation $<$ on a set $S$, we define $\perp$ on $S$ by
\[s\perp t\qquad\Leftrightarrow\qquad\{s,t\}_>=\emptyset.\]
For any $Q\subseteq S$, it follows that
\[Q_\perp=\{s\in S:s^>\cap Q^>=\emptyset\}.\]
If $<$ is transitive then we have an entailment $\vdash$ on $\mathsf{F}S$ defined by
\[\tag{$\perp$-Cover}\label{wedgeperpDensity}F\vdash G\qquad\Leftrightarrow\qquad F_>\cap G_\perp=\emptyset.\]
Indeed, monotonicity follows from the fact that $G\subseteq H$ implies $H_\perp\subseteq G_\perp$ and $H_>\subseteq G_>$.
Conversely, say $\{s\}\cup F\vdash G$ and $F\vdash G\cup\{s\}$.  If we had $f\in F_>\cap G_\perp$ then $f\notin s_\perp$ -- otherwise $f\in F_>\cap G_\perp\cap s_\perp=F_>\cap(G\cup\{s\})_\perp=\emptyset$, as $F\vdash G\cup\{s\}$, a contradiction.  Then $f_>\subseteq F_>\cap G_\perp$, as $<$ is transitive, so this would mean that we have $p\in\{f,s\}_>\subseteq(\{s\}\cup F)_>\cap G_\perp=\emptyset$, as $\{s\}\cup F\vdash G$, again a contradiction.  Thus $F_>\cap G_\perp=\emptyset$, i.e. $F\vdash G$, proving the cut rule and showing $\vdash$ is an entailment.

If $S$ is also \emph{rounded} in that $S\subseteq S^<$, i.e. $s^>\neq\emptyset$, for all $s\in S$, then $\vdash$ is also $1$-reflexive and hence a Scott relation.  Indeed, $s^>\neq\emptyset$ is equivalent to $s\notin s_\perp$ and hence $s^>\cap s_\perp=\emptyset$, again by transitivity, i.e. $s\vdash s$.

In this case, we can also view $\vdash$ as a special case of \eqref{Dcl}, once we identify each $s\in S$ with $s^>$.  To see this, consider $S$ in the Alexandroff topology consisting of all down-sets, i.e.
\[\mathcal{O}(S)=\{O\subseteq S:\forall s\in O\ (s^>\subseteq O)\}.\]
For any $s\in S$, note $\{s\}\cup s^>$ is open (as $<$ is transitive) and hence the smallest open set containing $s$.  Also $s^>\cap O=\emptyset$ implies $s\notin O$, for any $s\in S$ and $O\in\mathcal{O}(S)$ (as $s^>\neq\emptyset$ because $S$ is rounded).  Thus $s\notin\mathrm{cl}(G^>)$ iff $(\{s\}\cup s^>)\cap G^>=\emptyset$ iff $s^>\cap G^>=\emptyset$ iff $s\in G_\perp$, i.e. $S\setminus\mathrm{cl}(G^>)=G_\perp$.  For all $F,G\in\mathsf{F}S$, it follows that
\[F\vdash G\qquad\Leftrightarrow\qquad F_>\subseteq S\setminus G_\perp\qquad\Leftrightarrow\qquad\bigcap_{f\in F}f^>\subseteq\mathrm{cl}\Big(\bigcup_{g\in G}g^>\Big).\]
\end{xpl}

Although not closely related to our work here, convexity spaces are another area where entailments crop up naturally (see \cite{VanDeVel1993} and \cite{Kubis2002}).

\begin{xpl}[Convexity Spaces]\label{ConvexitySpaces}
A \emph{convexity} on a set $S$ is a family $\mathcal{C}\subseteq\mathcal{P}(S)$ of subsets of $S$ that is closed under arbitrary intersections and directed unions (including the empty intersection and union so necessarily $\emptyset,S\in \mathcal{C}$).  The elements of $\mathcal{C}$ are said to be \emph{convex}.  If $H\in \mathcal{C}$ and $S\setminus H\in \mathcal{C}$ then $H$ is a \emph{half-space}.  We call a convexity \emph{Kakutani} if disjoint convex sets are separated by half-spaces, i.e.
\[\tag{Kakutani}C,D\in \mathcal{C}\text{ and }C\cap D=\emptyset\ \ \Rightarrow\ \ \exists H\ (C\subseteq H\in \mathcal{C}\text{ and }D\subseteq S\setminus H\in \mathcal{C}).\]
The \emph{convex hull} $\mathrm{conv}(T)$ of any $T\subseteq S$ is the convex set generated by $T$, i.e. the smallest convex set containing $T$
\[\mathrm{conv}(T)=\bigcap\{C\in\mathcal{C}:T\subseteq C\}.\]
A \emph{polytope} is the convex hull $\mathrm{conv}(F)$ of some $F\in\mathsf{F}S$.  Convexities are always determined by their polytopes in that
\[C\in\mathcal{C}\qquad\Leftrightarrow\qquad\forall F\in\mathsf{F}C\ (\mathrm{conv}(F)\subseteq C).\]

Given a convexity $\mathcal{C}$ on a set $S$, we define $ \vdash$ on $\mathsf{F}S$ by
\[F \vdash G\qquad\Leftrightarrow\qquad\mathrm{conv}(F)\cap\mathrm{conv}(G)\neq\emptyset.\]
We immediately see that $\vdash$ is monotone, $1$-reflexive and symmetric.  We claim that $\vdash$ satisfies the cut rule precisely when $\mathcal{C}$ is Kakutani.  Indeed, if $\mathcal{C}$ is Kakutani and $F\not \vdash G$ then we have a half-space $H\supseteq F$ with $H\cap G=\emptyset$.  For any $s\in S$, either $s\in H$, in which case $\{s\}\cup F\not \vdash G$, or $s\notin H$, in which case $F\not \vdash G\cup\{s\}$, proving the cut rule.  Conversely, say the cut rule holds and we are given disjoint $C,D\in \mathcal{C}$.  By Kuratowski-Zorn, we have maximal $H\in\mathcal{C}$ with $C\subseteq H\subseteq S\setminus D$ and maximal $I\in \mathcal{C}$ with $D\subseteq I\subseteq S\setminus H$.  If we had $s\in S\setminus(H\cup I)$ then maximality would yield finite $F\subseteq H$ such that we have $d\in\mathrm{conv}(F\cup\{s\})\cap D$, as well as finite $G\subseteq I$ such that we have $h\in H\cap\mathrm{conv}(G\cup\{s\})$.  But then $F\cup\{h,s\} \vdash G\cup\{d\}$ and $F\cup\{h\} \vdash G\cup\{d,s\}$ and hence $F\cup\{h\} \vdash G\cup\{d\}$, by the cut rule, even though $F\cup\{h\}\subseteq H\in\mathcal{C}$, $G\cup\{d\}\subseteq I\in\mathcal{C}$ and $H\cap I=\emptyset$, a contradiction.  Thus we must have had $S=H\cup I$, i.e. $H$ is a half-space separating $C$ and $D$, so $S$ is Kakutani.

In other words, Kakutani convexities yield Scott relations.  This even extends to pairs of convexities on the same set \textendash\, see \cite{Kubis2002}.
\end{xpl}

Another simple example of a symmetric Scott relation is given by the meet relation on $\mathsf{F}S$, which we denote by $\between$ as in \cite{Vickers2004}, i.e.
\[\tag{Meet Relation}F\between G\qquad\Leftrightarrow\qquad F\cap G\neq\emptyset.\]
This can be verified directly or noted immediately as a special case of the above example where $\mathcal{C}=\mathsf{P}S=\{T:T\subseteq S\}$ is the discrete convexity.  Note that a monotone relation $\vdash$ on $\mathsf{F}S$ is $1$-reflexive precisely when $\between\ \subseteq\ \vdash$.

A natural example of a symmetric entailment that is not $1$-reflexive comes from the diagonal relation $\bowtie$ on $\mathsf{FF}S$.  As in \cite{Vickers2004}, the diagonal relation will also play a fundamental role soon when it comes to defining cut-composition.  First consider the \emph{selections} and \emph{supersets} of any $\mathcal{Q}\subseteq\mathsf{F}S$ given by
\begin{align*}
\tag{Selections}\mathcal{Q}_\between&=\{G\in\mathsf{F}S:\forall F\in\mathcal{Q}\ (F\between G)\}.\\
\tag{Supersets}\mathcal{Q}^\subseteq&=\{G\in\mathsf{F}S:\exists F\in\mathcal{Q}\ (F\subseteq G)\}.
\end{align*}
Note if $S\subseteq\mathsf{P}X$ is itself a family of subsets of some set $X$ then, for any $\mathcal{F}\in\mathsf{FF}S$,
\[\bigcap_{F\in\mathcal{F}}\bigcup F=\bigcup_{F\in\mathcal{F}_\between}\bigcap F.\]
Accordingly, we think of $\mathcal{F}$ and $\mathcal{F}_\between$ as representing the same set in different ways.

\begin{xpl}
The \emph{diagonal relation} $\bowtie$ on $\mathsf{FF}S$ is defined by
\[\mathcal{F}\bowtie\mathcal{G}\qquad\Leftrightarrow\qquad\mathcal{F}_\between\subseteq\mathcal{G}^\subseteq\qquad\Leftrightarrow\qquad\forall F\in\mathcal{F}_\between\ \exists G\in\mathcal{G}\ (G\subseteq F).\]
In other words, $\mathcal{F}$ is diagonal to $\mathcal{G}$ if every selection of $\mathcal{F}$ contains a subset in $\mathcal{G}$.  Again to aid one's intuition, note if $S\subseteq\mathsf{P}X$ is a family of subsets of a set $X$ then
\[\mathcal{F}\bowtie\mathcal{G}\qquad\Rightarrow\qquad\bigcap_{F\in\mathcal{F}}\bigcup F\subseteq\bigcup_{G\in\mathcal{G}}\bigcap G.\]
Thus we think of $\bowtie$ as signifying containment of the sets represented by $\mathcal{F}$ and $\mathcal{G}$.

Like the relations from \autoref{ConvexitySpaces}, the diagonal relation is symmetric, i.e.
\[\mathcal{F}\bowtie\mathcal{G}\qquad\Leftrightarrow\qquad\mathcal{G}\bowtie\mathcal{F},\]
for all $\mathcal{F},\mathcal{G}\in\mathsf{F}\mathsf{F} S  $.  To see this, note that $\mathcal{F}_{\between\between}=\mathcal{F}^\subseteq$, $(\mathcal{F}^\subseteq)_\between=(\mathcal{F}_\between)^\subseteq=\mathcal{F}_\between$ and
\[\mathcal{F}\subseteq\mathcal{G}\quad\Rightarrow\quad\mathcal{G}_\between\subseteq\mathcal{F}_\between.\]
So if $\mathcal{F}\bowtie\mathcal{G}$ then $\mathcal{F}_\between\subseteq\mathcal{G}^\subseteq$ and hence $\mathcal{G}_\between=(\mathcal{G}^\subseteq)_\between\subseteq\mathcal{F}_{\between\between}=\mathcal{F}^\subseteq$, i.e. $\mathcal{G}\bowtie\mathcal{F}$.\linebreak  Alternatively, this is immediate from the symmetric characterisation of $\bowtie$ in \cite[Lemma 2.2]{JungKegelmannMoshier1999}, specifically $\mathcal{F}\bowtie\mathcal{G}$ if and only if $F\between G$, for all $F\in\mathcal{F}_\between$ and $G\in\mathcal{G}_\between$, i.e.
\[\mathcal{F}\bowtie\mathcal{G}\qquad\Leftrightarrow\qquad\mathcal{F}_\between\ \between\ \mathcal{G}_\between\]

We claim that the diagonal relation restricted to $\mathsf{F}(\mathsf{F}S\setminus\{\emptyset\})$ is an entailment.  To see that $\bowtie$ is monotone, note $\mathcal{F}\bowtie\mathcal{G}$ implies $(\{E\}\cup\mathcal{F})_\between\subseteq\mathcal{F}_\between\subseteq\mathcal{G}^\subseteq\subseteq(\mathcal{G}\cup\{E\})^\subseteq$ and hence $\{E\}\cup\mathcal{F}\bowtie\mathcal{G}$ and $\mathcal{F}\bowtie\mathcal{G}\cup\{E\}$, for all $E\in\mathsf{F}S$.  Conversely, if $\mathcal{F}\bowtie\mathcal{G}\cup\{E\}$ then, for any $F\in\mathcal{F}_\between$, either $F\in\mathcal{G}^\subseteq$ or $E\subseteq F$.  If $\{E\}\cup\mathcal{F}\bowtie\mathcal{G}$ too then $\emptyset\neq E\subseteq F$ implies that $F\in(\mathcal{F}\cup\{E\})_\between$ and hence $F\in\mathcal{G}^\subseteq$ again.  So either way $\mathcal{F}_\between\subseteq\mathcal{G}^\subseteq$, i.e. $\mathcal{F}\bowtie\mathcal{G}$.  Thus the cut rule is also satisfied and $\bowtie$ is an entailment on $\mathsf{F}(\mathsf{F}S\setminus\{\emptyset\})$.

However $\bowtie$ is not $1$-reflexive and hence not a Scott relation.  Indeed, as long as we have distinct $s,t\in S$ then $\{\{s,t\}\}\nvdash\{\{s,t\}\}$ because $\{s,t\}\nsubseteq\{s\}\in\{\{s,t\}\}_\between$.
\end{xpl}

\subsection{Composition}\label{Composition}
Entailments are always $1$-transitive, i.e.
\[\tag{$1$-Transitive}\label{1Transitivity}p\vdash q\vdash r\qquad\Rightarrow\qquad p\vdash r,\]
for any entailment $\vdash$ on $\mathsf{F}S$ and any $p,q,r\in S$.  Indeed, $p\vdash q\vdash r$ implies $p\vdash\{q,r\}$ and $\{p,q\}\vdash r$, by monotonicity, and hence $p\vdash r$, by the cut rule.  More symbolically, this can be stated as ${}_1\hspace{-3pt}\vdash_1\circ\ {}_1\hspace{-3pt}\vdash_1\ \subseteq\,{}_1\hspace{-3pt}\vdash_1$.  In this form, transitivity can be extended to $\vdash$ itself as long as we replace standard composition with `cut-composition' from \cite{JungKegelmannMoshier1999} and \cite{Vickers2004} (so named due to its connections with the cut rule).

\begin{dfn}
For any relations $\vdash\ \subseteq R\times\mathsf{F}S$ and $\vDash\ \subseteq\mathsf{F}S\times T$, we define
\[\vdash\bullet\vDash\quad=\quad\vdash_\forall\circ\bowtie\circ\mathrel{{}_\forall\hspace{-3pt}\vDash}\]
The resulting relation $\vdash\bullet\vDash\ \subseteq R\times T$ is called the \emph{cut-composition} of $\vdash$ and $\vDash$.
\end{dfn}

More explicitly, $r\vdash\bullet\vDash t$ means we can find $\mathcal{F},\mathcal{G}\in\mathsf{FF}S$ with $\mathcal{F}\bowtie\mathcal{G}$ such that $r\vdash F$, for all $F\in\mathcal{F}$, and $G\vDash t$, for all $G\in\mathcal{G}$.  If $\vDash$ is a lower relation, this is equivalent to saying we have $\mathcal{F}\in\mathsf{FF}S$ such that $r\vdash F$, for all $F\in\mathcal{F}$, and $G\vDash t$, for all $G\in\mathcal{F}_\between$.  Alternatively, if $\vdash$ is an upper relation, this is equivalent to saying that we have $\mathcal{G}\in\mathsf{FF}S$ such that $r\vdash F$, for all $F\in\mathcal{G}_\between$, and $G\vDash t$, for all $G\in\mathcal{G}$.

The next result will be useful to show cut-composition is associative.  For the proof, it will be convenient to introduce the $\wedge$ operation defined on $\mathcal{Q},\mathcal{R}\in\mathsf{PF}S$ by
\[\mathcal{Q}\wedge\mathcal{R}=\{F\cup G:F\in\mathcal{Q}\text{ and }G\in\mathcal{R}\}.\]
Note $\wedge$ and $\cup$ are dual relative to selections, i.e. for any $\mathcal{Q},\mathcal{R}\in\mathsf{PF}S$,
\[(\mathcal{Q}\wedge\mathcal{R})_\between=\mathcal{Q}_\between\cup\mathcal{R}_\between\qquad\text{and}\qquad(\mathcal{Q}\cup\mathcal{R})_\between=\mathcal{Q}_\between\wedge\mathcal{R}_\between=\mathcal{Q}_\between\cap\mathcal{R}_\between.\]
Also $\wedge$ is associative and we thus extend it to finite sequences in the usual way, e.g. if $\mathcal{Q}_1,\ldots,\mathcal{Q}_3\in\mathsf{PF}S$ then $\bigwedge_{k\leq 3}\mathcal{Q}_k=((\mathcal{Q}_1\wedge\mathcal{Q}_2)\wedge\mathcal{Q}_3)=(\mathcal{Q}_1(\wedge\mathcal{Q}_2\wedge\mathcal{Q}_3))$.

\begin{prp}
For any $\vdash\ \subseteq R\times\mathsf{F}S$ and lower $\vDash\ \subseteq\mathsf{F}S\times T$,
\begin{equation}\label{diamondforall}
(\vdash\bullet\vDash)_\forall\ =\ \vdash\bullet\vDash_\forall.
\end{equation}
\end{prp}

\begin{proof}
Take any $r\in R$ and $H\in\mathsf{F}T$.  If $r\vdash\mathcal{F}\bowtie\mathcal{G}\vDash H$ then it certainly follows that $r\vdash_\forall\circ\bowtie\circ\mathrel{{}_\forall\hspace{-3pt}\vDash}h$, for all $h\in H$, i.e. $r\vdash\bullet\vDash H$.  Conversely, $r\vdash\bullet\vDash H$ means that, for each $h\in H$, we have $\mathcal{F}_h,\mathcal{G}_h\in\mathsf{FF}S$ with $r\vdash\mathcal{F}_h\bowtie\mathcal{G}_h\vDash h$.  Letting $\mathcal{F}=\bigcup_{h\in H}\mathcal{F}_h$ and $\mathcal{G}=\bigwedge_{h\in H}\mathcal{G}_h$, it follows that $r\vdash\mathcal{F}\bowtie\mathcal{G}\vDash H$, as $\vDash$ is a lower relation.  This completes the proof of \eqref{diamondforall}.
\end{proof}

Likewise, for any upper $\vdash\ \subseteq R\times\mathsf{F}S$ and any $\vDash\ \subseteq\mathsf{F}S\times T$,
\begin{equation}\label{foralldiamond}
{}_\forall(\vdash\bullet\vDash)\ =\ {}_\forall\hspace{-3pt}\vdash\bullet\vDash.
\end{equation}
Thus, for any upper $\vdash\ \subseteq R\times\mathsf{F}S$, monotone $\vDash\ \subseteq\mathsf{F}S\times\mathsf{F}T$ and lower $\vDdash\ \subseteq\mathsf{F}T\times U$,
\begin{align*}
(\vdash\bullet\vDash)\ \bullet\vDdash\ &=\ (\vdash\bullet\vDash)_\forall\ \circ\bowtie\circ\mathrel{{}_\forall\hspace{-3pt}\vDdash}\\
&=\ \ \vdash_\forall\circ\bowtie\circ\mathrel{{}_\forall\hspace{-3pt}\vDash_\forall}\circ\bowtie\circ\mathrel{{}_\forall\hspace{-3pt}\vDdash}\\
&=\ \ \vdash_\forall\circ\bowtie\circ\mathrel{{}_\forall(\vDash\bullet\vDdash)}\\
&=\ \ \vdash\bullet\ (\vDash\bullet\vDdash).
\end{align*}
In particular, cut-composition is associative on monotone relations, as similarly shown in \cite[Lemma 3.2]{JungKegelmannMoshier1999} and \cite[Proposition 32]{Vickers2004}.

Given $\vdash\ \subseteq R\times\mathsf{F}S$, it will be convenient to define $\vdash_\between$ on $R\times\mathsf{FF}S$ by
\[r\vdash_\between\mathcal{F}\qquad\Leftrightarrow\qquad r\vdash\mathcal{F}_\between,\]
i.e. $r\vdash_\between\mathcal{F}$ means $r\vdash G$, for all $G\in\mathcal{F}_\between$.  We immediately see that the relations $({}_\forall\hspace{-3pt}\vdash)_\between$ and ${}_\forall(\vdash_\between)$ coincide, so ${}_\forall\hspace{-3pt}\vdash_\between$ is unambiguous.  Also, if $\vdash$ is an upper relation,
\[\vdash_\between\ \ =\ \ \vdash_\forall\circ\bowtie.\]
The definition of cut-composition can then be rewritten as
\[\vdash\bullet\vDash\ \ =\ \ \vdash_\between\circ\mathrel{{}_\forall\hspace{-3pt}\vDash}.\]
By \eqref{diamondforall} and \eqref{foralldiamond}, the operation $\vdash\ \mapsto{}_\forall\hspace{-3pt}\vdash_\between$ then turns cut-composition into standard composition, i.e. for any upper $\vdash\ \subseteq R\times\mathsf{F}S$ and monotone $\vDash\ \subseteq\mathsf{F}S\times\mathsf{F}T$,
\[{}_\forall(\vdash\bullet\vDash)_\between\ =\ \ \mathrel{{}_\forall(\vdash\bullet\vDash)_\forall}\circ\bowtie\ \ =\ {}_\forall\hspace{-3pt}\vdash_\forall\circ\bowtie\circ\mathrel{{}_\forall\hspace{-3pt}\vDash}_\forall\circ\bowtie\ \ =\ {}_\forall\hspace{-3pt}\vdash_\between\circ\mathrel{{}_\forall\hspace{-3pt}\vDash_\between}.\]
Also, $\vdash$ is completely determined by ${}_\forall\hspace{-3pt}\vdash_\between$ as $r\vdash F$ is equivalent to $\{r\}\mathrel{{}_\forall\hspace{-3pt}\vdash_\between}\mathsf{1}F$.  It follows that the associativity of standard composition passes to cut-composition on monotone relations, thus providing an alternative proof of this fact.

Here is another simple but useful observation about the operation $\vdash\ \mapsto\ \vdash_\between$.

\begin{prp}
For any relation $\vdash\ \subseteq R\times S$,
\begin{equation}\label{existsbetween}
\vdash_{\exists\between}\ \ =\ \ \vdash_{\forall\exists}.
\end{equation}
\end{prp}

\begin{proof}
If $r\vdash_{\forall\exists}\mathcal{F}$ then we have $F\in\mathcal{F}$ such that $r\vdash F$, i.e. $r\vdash f$, for all $f\in F$.  As every $G\in\mathcal{F}_\between$ contains an element of $F$, it follows that $r\vdash_\exists G_\between$.  Conversely, if $r\vdash_{\forall\exists}\mathcal{F}$ fails then there is no $F\in\mathcal{F}$ such that $r\vdash F$, i.e. for every $F\in\mathcal{F}$, we have $g_F\in F$ with $r\nvdash g_F$.  Letting $G=\{g_F:F\in\mathcal{F}\}\in\mathcal{F}_\between$, it follows that $r\vdash_\exists G$ fails and hence $r\vdash_\exists\mathcal{F}_\between$ fails.
\end{proof}

For any other $\vDash\ \subseteq\mathsf{F}S\times T$, the above result yields
\begin{equation}\label{existsdiamond}
\vdash_\exists\bullet\vDash\ \ =\ \ \vdash_{\exists\between}\circ\mathrel{{}_\forall\hspace{-3pt}\vDash}\ \ =\ \ \vdash_{\forall\exists}\circ\mathrel{{}_\forall\hspace{-3pt}\vDash}\ \ =\ \ \vdash_\forall\circ\vDash.
\end{equation}
We then see that cut-composition respects the operation $\vdash\ \mapsto\ \vdash_{1\exists}$.

\begin{cor}\label{1existsdiamond}
For any relations $\vdash\ \subseteq R\times\mathsf{F}S$ and $\vDash\ \subseteq\mathsf{F}S\times\mathsf{F}T$,
\[\vdash_{1\exists}\bullet\vDash_{1\exists}\ \ \subseteq\ (\vdash\bullet\vDash)_{1\exists}.\]
\end{cor}

\begin{proof}
By \eqref{existsdiamond}, $\vdash_{1\exists}\bullet\vDash_{1\exists}\ =\ \vdash_{1\forall}\circ\vDash_{1\exists}$.  So if $r\vdash_{1\exists}\bullet\vDash_{1\exists}G$ then we have $F\in\mathsf{F}S$ and $g\in G$ with $r\vdash_{1\forall}F\vDash g$.  This means $r\vdash\mathsf{1}F\bowtie F\vDash g$ so $r\vdash\bullet\vDash g$ and hence $r\mathrel{(\vdash\bullet\vDash)_{1\exists}}G$, as required.
\end{proof}

\subsection{Auxiliarity}

We wish to show that entailments are \emph{cut-transitive}, i.e.
\[\tag{Cut-Transitivity}\vdash\bullet\vdash\ \subseteq\ \vdash.\]
This will imply ${}_\forall\hspace{-3pt}\vdash_\between\circ\mathrel{{}_\forall\hspace{-3pt}\vdash_\between}\ ={}_\forall(\vdash\bullet\vdash)_\between\subseteq{}_\forall\hspace{-3pt}\vdash_\between$, i.e. ${}_\forall\hspace{-3pt}\vdash_\between$ is genuinely transitive on $\mathsf{FF}S$.  To prove this, we first note the cut rule extends to finite subsets as follows.

\begin{prp}\label{FTrans}
If $\vdash$ is an upper cut relation on $\mathsf{F}S$ and $F,G,H\in\mathsf{F}S$ then
\[\tag{$\mathsf{F}$-Cut}\label{FCut}\forall h\in H\ (\{h\}\cup F\vdash G)\quad\text{and}\quad F\vdash G\cup H\qquad\Rightarrow\qquad F\vdash G.\]
\end{prp}

\begin{proof}
Say $\vdash$ is an upper cut relation on $\mathsf{F}S$.  If $\{h\}\cup F\vdash G$, for all $h\in H$, and $F\vdash G\cup H$ then, for any $h\in H$, upperness yields $\{h\}\cup F\vdash G\cup(H\setminus\{h\})$ and hence $F\vdash G\cup(H\setminus\{h\})$, by the cut rule.  For any other $h'\in H$, we again see that $\{h'\}\cup F\vdash G\cup(H\setminus\{h,h'\})$ and hence $F\vdash G\cup(H\setminus\{h,h'\})$.  Continuing in this way we get $F\vdash G$, proving \eqref{FCut}.
\end{proof}

Likewise we see that, for any lower cut relation $\vdash$ and any $F,G,H\in\mathsf{F}S$,
\[\tag{Dual-$\mathsf{F}$-Cut}\label{DualFCut}H\cup F\vdash G\quad\text{and}\quad\forall h\in H\ (F\vdash G\cup\{h\})\qquad\Rightarrow\qquad F\vdash G.\]
Indeed, this is just \autoref{FTrans} applied to the opposite relation $\dashv$ instead of $\vdash$.  Modifying these results using another relation $\vDash$ yields natural notions of auxiliarity.

\begin{dfn}
A relation $\vdash\ \subseteq\mathsf{F}S\times\mathsf{F}T$ is \emph{auxiliary} to a relation $\vDash$ on $\mathsf{F}S$ if
\[\tag{Auxiliarity}\label{Auxiliarity}\forall h\in H\ (\{h\}\cup F\vdash G)\quad\text{and}\quad F\vDash H\qquad\Rightarrow\qquad F\vdash G,\]
for all $F,H\in\mathsf{F}S$ and $G\in\mathsf{F}T$.  We call $\vdash$ \emph{dual-auxiliary} to $\vDash$ if $T=S$ and
\[\tag{Dual-Auxiliarity}\label{DualAuxiliarity}H\vdash G\quad\text{and}\quad\forall h\in H\ (F\vDash G\cup\{h\})\qquad\Rightarrow\qquad F\vdash G,\]
for all $F,G,H\in\mathsf{F}S$.  We call $\vdash$ a \emph{semicut} relation if $\vdash$ is self-dual-auxiliary.
\end{dfn}

More explicitly, a relation $\vdash$ on $\mathsf{F}S$ is a semicut relation if, for all $F,G,H\in\mathsf{F}S$,
\[\tag{Semicut}\label{Semicut}H\vdash G\quad\text{and}\quad\forall h\in H\ (F\vdash G\cup\{h\})\qquad\Rightarrow\qquad F\vdash G.\]
Note it follows immediately from \eqref{DualFCut} that
\[\vdash\text{ is a lower cut relation}\qquad\Rightarrow\qquad\vdash\text{ is a semicut relation}.\]
Often (dual-)auxiliarity can be expressed via cut-composition as $\vDash\bullet\vdash\ \subseteq\ \vdash$.

\begin{prp}\label{DiamondvsAuxiliarity}
Take any $\vDash\ \subseteq\mathsf{F}S\times\mathsf{F}S$ and $\vdash\ \subseteq\mathsf{F}S\times\mathsf{F}T$.
\begin{enumerate}
\item\label{DiamondToAux} If $\vDash$ is a $1$-reflexive lower relation and $\vDash\bullet\vdash\ \subseteq\ \vdash$ then $\vdash$ is auxiliary to $\vDash$.
\item\label{AuxToDiamond} Conversely, if $\vdash$ is a lower relation and $\vdash$ is auxiliary to $\vDash$ then $\vDash\bullet\vdash\ \subseteq\ \vdash$.
\item\label{DualAux} If $S=T$, $\vDash$ is upper, $\vdash\ \subseteq\ \vDash$ and $\vdash$ is dual-auxiliary to $\vDash$ then $\vDash\bullet\vdash\ \subseteq\ \vdash$.
\end{enumerate}
\end{prp}

\begin{proof}\
\begin{enumerate}
\item If $\vDash$ is a $1$-reflexive lower relation then, for any $F,H\in\mathsf{F}S$ and $G\in\mathsf{F}T$ such that $F\vDash H$ and $\{h\}\cup F\vdash G$, for all $h\in H$,
\[F\ \vDash\ \{H\}\cup\mathsf{1}F\ \bowtie\ \{\{h\}\cup F:h\in H\}\ \vdash\ G.\]
If $\vDash\bullet\vdash\ \subseteq\ \vdash$ then it follows that $F\vdash G$, showing that $\vdash$ is auxiliary to $\vDash$.

\item Say $\vdash$ is a lower relation auxiliary to $\vDash$.  Take $F_0\in\mathsf{F}S$ and $G_0\in\mathsf{F}T$ with $F_0\vDash\bullet\vdash G_0$, so we have $\mathcal{F},\mathcal{G}\in\mathsf{FF}S$ such that $F_0\vDash\mathcal{F}\bowtie\mathcal{G}\vdash G_0$.  Let $\mathcal{F}=\{F_1,\cdots,F_n\}$.  For each $F\in\mathcal{F}_\between$, we have $G\in\mathcal{G}$ with $F\supseteq G\vdash G_0$ so $F\cup F_0\vdash G_0$, as $\vdash$ is lower.  Put another way, $\{f\}\cup F\cup F_0\vdash G_0$, for all $F\in(\mathcal{F}\setminus\{F_1\})_\between$ and $f\in F_1$.  As $F_0\vDash F_1$, \eqref{Auxiliarity} yields $F\cup F_0\vdash G_0$.  But again this means $\{f\}\cup F\cup F_0\vdash G_0$, for all $F\in(\mathcal{F}\setminus\{F_1,F_2\})_\between$ and $f\in F_2$, and hence $F\cup F_0\vdash G_0$, by \eqref{Auxiliarity}, as $F_0\vDash F_2$.  Continuing in this way, we eventually get $F_0\vdash G_0$, showing that $\vDash\bullet\vdash\ \subseteq\ \vdash$.

\item Similarly, say $S=T$ and $\vdash$ is dual-auxiliary to some weaker upper relation $\vDash$.  Take $F_0,G_0\in\mathsf{F}S$ with $F_0\vDash\bullet\vdash G_0$, which means we have $\mathcal{F},\mathcal{G}\in\mathsf{FF}S$ such that $F_0\vDash\mathcal{F}\bowtie\mathcal{G}\vdash G_0$.  Let $\mathcal{G}=\{G_1,\cdots,G_n\}$.  For each $G\in\mathcal{G}_\between$, we have $F\in\mathcal{F}$ with $F_0\vDash F\subseteq G$ and hence $F_0\vDash G_0\cup G$, as $\vDash$ is upper.  Put another way, this means $F_0\vDash G_0\cup G\cup\{g\}$, for all $G\in(\mathcal{G}\setminus\{G_1\})_\between$ and $g\in G_1$.  As $G_1\vDash G_0$, \eqref{DualAuxiliarity} yields $F_0\vdash G_0\cup G$.  As $\vdash\ \subseteq\ \vDash$, for all $G\in(\mathcal{G}\setminus\{G_1,G_2\})_\between$ and $g\in G_2$, this again yields $F_0\vDash G_0\cup G\cup\{g\}$ and hence $F_0\vdash G_0\cup G$, by \eqref{DualAuxiliarity}, as $G_2\vdash G_0$.  Continuing in this way, we eventually get $F_0\vdash G_0$, showing that $\vDash\bullet\vdash\ \subseteq\ \vdash$. \qedhere
\end{enumerate}
\end{proof}

By \autoref{DiamondvsAuxiliarity} \eqref{DualAux}, every upper semicut relation is cut-transitive.  By \eqref{DualFCut}, this applies to all entailments.  Likewise, every lower self-auxiliary relation is cut-transitive, by \autoref{DiamondvsAuxiliarity} \eqref{AuxToDiamond}, which again applies to all entailments, by \eqref{FCut}.  Either argument thus yields the following.

\begin{cor}\label{DiagonalTransitive}
Entailments are always cut-transitive.
\end{cor}

A Scott relation $\vdash$ on $\mathsf{F}S$ is even \emph{cut-idempotent}, i.e.
\[\tag{Cut-Idempotence}\vdash\ =\ \vdash\bullet\vdash.\]
Indeed, if $\vdash$ is a $1$-reflexive upper relation then $F\vdash G$ implies $F\vdash\{G\}\bowtie\mathsf{1}G\vdash G$, showing that $\vdash\ \subseteq\ \vdash\bullet\vdash$.  While coherent frames can be encoded with Scott relations, the key insight of \cite{JungKegelmannMoshier1999} and \cite{Vickers2004} was that stably continuous frames can be encoded with more general cut-idempotent monotone relations.

\subsection{Divisibility}

To more precisely encode subbases of stably locally compact spaces we need to consider a slightly stronger version of the reverse inclusion above.  Indeed, note that any $1$-reflexive $\vdash$ on $\mathsf{F}S$ satisfies $g\vdash_{1\exists}G$, whenever $g\in G\in\mathsf{F}S$, and so the above observation actually shows that $\vdash$ is \emph{divisible}, meaning that
\[\tag{Divisibility}\vdash\ \ \subseteq\ \ \vdash\bullet\vdash_{1\exists}.\]

Alternatively, divisibility can be stated as a strong interpolation condition for $\vdash_\between$.  First note that, for any $\vdash$ on $\mathsf{F}S$, \eqref{existsbetween} allows us to define $\vdash_*$ on $\mathsf{FF}S$ by
\[\vdash_*\ \ =\ \ {}_\forall(\vdash_{1\forall\exists})\ \ =\ \ {}_\forall(\vdash_{1\exists\between}).\]
So $\mathcal{F}\vdash_*\mathcal{G}$ iff, for all $F\in\mathcal{F}$, we have $G\in\mathcal{G}$ with $F\vdash\mathsf{1}G$, i.e. $F\vdash g$, for all $g\in G$.  As long as $\vdash$ is upper, this implies $F\vdash H$, for all $H\in\mathcal{G}_\between$, i.e.
\[\vdash_*\ \ \subseteq\ {}_\forall\hspace{-3pt}\vdash_\between.\]
If $\vdash$ is also lower and divisible then \eqref{diamondforall} yields
\[\vdash_\between\ \ \subseteq\ (\vdash\bullet\vdash_{1\exists})_\between\ =\ \ \vdash_\between\circ\mathrel{{}_\forall(\vdash_{1\exists})_\forall}\circ\bowtie\ \ =\ \ \vdash_\between\circ\mathrel{{}_\forall(\vdash_{1\exists\between})}\ \ =\ \ \vdash_\between\circ\vdash_*.\]
On the other hand, if $\vdash_\between\ \ \subseteq\ \ \vdash_\between\circ\vdash_*$ then $F\vdash G$ implies $F\vdash_\between\mathsf{1}G$ so $F\vdash_\between\circ\vdash_* \mathsf{1}G$ and hence $F\vdash_\between\circ\vdash_{1\exists}G$.  Thus, for monotone $\vdash$,
\begin{equation}\label{BetweenInterpolation}
\vdash\text{ is divisible}\qquad\Leftrightarrow\qquad\vdash_\between\ \ \subseteq\ \ \vdash_\between\circ\vdash_*.
\end{equation}

Let us call any monotone divisible cut-transitive $\vdash$ a \emph{strong idempotent}, i.e. any monotone $\vdash$ satisfying $\vdash\bullet\vdash\ \subseteq\ \vdash\ \subseteq\ \vdash\bullet\vdash_{1\exists}$.  Under montonicity, $\vdash_{1\exists}\ \subseteq\ \vdash$ and so these inclusions actually become equalities, i.e. $\vdash$ is a strong idempotent when
\[\tag{Strong Idempotence}\label{StrongIdempotence}\supseteq_{\mathsf{F}S}\circ\vdash\circ\subseteq_{\mathsf{F}S}\ \ =\ \ \vdash\ \ =\ \ \vdash\bullet\vdash\ \ =\ \ \vdash\bullet\vdash_{1\exists}.\]
In particular, any strong idempotent is indeed cut-idempotent.  Also, by \autoref{DiamondvsAuxiliarity} \eqref{DualAux} (with $\vDash\ =\ \vdash$), any monotone divisible semicut relation is cut-transitive and thus a strong idempotent.  In fact, the converse also holds.

\begin{prp}\label{DiamondvsAuxiliarity2}
If $\vdash$ is a divisible upper relation on $\mathsf{F}S$ then
\[\vdash\text{ is a semicut relation}\qquad\Leftrightarrow\qquad\vdash\text{ is cut-transitive}.\]
\end{prp}

\begin{proof}
As noted above, if $\vdash$ is an upper semicut relation then $\vdash$ is cut-transitive, thanks to \autoref{DiamondvsAuxiliarity} \eqref{DualAux}.  Conversely, say $\vdash$ is a divisible upper cut-transitive relation on $\mathsf{F}S$ and take $F,G,H\in\mathsf{F}S$ such that $H\vdash G$ and $F\vdash G\cup\{h\}$, for all $h\in H$.  For all $h\in H$, divisibility and upperness yield $\mathcal{G}_h\vdash G$ and $\mathcal{H}_h\vdash h$ such that $F\vdash(\mathcal{G}_h\cup\mathcal{H}_h)_\between$.  Letting $\mathcal{H}=\bigwedge_{h\in H}\mathcal{H}_h$, note $\mathcal{H}\vdash\mathsf{1}H\bowtie H\vdash G$ so $\mathcal{H}\vdash G$, by cut-transitivity.  Setting $\mathcal{I}=\bigcup_{h\in H}\mathcal{G}_h\cup\mathcal{H}$, then $\mathcal{I}\vdash G$.  Moreover, $F\vdash\mathcal{I}_\between$ because
\[\mathcal{I}_\between=\bigcap_{h\in H}\mathcal{G}_{h\between}\cap(\bigcup_{h\in H}\mathcal{H}_{h\between})=\bigcup_{h\in H}\bigcap_{i\in H}\mathcal{G}_{i\between}\cap\mathcal{H}_{h\between}\subseteq\bigcup_{h\in H}\mathcal{G}_{h\between}\cap\mathcal{H}_{h\between}=\bigcup_{h\in H}(\mathcal{G}_h\cup\mathcal{H}_h)_\between\]
Thus $F\vdash G$, again by cut-transitivity, showing that $\vdash$ is a semicut relation.
\end{proof}

Thus we have the following alternative characterisation of strong idempotents:
\[\vdash\text{ is a strong idempotent}\quad\Leftrightarrow\quad\vdash\text{is a monotone divisible semicut relation}.\]
By \autoref{DiagonalTransitive}, it thus follows that any divisible entailment is a strong idempotent.  Conversely, we can prove the cut rule for strong idempotents which are auxiliary to some weaker $1$-reflexive relation.

\begin{prp}\label{DiamondCut}
If $\vDash$ is a $1$-reflexive lower relation on $\mathsf{F}S$ and $\vdash$ is an upper relation on $\mathsf{F}S$ satisfying $\vDash\bullet\vdash\ \subseteq\ \vdash\ \subseteq\ \vDash\bullet\vdash_{1\exists}$ then $\vdash$ is a cut relation.
\end{prp}

\begin{proof}
Say $\{s\}\cup F\vdash G$ and $F\vdash G\cup\{s\}$.  As $\vdash\ \subseteq\ \vDash\bullet\vdash_{1\exists}$, we have $\mathcal{F},\mathcal{G}\in\mathsf{FF}S$ such that $F\vDash\mathcal{F}\bowtie\mathcal{G}\vdash_{1\exists}G\cup\{s\}$.  Let $\mathcal{H}=\{H\in\mathcal{G}:H\vdash s\}$.  For any $H\in\mathcal{H}$,
\[F\cup H\ \vDash\ \{\{s\}\}\cup\mathsf{1}F\ \bowtie\ \{\{s\}\cup F\}\ \vdash\ G,\]
as $\vDash$ is 1-reflexive and lower, and hence $F\cup H\vdash G$, as $\vDash\bullet\vdash\ \subseteq\ \vdash$.  It follows that
\begin{equation}\label{FHG}
F\ \vDash\ \mathsf{1}F\cup\mathcal{F}\ \bowtie\ \{F\cup H:H\in\mathcal{H}\}\cup(\mathcal{G}\setminus\mathcal{H})\ \vdash\ G.
\end{equation}
Indeed, $F\vDash\mathsf{1}F\cup\mathcal{F}$ follows again from the fact that $\vDash$ is a 1-reflexive lower relation.  The middle diagonal relation follows from
\begin{align*}
(\{F\cup H:H\in\mathcal{H}\}\cup(\mathcal{G}\setminus\mathcal{H}))_\between&=\{F\cup H:H\in\mathcal{H}\}_\between\cap(\mathcal{G}\setminus\mathcal{H})_\between\\
&=(\{F\}_\between\cup\mathcal{H}_\between)\cap(\mathcal{G}\setminus\mathcal{H})_\between\\
&\subseteq\{F\}_\between\cup(\mathcal{H}_\between\cap(\mathcal{G}\setminus\mathcal{H})_\between)=\{F\}_\between\cup\mathcal{G}_\between\\
&\subseteq(\mathsf{1}F)^\subseteq\cup\mathcal{F}^\subseteq=(\mathsf{1}F\cup\mathcal{F})^\subseteq.
\end{align*}
For the last relation, note $H\vdash_{1\exists}G\cup\{s\}$ and $H\nvdash s$ implies $H\vdash_{1\exists}G$ and hence $H\vdash G$, as $\vdash$ is an upper relation.  This means $\mathcal{G}\setminus\mathcal{H}\vdash G$, while $F\cup H\vdash G$ was already noted above, completing the verification of \eqref{FHG}.  As $\vDash\bullet\vdash\ \subseteq\ \vdash$, it follows that $F\vdash G$, showing that $\vdash$ is a cut relation.
\end{proof}

We can use the above result to prove \cite[Proposition 44]{Vickers2004}, which in our terminology says that, whenever $\vdash$ is $1$-reflexive and monotone,
\[\vdash\text{ is a cut relation}\qquad\Leftrightarrow\qquad\vdash\text{ is cut-transitive}.\]
Indeed, if $\vdash$ is also cut-transitive then we can apply \autoref{DiamondCut} with $\vDash\ =\ \vdash$ to show that $\vdash$ is a cut relation.  Conversely, if $\vdash$ is a cut-relation then it must be cut-transitive, by \autoref{DiagonalTransitive}.  Below in \autoref{DiamondCover}, we will instead take $\vDash$ to be a certain $1$-reflexive relation $\Vdash$ defined from $\vdash$, even when $\vdash$ itself is not $1$-reflexive.  This could thus be viewed as a non-$1$-reflexive extension of \cite[Proposition 44]{Vickers2004}, much like \cite[Theorem 3.5]{JungKegelmannMoshier1999} but without any need for connectives or interpolants (or rather with interpolants replaced by \eqref{SuperIdempotence} below).

In order to define $\Vdash$, let us first denote the \emph{bounded} $F\in\mathsf{F}S$ by
\begin{equation}\label{BSDefinition}
\tag{Bounded Subsets}\mathsf{B}S=\mathsf{F}S^\dashv=\{F\vdash G:G\in\mathsf{F}S\}.
\end{equation}
The restriction of any relation $\vDash\ \subseteq \mathsf{F}S\times T$ to $\mathsf{B}S\times T$ will be denoted by ${}_\mathsf{B}\hspace{-3pt}\vDash$, i.e.
\[F\ {}_\mathsf{B}\hspace{-3pt}\vDash t\qquad\Leftrightarrow\qquad\mathsf{B}S\ni F\vdash t.\]
Using bounded subsets, we can now define the relation $\Vdash$.  Actually, the use $\mathsf{B}S$ rather than $\mathsf{F}S$ below in the definition of $\Vdash$ is only relevant if $F=\emptyset$ -- otherwise $H\vdash\mathsf{1}F$ already implies that $H$ is bounded.

\begin{prp}\label{uCdensity}
For any relation $\vdash$ on $\mathsf{F}S$, the relation $\Vdash$ on $\mathsf{F}S$ defined by
\begin{equation}\label{DfromC}
F\Vdash G\qquad\Leftrightarrow\qquad\forall H\in\mathsf{B}S\ (H\vdash\mathsf{1}F\ \Rightarrow\ H\vdash G)
\end{equation}
is always lower, $1$-reflexive and satisfies ${}_\mathsf{B}(\vdash_{1\exists}\bullet\Vdash)\subseteq\ \vdash$.  Moreover,
\begin{enumerate}
\item If $\vdash$ is upper then so is $\Vdash$.
\item\label{Vsubv} If $\vdash$ is lower and $1$-reflexive then ${}_\mathsf{B}\hspace{-3pt}\Vdash\ \subseteq\ \vdash$.
\item\label{vsubV} If $\vdash$ is cut-transitive then $\vdash\ \subseteq\ \Vdash$.
\item If $\vdash$ is a strong idempotent then $\Vdash$ is a Scott relation satisfying
\begin{equation}\label{VdashDualAux}
_\mathsf{B}(\vdash\bullet\Vdash)=\ \vdash.
\end{equation}
\end{enumerate}
\end{prp}

\begin{proof}
If $E\supseteq F\Vdash G$ then, for all $H\in\mathsf{B}S$ with $H\vdash\mathsf{1}E$, certainly $H\vdash\mathsf{1}F$ and hence $H\vdash G$, showing that $E\Vdash G$, i.e. $\Vdash$ is a lower relation.  Also, $H\vdash\mathsf{1}\{s\}$ just means $H\vdash s$ and hence $s\Vdash s$, for all $s\in S$, i.e. $\Vdash$ is $1$-reflexive.  To see that $_\mathsf{B}(\vdash_{1\exists}\bullet\Vdash)\subseteq\ \vdash$, take $F\in\mathsf{B}S$ and $G\in\mathsf{F}S$ with $F\vdash_{1\exists}\bullet\Vdash G$, so we have $\mathcal{H}\in\mathsf{FF}S$ with $F\vdash_{1\exists}\mathcal{H}$ and $\mathcal{H}_\between\Vdash G$.  In particular, for every $H\in\mathcal{H}$, we have $h\in H$ with $F\vdash h$.  Thus we have $H\in\mathcal{H}_\between$ with $F\vdash\mathsf{1}H$ and $H\Vdash G$.  This implies $F\vdash G$, by the definition of $\Vdash$, showing that $_\mathsf{B}(\vdash_{1\exists}\bullet\Vdash)\subseteq\ \vdash$.

\begin{enumerate}
\item If $\vdash$ is upper and $F\Vdash G\subseteq H$ then, for all $E\in\mathsf{B}S$ with $E\vdash\mathsf{1}F$, $E\vdash G\subseteq H$ and hence $E\vdash H$, showing that $\Vdash$ is also upper.

\item If $\vdash$ is lower and $1$-reflexive and $\mathsf{B}S\ni F\Vdash G$ then $F\vdash\mathsf{1}F$ and hence $F\vdash G$, showing that ${}_\mathsf{B}\hspace{-3pt}\Vdash\ \subseteq\ \vdash$.

\item If $\vdash$ is cut-transitive and $F\vdash G$ then $H\vdash\mathsf{1}F$ implies
\[H\vdash\mathsf{1}F\bowtie\{F\}\vdash G\]
and hence $H\vdash G$.  This shows that $F\Vdash G$, which in turn shows that $\vdash\ \subseteq\ \Vdash$.

\item Say $\vdash$ is a strong idempotent.  To show that $\Vdash$ is a Scott relation, it suffices to show that $\Vdash$ is a cut relation, thanks to what we have already proved.  Accordingly, say $\{p\}\cup F\Vdash G$ and $F\Vdash G\cup\{p\}$.  Take $H\in\mathsf{B}S$ with $H\vdash\mathsf{1}F$, i.e. $H\vdash f$, for all $f\in F$.  As $\vdash\ \subseteq\ \vdash\bullet\vdash_{1\exists}$, for each $f\in F$, we have $\mathcal{I}_f\in\mathsf{FB}S$ with $H\vdash_\between\mathcal{I}_f\vdash f$.  Let $\mathcal{I}=\bigwedge_{f\in F}\mathcal{I}_f$ so $H\vdash\mathcal{I}_\between(=\bigcup_{f\in F}\mathcal{I}_{f\between})$ and $\mathcal{I}\vdash\mathsf{1}F$ and hence $\mathcal{I}\vdash G\cup\{p\}$, as $F\Vdash G\cup\{p\}$.  For each $I\in\mathcal{I}$, we then have $\mathcal{J}_I\in\mathsf{FB}S$ with $I\vdash_\between\mathcal{J}_I\vdash_{1\exists}G\cup\{p\}$.  Let $\mathcal{J}=\bigcup_{I\in\mathcal{I}}\mathcal{J}_I$ so $\mathcal{I}\vdash\mathcal{J}_\between(=\bigwedge_{I\in\mathcal{I}}\mathcal{J}_{I\between})$ and $\mathcal{J}\vdash_{1\exists}G\cup\{p\}$.  So for each $J\in\mathcal{J}$, either $J\vdash G$ or $J\vdash p$, in which case $I\cup J\vdash\mathsf{1}F\cup\{\{p\}\}$ and hence $I\cup J\vdash G$, for all $I\in\mathcal{I}$, as $\{p\}\cup F\Vdash G$.  Thus $\mathcal{I}\wedge\mathcal{J}\vdash G$.  As $H\vdash_\between\mathcal{I}\vdash\mathcal{J}_\between$, cut-transitivity yields $H\vdash\mathcal{J}_\between$ and hence $H\vdash\mathcal{I}_\between\cup\mathcal{J}_\between=(\mathcal{I}\wedge\mathcal{J})_\between$.  As $\mathcal{I}\wedge\mathcal{J}\vdash G$, cut-transitivity again yields $H\vdash G$, showing that $F\Vdash G$.  This shows that $\Vdash$ is a cut relation and hence a Scott relation.

As $\Vdash$ is $1$-reflexive, $\vdash\ \subseteq{}_\mathsf{B}(\vdash\bullet\Vdash)$.  Conversely, as $\vdash$ is a strong idempotent, $\vdash\bullet\Vdash\ \ \subseteq\ \ \vdash\bullet\vdash_{1\exists}\bullet\Vdash$.
Thus if $\mathsf{B}S\ni F\vdash\bullet\Vdash G$ then we have $\mathcal{H}\in\mathsf{FF}S$ with $F\vdash_\between\mathcal{H}\vdash_{1\exists}\bullet\Vdash G$.  If $\mathcal{H}\subseteq\mathsf{B}S$ then $F\vdash_\between\mathcal{H}\vdash G$, as ${}_\mathsf{B}(\vdash_{1\exists}\bullet\Vdash)\subseteq\ \vdash$, and hence $F\vdash G$, by cut-transitivity.  On the other hand, if $\mathcal{H}\nsubseteq\mathsf{B}S$, then the only way $\mathcal{H}\vdash_{1\exists}\bullet\Vdash G$ could hold is if $\mathcal{H}\vdash_{1\exists\forall}\emptyset\bowtie\mathcal{I}\ {}_\forall\hspace{-3pt}\Vdash G$, for some $\mathcal{I}\subseteq\mathsf{FF}S$.  But $\emptyset\bowtie\mathcal{I}$ implies $\emptyset\in\mathcal{I}$ and hence $\emptyset\Vdash G$, which means that $B\vdash G$, for all $B\in\mathsf{B}S$.  In particular, this again yields $F\vdash G$.  This shows that $_\mathsf{B}(\vdash\bullet\Vdash)\subseteq\ \vdash$, thus proving \eqref{VdashDualAux}.\qedhere
\end{enumerate}
\end{proof}

Cover relations are strong idempotents where the opposite of \eqref{VdashDualAux} also holds.

\begin{dfn}\label{CoverRelationDef}
A \emph{cover relation} is a strong idempotent $\vdash$ auxiliary to $\Vdash$.
\end{dfn}

In other words, a cover relation is a strong idempotent $\vdash$ satisfying $\Vdash\bullet\vdash\ \subseteq\ \vdash$, by \autoref{DiamondvsAuxiliarity} \eqref{DiamondToAux} and \eqref{AuxToDiamond}, as $\Vdash$ is always a $1$-reflexive lower relation.  This also means that $F\vdash G$ implies $F\Vdash\mathsf{1}F\bowtie\{F\}\vdash G$, so the reverse inclusion $\vdash\ \subseteq\ \Vdash\bullet\vdash$ is immediate.  Thus a relation $\vdash$ on $\mathsf{F}S$ is a cover relation precisely when
\[\tag{Cover Relation}\supseteq_{\mathsf{F}S}\circ\vdash\circ\subseteq_{\mathsf{F}S}\ \ =\ \ \vdash\ \ =\\  \vdash\bullet\vdash\ \ =\ \ \vdash\bullet\vdash_{1\exists}\ \ =\ \ \Vdash\bullet\vdash.\]

Note that, as long as $\vdash\ \subseteq\ \vdash\bullet\vdash$,
\[F\in\mathsf{B}S\qquad\Leftrightarrow\qquad\exists\mathcal{H}\in\mathsf{FB}S\ (F\vdash\mathcal{H}_\between).\]
Indeed, if $F\in\mathsf{B}S$ then we have $G\in\mathsf{F}S$ with $F\vdash G$ and $\vdash\ \subseteq\ \vdash\bullet\vdash$ then yields $\mathcal{H}\in\mathsf{FF}S$ with $F\vdash_\between\mathcal{H}\vdash G$ and hence $\mathcal{H}\subseteq\mathsf{B}S$.  Conversely, if $F\vdash\mathcal{H}_\between$, for some $\mathcal{H}\in\mathsf{FB}S$, then $F\in\mathsf{B}S$ -- if $\emptyset\notin\mathcal{H}$ and hence $\mathcal{H}_\between\neq\emptyset$ then this is immediate from the definition of $\mathsf{B}S$, while if $\emptyset\in\mathcal{H}\subseteq\mathsf{B}S$ then $\mathsf{F}S=\mathsf{B}S$ and, in particular, $F\in\mathsf{B}S$.

The following definition of boundedness for families of finite subsets thus agrees with the previous notion for single finite subsets in \eqref{BSDefinition}.

\begin{dfn}
We call $\mathcal{B}\subseteq\mathsf{F}S$ \emph{bounded} if $\mathcal{B}\vdash\mathcal{F}_\between$, for some $\mathcal{F}\in\mathsf{FB}S$.
\end{dfn}

\begin{prp}
If $\vdash$ is a cover relation then
\[\emptyset\in\mathsf{B}S\quad\Leftrightarrow\quad\mathsf{F}S=\mathsf{B}S\quad\Leftrightarrow\quad\mathsf{F}S\text{ is bounded}\quad\Leftrightarrow\quad\mathsf{B}S\text{ is bounded}.\]
\end{prp}

\begin{proof}
Note $\emptyset\vdash\mathcal{F}_\between$ is equivalent to $\mathsf{F}S\vdash\mathcal{F}_\between$, as $\vdash$ is a lower relation.  This immediately yields the first two equivalences.  We also immediately see that $\mathsf{B}S$ is bounded if $\mathsf{F}S$ is.  Now say that $\mathsf{B}S$ is bounded, so we have $\mathcal{F}\in\mathsf{FB}S$ with $B\vdash\mathcal{F}_\between$, for all $B\in\mathsf{B}S$.  This is exactly saying that $\emptyset\Vdash\mathcal{F}_\between$.  As $\mathcal{F}\in\mathsf{FB}S$, we have $G\in\mathsf{F}S$ with $\mathcal{F}\vdash G$.  Then $\emptyset\Vdash_\between\mathcal{F}\vdash G$ and hence $\emptyset\vdash G$, as $\Vdash\bullet\vdash\ \subseteq\ \vdash$.  But this means $\emptyset\in\mathsf{B}S$, thus proving the final equivalence.
\end{proof}

We will later see that these conditions are also equivalent to saying that the tight spectrum is compact.  The equivalence of these conditions also characterises the cover relations among Scott relations.  These can also be characterised by the agreement of $\vdash$ and $\Vdash$, in particular at $\emptyset$, the only potentially unbounded subset with respect to a Scott relation.

\begin{prp}
If $\vdash$ is a Scott relation then ${}_\mathsf{B}\hspace{-3pt}\Vdash\ =\ \vdash$ and, moreover,
\[\vdash\text{ is a cover relation}\quad\ \Leftrightarrow\quad\ \Vdash\ =\ \vdash\quad\ \Leftrightarrow\quad\ \mathsf{F}S\text{ is bounded or $\mathsf{B}S$ is unbounded}.\]
\end{prp}

\begin{proof}
Assume $\vdash$ is a Scott relation.  Then ${}_\mathsf{B}\hspace{-3pt}\Vdash\ =\ \vdash$ follows from \autoref{uCdensity} \eqref{Vsubv} and \eqref{vsubV}.  This and the fact $\vdash$ is $1$-reflexive yields $\vdash\ \subseteq\ \Vdash\ \subseteq\ \Vdash\bullet\vdash$.  If $\vdash$ is a cover relation then $\Vdash\bullet\vdash\ \subseteq\ \vdash$ too and hence $\Vdash\ =\ \vdash$.  Conversely, if $\Vdash\ =\ \vdash$ then $\vdash\ =\ \vdash\bullet\vdash\ =\ \Vdash\bullet\vdash$ and hence $\vdash$ is a cover relation.  This proves the first equivalence.

For the second, note that if $\mathsf{F}S$ is bounded then, in particular, $\mathsf{F}S=\mathsf{B}S$ so $\Vdash\ ={}_\mathsf{B}\hspace{-3pt}\Vdash\ =\ \vdash$.  On the other hand, if $\mathsf{B}S$ is unbounded then there can not be any $G\in\mathsf{F}S$ satisfying $\mathsf{B}S\vdash G$, because then $\mathsf{B}S\vdash(\mathsf{1}G)_\between$ and $\mathsf{1}G\subseteq\mathsf{1}S\subseteq\mathsf{F}S\setminus\{\emptyset\}\subseteq\mathsf{B}S$, by $1$-reflexivity.  But this is just saying that $\emptyset\not\Vdash G$, for all $G\in\mathsf{F}S$.  This yields $\Vdash\ \subseteq{}_\mathsf{B}\hspace{-3pt}\Vdash\ \subseteq\ \vdash\ \subseteq\ \Vdash$, again showing that $\Vdash\ =\ \vdash$.
\end{proof}

For example, if $\vdash$ is a Scott relation on a finite set $S$ then any bounded $F\in\mathsf{B}S$ must be bounded by $S$ itself, thanks to $\vdash$ being upper.  This is just saying that $\emptyset\Vdash S$.  Thus $\vdash$ is a cover relation precisely when $\emptyset\vdash S$ also holds -- if $\emptyset\vdash S$ then $\mathsf{F}S=\mathsf{B}S$ so $\Vdash\ =\ {}_\mathsf{B}\hspace{-3pt}\Vdash\ =\ \vdash$, while if $\emptyset\not\vdash S$ then $\Vdash\ \neq\ \vdash$ is immediate.

In general, however, cover relations need not be $1$-reflexive.  Indeed, they are precisely the right non-$1$-reflexive generalisation of Scott relations (at least those satisfying $\Vdash\ =\ \vdash$) which are needed to characterise subbases of stably locally compact spaces, as we will soon see in the next section.  Cover relations can also be characterised as certain kinds of entailments.

\begin{cor}\label{DiamondCover}
A relation $\vdash$ on $\mathsf{F}S$ is a cover relation if and only if $\vdash$ is a divisible entailment which is auxiliary to $\Vdash$, i.e. an entailment satisfying
\begin{equation}\label{SuperIdempotence}
\Vdash\bullet\vdash\ \ \subseteq\ \ \vdash\ \ \subseteq\ \ \vdash\bullet\vdash_{1\exists}.
\end{equation}
\end{cor}

\begin{proof}
If $\vdash$ is a cover relation then, by \autoref{uCdensity}, $\Vdash$ is a Scott relation and
\[\Vdash\bullet\vdash\ \ \subseteq\ \ \vdash\ \ \subseteq\ \ \vdash\bullet\vdash_{1\exists}\ \ \subseteq\ \ \Vdash\bullet\vdash_{1\exists}.\]
Thus $\vdash$ is a cut relation, by \autoref{DiamondCut}, and hence an entailment.

Conversely, if $\vdash$ is a divisible entailment auxiliary to $\Vdash$ then $\vdash$ is cut-transitive, by \autoref{DiagonalTransitive}, and hence a cover relation.
\end{proof}

In the next section we will show that all cover relations have a topological representation.  But before moving on, we again provide a number of natural examples.

\subsection{More Examples}\label{MoreExamples}
First we recall some notation and terminology from \cite{Goubault2013}.  Given a topological space $X$, we say $p\subseteq X$ is \emph{compactly contained} in $q\subseteq X$ if any open cover of $q$ has a finite subcover of $p$.  We denote this relation by $\Subset$ so, letting $\mathsf{O}X$ denote the open subsets of $X$, we can restate this as
\[p\Subset q\qquad\Leftrightarrow\qquad\forall S\subseteq\mathsf{O}X\ (q\subseteq\bigcup S\ \Rightarrow\ \exists F\in\mathsf{F}S\ (p\subseteq\bigcup F)).\]
In particular, $\Subset$ restricted to $\mathsf{O}X$ is just the usual way-below relation relative to $\subseteq$.

We call a topological space $X$ \emph{core compact} if
\[\tag{Core Compact}x\in p\in\mathsf{O}X\qquad\Rightarrow\qquad\exists q\in\mathsf{O}X\ (x\in q\Subset p).\]
Equivalently, $X$ is core compact iff $p=\bigcup\{q\in\mathsf{O}X:q\Subset p\}$, for all $p\in\mathsf{O}X$.  A usual, we call $X$ \emph{sober} if each irreducible closed set $C$ has a unique dense point $x$, i.e. $C=\mathrm{cl}\{x\}$.  If $X$ is sober and core compact then, for any $p,q\in\mathsf{O}X$, the Hofmann-Mislove theorem implies that
\[p\Subset q\qquad\Leftrightarrow\qquad\exists\text{ compact }k\subseteq X\ (p\subseteq k\subseteq q).\]
Among sober spaces, it follows that core compactness is equivalent to local compactness, in the sense that every neighbourhood contains a compact neighbourhood of the same point (see \cite[Theorem 8.3.10]{Goubault2013}).

Still following \cite{Goubault2013}, we call $X$ \emph{core coherent} if, for all $p,q,r\in\mathsf{O}X$,
\[\tag{Core Coherent}p\Subset q,r\qquad\Rightarrow\qquad p\Subset q\cap r.\]
A space is \emph{stably locally compact} if it is core compact, core coherent and sober.  Equivalently, by \cite[\S VI-6]{GierzHofmannKeimelLawsonMisloveScott2003}, $X$ is stably locally compact if it is core compact and \emph{ultrasober} meaning that any ultrafilter $U\subseteq\mathsf{P}X$ with a limit has a unique maximum limit $u\in X$, i.e. for all $x\in X$,
\[S_x\subseteq U\qquad\Leftrightarrow\qquad S_x\subseteq S_u,\]
where $S\subseteq\mathsf{O}X$ is any subbasis for $X$ and $S_x=\{p\in S:x\in p\}$.

We say $S\subseteq\mathsf{O}X$ is \emph{$\cap$-round} if
\[\tag{$\cap$-Round}\label{capRound}x\in p\in S\quad\Rightarrow\quad\exists F\in\mathsf{F} S \ (x\in\bigcap F\Subset p).\]
Equivalently, $S$ is $\cap$-round iff $p=\bigcup\{\bigcap F:F\in\mathsf{F} S \text{ and }\bigcap F\Subset p\}$, for all $p\in S$ (e.g. any subbasis of a core compact space is $\cap$-round).  If $X$ is core coherent then this is the same as saying that finite intersections are \emph{point-round} in that
\[\tag{Point-Round}\label{PointRound}x\in p\in S\quad\Rightarrow\quad\exists q\in S\ (x\in q\Subset p),\]
which again is equivalent to saying $p=\bigcup\{q\in S:q\Subset p\}$, for all $p\in S$.

\begin{xpl}[Topological Spaces]\label{SubsetCex}
Take $S\subseteq\mathsf{O}X$ and define $\vdash\ =\ \vdash_\Subset$ by
\begin{equation}\tag{$\Subset$-Cover}\label{SubsetC}
F\vdash G\qquad\Leftrightarrow\qquad\bigcap F\Subset\bigcup G.
\end{equation}

We claim that $\vdash$ is an entailment.  Indeed, $\vdash$ is certainly monotone.  To see that $\vdash$ is a cut relation, say $p\cap\bigcap F\Subset\bigcup G$ and $\bigcap F\Subset\bigcup G\cup p$.  Take any open cover $C$ of $\bigcup G$.  As $p\cap\bigcap F\Subset\bigcup G$, we have finite $D\subseteq C$ covering $p\cap\bigcap F$.  As $\bigcap F\Subset\bigcup G\cup p\subseteq\bigcup C\cup p$, we also have finite $E\subseteq C$ with $\bigcap F\subseteq\bigcup E\cup p$.  Now
\[\bigcap F=\bigcap F\cap(\bigcup E\cup p)\subseteq(p\cap\bigcap F)\cup\bigcup E\subseteq\bigcup D\cup\bigcup E,\]
so $D\cup E$ is a finite subcover of $\bigcap F$, showing that $\bigcap F\Subset\bigcup G$, i.e. $F\vdash G$.  This proves the claim, i.e. $\vdash$ is indeed an entailment.

If $S$ is $\cap$-round then we claim that $\vdash$ is even a strong idempotent.  To see that $\vdash\ \subseteq\ \vdash\bullet\vdash_{1\exists}$, say $\bigcap F\Subset\bigcup G$.  For each $x\in\bigcup G$, we have $g_x\in G$ with $x\in g_x$ and then \eqref{capRound} yields $F_x\in\mathsf{F} S $ with $x\in\bigcap F_x\Subset g_x$.  Then \eqref{capRound} again yields $G_x\in\mathsf{F} S $ with $x\in\bigcap G_x\Subset f$, whenever $f\in F_x$.  As $\bigcap F\Subset\bigcup G$, we have $x_1,\cdots,x_n$ with $\bigcap F\subseteq\bigcup_k\bigcap G_{x_k}$ and hence $\bigcap F\Subset\bigcup H$, for all $H\in\mathcal{F}_\between$, where $\mathcal{F}=\{F_{x_1},\cdots,F_{x_n}\}$.  Thus $F\vdash\mathcal{F}_\between$ and $\mathcal{F}\vdash_{1\exists}G$, showing that $\vdash\ \subseteq\ \vdash\bullet\vdash_{1\exists}$.

As long as $S$ also covers $X$, we claim that $\vdash$ is also auxiliary to $\Vdash$.  First we note
\begin{equation}\label{FsubG}
F\Vdash G\qquad\Rightarrow\qquad\bigcap F\subseteq\bigcup G.
\end{equation}
To see this, say $F\Vdash G$ and take $x\in\bigcap F$.  By \eqref{capRound}, we have $H\in\mathsf{F}S$ with $x\in\bigcap H\Subset f$, for all $f\in F$.  Even if $F=\emptyset$, we still have $s\in S$ with $x\in s$ so, enlarging $H$ if necessary, we may further assume that $\bigcap H\Subset s$ and hence $H\in\mathsf{B}S$.  Now $H\vdash\mathsf{1}F$ and hence $H\vdash G$, by the definition of $\Vdash$.  Thus $x\in\bigcap H\Subset\bigcup G$, by the definition of $\vdash$, showing that $\bigcap F\subseteq\bigcup G$.

Now if $F\Vdash\mathcal{F}\bowtie\mathcal{G}\vdash G$ then \eqref{FsubG} yields
\[\bigcap F\subseteq\bigcap_{H\in\mathcal{F}}\bigcup H\subseteq\bigcup_{H\in\mathcal{G}}\bigcap H\Subset\bigcup G\]
and hence $F\vdash G$.  So $\Vdash\bullet\vdash\ \subseteq\ \vdash$ and hence $\vdash$ is a cover relation.

Incidentally, note the converse of \eqref{FsubG} only holds if $X$ is core coherent. Indeed, if $X$ is core coherent and $\bigcap F\subseteq\bigcup G$ then $\mathsf{B}S\ni H\vdash\mathsf{1}F$ implies $\bigcap H\Subset\bigcap F\subseteq\bigcup G$ and hence $H\vdash G$, showing that $F\Vdash G$.
\end{xpl}

For the next two examples, the following result will be useful.  It says we can create cover relations from Scott relations and compatible idempotent relations.

\begin{prp}\label{ScottCover}
Given a relation $<$ on a set $S$ and another relation $\vDash$ on $\mathsf{F}S$, let $\vartriangleleft \ =\ \mathrel{{}_\forall(<_\exists)}$ and $\vdash\ =\ \vDash\circ\vartriangleleft $, i.e. for $F,G\in\mathsf{F}S$,
\begin{align*}
F\vartriangleleft G\qquad&\Leftrightarrow\qquad\forall f\in F\ \exists g\in G\ (f< g).\\
F\vdash G\qquad&\Leftrightarrow\qquad\exists H\in\mathsf{F}S\ (F\vDash H\vartriangleleft G).
\end{align*}
\begin{enumerate}
\item\label{ScottCover1} $\vdash$ is an entailment if $\vDash$ is an entailment and, for all $s\in S$ and $F,G\in\mathsf{F}S$,
\begin{equation}\label{1Aux}
F\vDash G\qquad\Rightarrow\qquad F\vDash(G\setminus s^>)\cup\{s\}.
\end{equation}
\item\label{ScottCover2} $\vdash$ is a strong idempotent if $\vDash$ is also $1$-reflexive and $<\ \subseteq\ \mathrel{{}_1\hspace{-3pt}\vdash}\circ\mathrel{\vartriangleleft _1}$, i.e.
\begin{equation}\label{CInterpolation}
q< r\qquad\Rightarrow\qquad\exists F\in\mathsf{F}S\ (q\vdash F< r).
\end{equation}
\item\label{ScottCover3} $\vdash$ is a cover relation if, moreover, for all $F,G\in\mathsf{F}S$,
\begin{equation}\label{Fsucc}
\forall f\in F_>\cap S^>\ (f\vDash G)\qquad\Rightarrow\qquad F\vDash G.
\end{equation}
\end{enumerate}
\end{prp}

\begin{proof}\ 
\begin{enumerate}
\item As $\vartriangleleft $ is an upper relation, so is $\vdash$.  If $\vDash$ is a lower relation then so is $\vdash$.  If $\vDash$ is an entailment satisfying \eqref{1Aux} then we claim that $\vdash$ is also a cut relation and hence an entailment.  Indeed, if $\{p\}\cup F\vdash G$ and $F\vdash G\cup\{p\}$ then we have $H,I\in\mathsf{F} S $ with $\{p\}\cup F\vDash H\vartriangleleft G$ and $F\vDash I\vartriangleleft G\cup\{p\}$.  Letting $J=I\setminus p^>\vartriangleleft G$, \eqref{1Aux} yields $F\vDash J\cup\{p\}$.  Then $\{p\}\cup F\vDash H\cup J$ and $F\vDash H\cup J\cup\{p\}$, as $\vDash$ is upper, and hence $F\vDash H\cup J\vartriangleleft G$, as $\vDash$ is a cut relation.  This shows that $F\vdash G$, as required.

Further note $\vdash\ \subseteq\ \vDash$.  Indeed, $F\vdash G$ means $F\vDash H\vartriangleleft G$, for some $H\in\mathsf{F}S$, and then repeated applications of \eqref{1Aux} yield $F\vDash(H\setminus G^>)\cup G=G$.

\item If $\vDash$ is also $1$-reflexive then $s< t$ implies $s\vdash t$, i.e. $<\ \subseteq{}_1\hspace{-3pt}\vdash_1$.  Now take $F,G\in\mathsf{F}S$ with $F\vdash G$ so $F\vDash H\vartriangleleft G$, for some $H\in\mathsf{F}S$.  If \eqref{CInterpolation} also holds then, for each $h\in H$, we have $I_h\in\mathsf{F}S$ with $h\vdash I_h\vartriangleleft  G$, i.e. $h\vDash J_h\vartriangleleft  I_h$, for some $J_h\in\mathsf{F} S $.  Letting $I=\bigcup_{h\in H}I_h$ and $J=\bigcup_{h\in H}J_h$, we see that $F\vDash H\bowtie\mathsf{1}H\vDash J$ and hence $F\vDash J\vartriangleleft I\vartriangleleft G$.  Thus $F\vdash I\bowtie\mathsf{1}I\vdash_{1\exists}G$, as $<\ \subseteq{}_1\hspace{-3pt}\vdash_1$.  This shows that $\vdash\ \subseteq\ \vdash\bullet\vdash_{1\exists}$ so $\vdash$ is a strong idempotent.

\item If \eqref{Fsucc} also holds, we claim $\Vdash\ \subseteq\ \vDash$.  Indeed, if $F\Vdash G$ then $f\in F_>\cap S^>$ implies $f\vdash\mathsf{1}F$ and $f\in\mathsf{B}S$, as $<\ \subseteq{}_1\hspace{-3pt}\vdash_1$, and hence $f\vdash G$, by the definition of $\Vdash$.  We already noted above that $\vdash\ \subseteq\ \vDash$ so this implies $F\vDash G$, by \eqref{Fsucc}.

Now if $F\Vdash\mathcal{F}\bowtie\mathcal{G}\vdash G$ then $F\vDash\mathcal{F}\bowtie\mathcal{G}\vDash H\vartriangleleft G$, where $H$ is the union of the $H$'s witnessing $\mathcal{G}\vdash G$, so $F\vDash H\vartriangleleft G$ and hence $F\vdash G$.  This shows that $\vdash$ is auxiliary to $\Vdash$ and hence $\vdash$ is a cover relation. \qedhere
\end{enumerate}
\end{proof}

The next example extends \autoref{OrderedSets}, which is again from \cite{BiceStarling2020HTight} (in this case, $\vdash$ corresponds to the `centred cover relation' $\widehat{\mathsf{C}}$ in \cite{BiceStarling2020HTight}).

\begin{xpl}[Ordered Sets]\label{OrderedSetsXpl}
Say we have a transitive relation $<$ on a rounded set $S$ (i.e. $S\subseteq S^<$) and again define a Scott relation $\vDash$ on $\mathsf{F}S$ as in \eqref{wedgeperpDensity}, i.e.
\[F\vDash G\qquad\Leftrightarrow\qquad F_>\cap G_\perp=\emptyset.\]
If $F\vDash G$ then $((G\setminus s^>)\cup\{s\})_\perp\subseteq G_\perp$ and hence $F\vDash(G\setminus s^>)\cup\{s\}$, for any $s\in S$, i.e. \eqref{1Aux} holds.  Thus $\vdash\ =\ \vDash\circ\vartriangleleft $ is an entailment, by \eqref{ScottCover1} above.

If \eqref{CInterpolation} also holds, for example if $<\ \subseteq\ <\circ<$, then $\vdash$ is even a cover relation.  Indeed, this is follows from \eqref{ScottCover3} above because \eqref{Fsucc} holds.  To see this, just note that if $F\not\vDash G$ then we have $s\in F_>\cap G_\perp$.  Roundedness then yields $e,f\in S$ with $e<f\in s^>\subseteq F_>\cap S^>$ so $e\in f^>\subseteq s^>\subseteq G_\perp$ and hence $f\not\vDash G$.
\end{xpl}

In a similar manner, we can also extend \autoref{veeSemilattices}.

\begin{xpl}[$\vee$-Idempotents]\label{PredomainsXpl}
Say we have an idempotent relation $<$ on a set $S$, i.e. $<\ =\ <\circ<$, and that $S$ has a $<$-minimum $0$, i.e. $S=0^<$.  In particular, $<$ is transitive so we have a weaker preorder $\leq\ \supseteq\ <$ defined by
\[p\leq q\qquad\Leftrightarrow\qquad p^>\subseteq q^>.\]
Further assume that $S$ is a $\vee$-semilattice with respect to $\leq$ and again define a Scott relation $\vDash$ on $\mathsf{F}S$ as in \eqref{veeDensity}, i.e. for all $F,G\in\mathsf{F}S$,
\[F\vDash G\qquad\Leftrightarrow\qquad\forall p,q\in S\ \Big(\forall f\in F\ (p\leq f\vee q)\ \Rightarrow\ p\leq\bigvee G\vee q\Big).\]
If $F\vDash G$ then $\bigvee G\leq\bigvee(G\setminus s^\geq)\vee s\leq\bigvee(G\setminus s^>)\vee s$
 and hence $F\vDash(G\setminus s^>)\cup\{s\}$, for any $s\in S$, i.e. \eqref{1Aux} holds.  As $<$ is idempotent, \eqref{CInterpolation} also holds and hence $\vdash\ =\ \vDash\circ\vartriangleleft$ is a strong idempotent, by \autoref{ScottCover} \eqref{ScottCover2}.
 
 If $S$ is also distributive (as a $\vee$-semilattice with respect to $\leq$) then $F\vDash G$ is equivalent to $F_\geq\leq\bigvee G$, as noted in \eqref{vDistributive}, which is equivalent to $(F_\geq)^>\subseteq(\bigvee G)^>$, by the definition of $\leq$.  If $f\vDash G$ and hence $f^>=f^{\geq>}\subseteq(\bigvee G)^>$, for all $f\in F_>\cap S^>$, then $<\ \subseteq\ <\circ<$ yields $(F_\geq)^>=(F_>\cap S^>)^>\subseteq(\bigvee G)^>$ and hence $F\vDash G$.  Thus \eqref{Fsucc} also holds so $\vdash$ is a cover relation, by \autoref{ScottCover} \eqref{ScottCover3}.
\end{xpl}

In particular, the above example encompasses the compingent lattices of Shirota and de Vries (see \cite{Shirota1952}, \cite{deVries1962} and \cite{BaayendeRijk1963}) as will as the (strong) proximity lattices considered by Smyth, Jung-S\"underhauf and van Gool (see \cite{Smyth1992}, \cite{JungSunderhauf1995} and \cite{vanGool2012}).  Indeed, the distributive semilattices arising from idempotent relations above are already considerably more general.  In \cite[\S2]{vanGool2012}, van Gool asked whether his duality could be extended even to non-distributive proximity lattices.  \autoref{veeSemilattices} shows that this is perfectly feasible and is likely just a matter of finding the right modification of \eqref{veeDensity}, an approach already hinted at in the conclusion of \cite{vanGool2012}.  We will address this question in detail in future work.

\section{Topological Representation}\label{TopologicalRepresentation}

Our goal now is to represent abstract cover relations on topological spaces as in \autoref{SubsetCex}.  The points of the space will be `tight subsets', which generalise the prime filters used in the classic Stone duality.  These were originally considered in \cite{BiceStarling2020HTight} (see \autoref{OrderedSetsXpl}) based on earlier work of Exel in \cite{Exel2008} (they also correspond to Kawai's `models', at least in the context of certain geometric theories associated to his strong continuous entailment relations -- see \cite[\S2 and (5.6)]{Kawai2020}).

\subsection{Tight Subsets}

\begin{dfn}
Given a relation $\vdash$ on $\mathsf{F}S$, we call $T\subseteq S$ \emph{tight} if, for all $G\in\mathsf{F} S $,
\[\tag{Tight}\label{Tight}\exists F\in\mathsf{F}S\ (T\supseteq F\vdash G)\qquad\Leftrightarrow\qquad T\cap G\neq\emptyset.\]
\end{dfn}

As long as $\vdash$ is an upper relation, this naturally splits into two conditions, namely
\begin{align*}
\tag{Prime}\label{Prime}T\supseteq F\vdash G\qquad&\Rightarrow\qquad T\cap G\neq\emptyset.\\
\tag{Round}\label{Round}\exists F\in\mathsf{F}S\ (T\supseteq F\vdash t)\qquad&\Leftarrow\qquad t\in T.
\end{align*}
In particular, \eqref{Prime} implies $F\not\vdash\emptyset$, for all $F\in\mathsf{F}T$.  We can also restate \eqref{Prime} by extending $\vdash$ to infinite subsets.  Specifically, given $\vdash$ on $\mathsf{F}S$, define $\vdash_\mathsf{P}$ on $\mathsf{P}S$ by
\[Q\vdash_\mathsf{P}R\qquad\Leftrightarrow\qquad\exists F\in\mathsf{F}Q\ \exists G\in\mathsf{F}R\ (F\vdash G).\]
For any $T\subseteq S$, we then see that
\[T\text{ is prime}\qquad\Leftrightarrow\qquad T\not\vdash_\mathsf{P}S\setminus T.\]
In particular, the empty set is tight iff $\emptyset\not\vdash_\mathsf{P}S$.  However, even in this case, we will only consider non-empty tight subsets as points of the spectrum (see \autoref{TheSpectrum} below).

\begin{xpl}
Say $S$ is a $\vee$-semilattice with minimum $0$ and $\vdash$ defined as in \eqref{veeDensity}.  As $\vdash$ is $1$-reflexive, \eqref{Round} trivially holds for any $T\subseteq S$.

If $S$ is a distributive lattice then, $\vdash$ can be characterised as in \eqref{MeetJoinCover}, i.e.
\[F\vdash G\qquad\Leftrightarrow\qquad\bigwedge F\leq\bigvee G\]
(again with $\bigwedge\emptyset=1$ if $S$ has a maximum $1$, otherwise $\emptyset\nvdash G$, for all $G\in\mathsf{F}S$).  In this case, take $T\subseteq S$ satisfying \eqref{Prime}.  Then $T$ is closed under finite meets as $F\vdash\bigwedge F$, for any finite $F\subseteq T$, i.e. $T$ is down-directed.  Also, $T$ is an up-set and hence a filter, as $t\vdash s$ whenever $s\geq t\in T$.  Moreover, $S\setminus T$ is closed under finite joins as $\bigvee G\vdash G$, for any finite $G\subseteq S\setminus T$.  In particular, $\bigvee\emptyset=0\vdash\emptyset$ so $0\notin T$.  Thus $T$ is a prime filter.  Conversely, if $T\subseteq S$ is a prime filter and $T\supseteq F\vdash G$ then $T\ni\bigwedge F\leq\bigvee G$ and hence $T\cap G\neq\emptyset$, showing that $T$ satisfies \eqref{Prime} and hence \eqref{Tight}.  So in this case
\[T\text{ is tight}\qquad\Leftrightarrow\qquad T\text{ is a prime filter}.\]
These prime filters are precisely those considered in the classic Stone duality.

If $S$ is a distributive $\vee$-semilattice then $\vdash$ can be characterised as in \eqref{vDistributive}, i.e.
\[F\vdash G\qquad\Leftrightarrow\qquad F_\geq\leq\bigvee G.\]
Again take $T\subseteq S$ satisfying \eqref{Prime}.  In particular, $t\in T$ whenever $t\geq F^\geq$ and $F\in\mathsf{F}T$, which is saying that $T$ is a Frink filter (i.e. the dual of a Frink ideal).  Again we see that $S\setminus T$ is closed under finite joins and is thus a `weakly prime ideal' (see \cite[Definition 1.6]{HansoulPoussart2008}).  Conversely, if $S\setminus T$ is a weakly prime ideal and $T\supseteq F\vdash G$ then $F^\geq\leq\bigvee G$ so $\bigvee G\notin T$ and hence $T\cap G\neq\emptyset$, showing that $T$ satisfies \eqref{Prime} and hence \eqref{Tight}.  So in this case
\[T\text{ is tight}\qquad\Leftrightarrow\qquad S\setminus T\text{ is a weakly prime ideal}.\]
These weakly prime ideals are precisely those used in the duality of Hansoul-Poussart and are dual to the optimal filters in the duality of Bezhanishvili-Jansana (see \cite[Remark 4.6]{BezhanishviliJansana2011}).
\end{xpl}

As in \autoref{PredomainsXpl}, say we have an idempotent relation $<$ on a set $S$ with $<$-minimum $0$ and define a preorder $\leq$ satisfying $<\circ\leq\ =\ <\ \subseteq\ \leq$ by
\[p\leq q\qquad\Leftrightarrow\qquad p^>\subseteq q^>.\]
We call $(S,<)$ a \emph{proximity lattice} if $S$ is a lattice w.r.t. $\leq$, the dual condition $\leq\circ<\ =\ <$ also holds and $<$ also respects the lattice structure, i.e.
\[\tag{Proximity}p<q\quad\text{and}\quad r<s\qquad\Rightarrow\qquad p\wedge r<q\wedge s\quad\text{and}\quad p\vee r<q\vee s.\]
These are like the proximity lattices considered in \cite{Smyth1992} except that here $<$ automatically `approximates' $\leq$, as we are taking $\leq$ to be defined from $<$ as above.  The key difference with the proximity lattices in \cite{JungSunderhauf1995} and \cite{vanGool2012} is that we do not require $S$ to have a maximum.

\begin{xpl}
Take a proximity lattice $(S,<)$ and define $\vdash$ on $\mathsf{F}S$ by
\[F\vdash G\qquad\Leftrightarrow\qquad\exists H\in\mathsf{F}(G^>)\ \Big(\bigwedge F\leq\bigvee H\Big)\]
(as long as $(S,\leq)$ is distributive, this agrees with the cover relation $\vdash\ =\ \vDash\circ\vartriangleleft$ in \autoref{PredomainsXpl}).  Now take $T\subseteq S$ satisfying \eqref{Tight}.  For any finite $G\subseteq T$, \eqref{Round} yields $F\in\mathsf{F}T$ with $F\vdash\mathsf{1}G$, i.e. for each $g\in G$ we have $H_g\in\mathsf{F}S$ with $F\vdash H_g<g$ and hence $\bigwedge F\leq\bigvee H_g<g$.  Setting $I=\{\bigvee H_g:g\in G\}\in\mathsf{F}(G^>)$, it follows that $\bigwedge F\leq\bigwedge I<\bigwedge G$.  Thus $\bigwedge F\vdash\bigwedge G$ and hence $\bigwedge G\in T$, by \eqref{Prime}, i.e. $T$ is down-directed.  Now \eqref{Round} again implies that $T$ is rounded, i.e. $T\subseteq T^<$.  Also \eqref{Prime} now reduces to
\[\tag{Proximal}\label{Proximal}T\ni t\leq\bigvee F\text{ and }F\subseteq G^>\qquad\Rightarrow\qquad T\cap G\neq\emptyset.\]
Conversely, if $T$ is rounded then certainly $T$ satisfies \eqref{Round}.  If $T$ is also down-directed and satisfies \eqref{Proximal} then $T\supseteq F\vdash G$ implies that we have $H\in\mathsf{F}S$ with $T\ni\bigwedge F\leq\bigvee H$ and $H\subseteq G^>$ and hence $T\cap G\neq\emptyset$, showing that $T$ also satisfies \eqref{Prime} and is thus tight.  This shows that
\[T\text{ is tight}\qquad\Leftrightarrow\qquad T\text{ is a rounded proximal filter}.\]
These rounded proximal filters are precisely those from \cite[Definition 5]{Smyth1992}.

When $S$ is also \emph{$\vee$-interpolative} meaning that, for all $p,q,r\in S$,
\[\tag{$\vee$-Interpolation}\label{veeInterpolation}p<q\vee r\qquad\Rightarrow\qquad\exists s<q\ \exists t<r\ (p<q\vee r)\]
(called simply `interpolation' in \cite{Smyth1986} and `join-strong' in \cite{JungSunderhauf1995} and \cite{vanGool2012}), \eqref{Proximal} above for rounded $T$ reduces to $T\cap F\neq\emptyset$ whenever $T\ni\bigvee F$, i.e. then 
\[T\text{ is tight}\qquad\Leftrightarrow\qquad T\text{ is a rounded prime filter}.\]
These rounded prime filters are precisely those considered in \cite{JungSunderhauf1995} and \cite{vanGool2012}.
\end{xpl}

Note $\vdash_\mathsf{P}$ above agrees with $\vdash$ on $\mathsf{F}S$ precisely when $\vdash$ is monotone.  In this case, we often drop the subscript in single instances of this relation, i.e. for any $Q,R\subseteq S$,
\[Q\vdash R\qquad\text{means}\qquad Q\vdash_\mathsf{P}R.\]

The following result is a generalisation of the Birkhoff-Stone prime filter theorem (further generalisating \cite[Lemma 4.7]{BezhanishviliJansana2011} and \cite[Theorem 3.2]{vanGool2012}).

\begin{thm}\label{BirkhoffStone}
For any $Q\subseteq S$, strong idempotent $\vdash$ on $\mathsf{F}S$ and round $R\subseteq S$,
\[R\not\vdash Q\qquad\Rightarrow\qquad\exists\text{ tight }T\supseteq R\ (T\cap Q=\emptyset).\]
\end{thm}

\begin{proof}
If $R$ is round and $R\not\vdash Q$ then the Kuratowski-Zorn lemma yields a maximal extension $P\supseteq Q$ with $R\not\vdash P$.  If we had $p\in P\cap R$ then the roundness of $R$ would yield $F\in\mathsf{F} R $ with $F\vdash p$, contradicting the definition of $P$.  Thus $P\cap R=\emptyset$ and we may take $T=S\setminus P\supseteq R$.

If $T$ were not prime, this would mean $F\vdash G$, for some finite $F\subseteq T$ and $G\subseteq P$.  The maximality of $P$ would then yield finite $H\subseteq R$ and $I\subseteq P$ with $H\vdash I\cup\{f\}$, for all $f\in F$.  Then $F\vdash G\cup I$ and $H\vdash G\cup I\cup\{f\}$, for all $f\in F$, by monotonicity.  As $\vdash$ is a strong idempotent and hence a semicut relation (see \autoref{DiamondvsAuxiliarity2}), then $R\supseteq H\vdash I\cup G\subseteq P$, again contradicting the definition of $P$.  Thus $T$ is prime.

To see that $T$ is also round and hence tight, take any $t\in T$.  The maximality of $P$ yields finite $F\subseteq R$ and $G\subseteq P$ with $F\vdash G\cup\{t\}$.  As $\vdash\ \subseteq\ \vdash\bullet\vdash_{1\exists}$, we have $\mathcal{H},\mathcal{I}\in\mathsf{F}\mathsf{F}S$ with $\mathcal{H}\vdash G$, $\mathcal{I}\vdash t$ and $F\vdash(\mathcal{H}\cup\mathcal{I})_\between$.  This implies $H\nsubseteq T$, for all $H\in\mathcal{H}$ -- otherwise $T\supseteq H\vdash G\subseteq P$, contradicting the definition of $P$.  If we also had $I\nsubseteq T$, for all $I\in\mathcal{I}$, then we would have $J\in(\mathcal{H}\cup\mathcal{I})_\between$ with $T\supseteq F\vdash J\subseteq P$, contradicting the fact that $T$ is prime.  Thus we have $I\in\mathcal{I}$ with $T\supseteq I\vdash t$.
\end{proof}

The following generalisation will also be useful.

\begin{cor}\label{BirkhoffStone2}
For any $\mathcal{Q}\subseteq\mathsf{F}S$, strong idempotent $\vdash$ on $\mathsf{F}S$ and round $R\subseteq S$,
\[\forall F\in\mathsf{F}R\ \forall\mathcal{G}\in\mathsf{F}\mathcal{Q}\ \exists G\in\mathcal{G}_\between\ (F\not\vdash G)\qquad\Rightarrow\qquad\exists\text{ tight }T\supseteq R\ \forall H\in\mathcal{Q}\  (H\nsubseteq T).\]
\end{cor}

\begin{proof}
By the Kuratowski-Zorn lemma, we have a maximal extension $\mathcal{P}\subseteq\mathsf{F}S$ of $\mathcal{Q}$ such that, for all $F\in\mathsf{F}R$ and $\mathcal{G}\in\mathsf{F}\mathcal{P}$, we have $G\in\mathcal{G}_\between$ with $F\nvdash G$.  We claim that, for all $G\in\mathcal{P}$, we have $g\in G$ with $\{g\}\in\mathcal{P}$.  If not then, for every $g\in G$, maximality would yield $F_g\in\mathsf{F}S$ and $\mathcal{H}_g\in\mathsf{F}\mathcal{P}$ with $F_g\vdash(\mathcal{H}_g\cup\{\{g\}\})_\between=\mathcal{H}_{g\between}\wedge\{\{g\}\}$.  Letting $F=\bigcup_{g\in G}F_g\in\mathsf{F}R$ and $\mathcal{H}=\{G\}\cup\bigcup_{g\in G}\mathcal{H}_g\in\mathsf{F}\mathcal{P}$, it would follow that $F\vdash\mathcal{H}_\between$, contradicting the defining property of $\mathcal{P}$.  This proves the claim.

Letting $Q=\{q\in S:\{q\}\in\mathcal{P}\}$, it then follows that $Q\cap H\neq\emptyset$, for all $H\in\mathcal{Q}$, and $F\nvdash G$, for all $F\in\mathsf{F}R$ and $G\in\mathsf{F}Q$ (because $\mathsf{1}G\in\mathsf{F}\mathcal{P}$ and so $F\nvdash H$, for some $H\in(\mathsf{1}G)_\between=\{G\}^\subseteq$, and hence $F\nvdash G$, as $\vdash$ is upper), i.e. $R\nvdash Q$.  \autoref{BirkhoffStone} then yields tight $T\supseteq R$ with $T\cap Q=\emptyset$ and hence $H\nsubseteq T$, for all $H\in\mathcal{Q}$.
\end{proof}

We can also shrink any prime subset slightly to make it tight.

\begin{prp}\label{Prime->Tight}
If $\vdash$ is a strong idempotent on $\mathsf{F}S$ and $P\subseteq S$ is prime then
\[\mathsf{F}P^{\vdash_1}=\{p\in S:\exists F\in\mathsf{F} P \ (F\vdash p)\}\text{ is tight}.\]
\end{prp}

\begin{proof}
To see that $\mathsf{F}P^{\vdash_1}$ is round, say $\mathsf{F}P\ni F\vdash p$.  As $\vdash\ \subseteq\ \vdash\bullet\vdash_{1\exists}$, we have $\mathcal{H}\in\mathsf{F}\mathsf{F}S$ with $F\vdash\mathcal{H}$ and $\mathcal{H}_\between\vdash p$.  Again using $\vdash\ \subseteq\ \vdash\bullet\vdash_{1\exists}$, for all $H\in\mathcal{H}$, we have $\mathcal{G}_H\in\mathsf{FF}S$ with $F\vdash\mathcal{G}_{H\between}$ and $\mathcal{G}_H\vdash_{1\exists}H$.  As $P$ is prime, $P\cap G\neq\emptyset$, for all $H\in\mathcal{H}$ and $G\in\mathcal{G}_{H\between}$.  Thus we must have $G\subseteq P$, for some $G\in\bigwedge_{H\in\mathcal{H}}\mathcal{G}_H$, and hence $G\vdash\mathsf{1}H$, for some $H\in\mathcal{H}_\between$.  Thus $\mathsf{F}P^{\vdash_1}\supseteq H\vdash p$, showing that $\mathsf{F}P^{\vdash_1}$ is round.

To show $\mathsf{F}P^{\vdash_1}$ is prime, say $F,G\in\mathsf{F} S $ and $\mathsf{F}P^{\vdash_1}\supseteq F\vdash G$.  Again $\vdash\ \subseteq\ \vdash\bullet\vdash_{1\exists}$ yields $\mathcal{H}\in\mathsf{F}\mathsf{F}S$ with $F\vdash\mathcal{H}_\between$ and $\mathcal{H}\vdash_{1\exists}G$.  As $F\subseteq\mathsf{F}P^{\vdash_1}\subseteq P$, it again follows that $P\cap H\neq\emptyset$, for all $H\in\mathcal{H}_\between$, and hence $H\subseteq P$, for some $H\in\mathcal{H}$.  Then we have $g\in G$ with $P\supseteq H\vdash g$ so $g\in\mathsf{F}P^{\vdash_1}$, showing $\mathsf{F}P^{\vdash_1}$ is prime and hence tight.
\end{proof}

\subsection{The Spectrum}\label{TheSpectrum}

Say we are given a relation $\vdash$ on $\mathsf{F} S $.  For any $Q\subseteq S$, let
\begin{align*}
\mathsf{T}_Q&=\{T\subseteq S:T\text{ is tight and }Q\subseteq T\}.\\
\mathsf{T}^Q&=\{T\subseteq S:T\text{ is tight and }Q\cap T\neq\emptyset\}.
\end{align*}
The \emph{(tight) spectrum} is the space $\mathsf{T}^S$ with subbasis $(\mathsf{T}_p)_{p\in S}$, i.e. the space of all non-empty tight subsets in the topology generated by $(\mathsf{T}_p)_{p\in S}$.

\begin{thm}\label{TopRep}
If $\vdash$ is a strong idempotent then $\mathsf{T}^S$ is stably locally compact and
\begin{align}
\label{FDGT}F\Vdash G\qquad&\Leftrightarrow\qquad\mathsf{T}_F\subseteq\mathsf{T}^G.\\
\label{FCGT}F\vdash G\qquad&\Rightarrow\qquad\mathsf{T}_F\Subset\mathsf{T}^G.
\end{align}
If $\vdash$ is a cover relation then the converse $\Leftarrow$ also holds in \eqref{FCGT} above.
\end{thm}

\begin{proof}
First we show $\mathsf{T}^S$ is ultrasober.  For any ultrafilter $\mathscr{U}\subseteq\mathsf{P}(\mathsf{T}^S)$, we claim
\[P=\bigcup_{\mathfrak{U}\in\mathscr{U}}\bigcap\mathfrak{U}\]
is prime.  To see this, say $\mathsf{F} P \ni F\vdash G$.  For each $f\in F$, we have $\mathfrak{U}_f\in\mathscr{U}$ with $f\in\bigcap\mathfrak{U}_f$, i.e. $f\in T$, for all $T\in\mathfrak{U}_f$.  As $\mathscr{U}$ is a filter, $\mathfrak{U}=\bigcap_{f\in F}\mathfrak{U}_f\in\mathscr{U}$ and $F\subseteq\bigcap\mathfrak{U}$, i.e. $F\subseteq T$, for all $T\in Y$.  As each $T\in\mathfrak{U}$ is prime, it follows that $T\cap G\neq\emptyset$ and hence $\mathfrak{U}=\bigcup_{g\in G}\mathfrak{U}_g$, where $\mathfrak{U}_g=\mathfrak{U}\cap\mathsf{T}_g=\{T\in\mathfrak{U}:g\in T\}$.  As $\mathscr{U}$ is an ultrafilter, we have $g\in G$ with $\mathfrak{U}_g\in\mathscr{U}$ and hence $g\in\bigcap\mathfrak{U}_g\subseteq P$, showing that $P$ is prime.  Thus $\mathsf{F}P^{\vdash_1}=\{p\in S:\exists F\in\mathsf{F} P \ (F\vdash p)\}$ is tight, by \autoref{Prime->Tight}.

Note that $T\in\mathsf{T}^S$ is a limit of $\mathscr{U}$ iff, for all $t\in T$, we have $\mathfrak{U}\in\mathscr{U}$ such that $\mathfrak{U}\subseteq\mathsf{T}_t$, i.e. $t\in\bigcap\mathfrak{U}$.  In other words, $\mathscr{U}\rightarrow T$ iff $T\subseteq\bigcup_{\mathfrak{U}\in\mathscr{U}}\bigcap\mathfrak{U}=P$.  But as each $T\in\mathsf{T}^S$ is round, this is equivalent to $T\subseteq \mathsf{F}P^{\vdash_1}$.  So if $\mathscr{U}$ has at least one limit in $\mathsf{T}^S$ then $\mathsf{F}P^{\vdash_1}\neq\emptyset$ is the unique maximum limit, showing that $\mathsf{T}^S$ is ultrasober.

Next we claim that, for all $F\in\mathsf{B}S$, $\mathcal{G}\in\mathsf{F}\mathsf{F}S$ and $Q\subseteq S$,
\begin{equation}\label{KeyLemma}
F\vdash\mathcal{G}\quad\text{and}\quad F\not\vdash Q\qquad\Rightarrow\qquad\bigcap_{G\in\mathcal{G}}\mathsf{T}^G\nsubseteq\mathsf{T}^Q.
\end{equation}
As $F\in\mathsf{B}S$, we can add a subset to $\mathcal{G}$ if necessary to ensure that $\mathcal{G}\neq\emptyset$ and hence $\emptyset\notin\mathcal{H}_0=\{H\in\mathcal{G}_\between:H\subseteq\bigcup\mathcal{G}\}$.  By \eqref{BetweenInterpolation}, we can then extend this to a sequence $(\mathcal{H}_n)\subseteq\mathsf{F}\mathsf{F}S$ such that, for all $n\geq0$, $F\vdash_\between\mathcal{H}_{n+1}\vdash_*\mathcal{H}_n$.  For each $n\geq0$, note $\mathcal{I}_n=\{H\in\mathcal{H}_n:H\not\vdash Q\}\neq\emptyset$ \textendash\, otherwise $F\vdash_\between\mathcal{H}_n\vdash Q$ and hence $F\vdash Q$, as $\vdash$ is cut-transitive, a contradiction.  Also $\mathcal{I}_{n+1}\ni I\vdash_{1\forall}H\in\mathcal{H}_n$ implies $H\in\mathcal{I}_n$ -- otherwise $I\vdash\mathsf{1}H\bowtie\{H\}\vdash Q$ and hence $I\vdash Q$, again by cut-transitivity, contradicting $I\in\mathcal{I}_{n+1}$.  Thus we can view $(\mathcal{I}_n)$ as finite levels of an infinite tree.  K\"onig's lemma then yields an infinite branch, i.e. we have $I_n\in\mathcal{I}_n$ with $I_{n+1}\vdash_{1\forall}I_n$, for all $n\geq0$.  Then we get round $R=\bigcup_{n\geq0}I_n\neq\emptyset$, as $\emptyset\notin\mathcal{H}_0\supseteq\mathcal{I}_0$, with $R\not\vdash Q$.  \autoref{BirkhoffStone} then yields tight $T\supseteq R$ with $T\cap Q=\emptyset$, which means that $T\in\mathsf{T}_{I_0}\setminus\mathsf{T}^Q\subseteq\bigcap_{G\in\mathcal{G}}\mathsf{T}^G\setminus\mathsf{T}^Q$ and hence $\bigcap_{G\in\mathcal{G}}\mathsf{T}^G\nsubseteq\mathsf{T}^Q$, proving the claim.

We further claim that this yields the following strengthening of \eqref{FCGT}:
\begin{equation}\label{FCG}
\mathsf{B}S\ni F\vdash\mathcal{G}\qquad\Rightarrow\qquad\mathsf{T}_F\Subset\bigcap_{G\in\mathcal{G}}\mathsf{T}^G.
\end{equation}
To see this, say $\mathsf{B}S\ni F\vdash\mathcal{G}$.  To show $\mathsf{T}_F\Subset\bigcap_{G\in\mathcal{G}}\mathsf{T}^G$, it suffices to show that every subbasic cover of $\bigcap_{G\in\mathcal{G}}\mathsf{T}^G$ has a finite subcover of $\mathsf{T}_F$ (see \cite[Exercise 5.2.11]{Goubault2013}).  Put another way, it suffices to show that if a subbasic family has no finite subset covering $\mathsf{T}_F$ then the whole family can not cover $\bigcap_{G\in\mathcal{G}}\mathsf{T}^G$.  Accordingly, take a subbasic family $(\mathsf{T}_q)_{q\in Q}$, such that $\mathsf{T}_F\nsubseteq\mathsf{T}^H$, for all finite $H\subseteq Q$.  Then $F\not\vdash Q$, as $F\vdash Q$ would mean $F\vdash H$, for some $H\in\mathsf{F} Q $, and hence $\mathsf{T}_F\subseteq\mathsf{T}^H$, by \eqref{Prime}, a contradiction.  Thus $\bigcap_{G\in\mathcal{G}}\mathsf{T}^G\nsubseteq\mathsf{T}^Q$, by \eqref{KeyLemma}, as required.

If $T\in\mathsf{T}_F$ then, as $T$ is round, we have $G\in\mathsf{B}T$ with $G\vdash\mathsf{1}F$ and hence $T\in\mathsf{T}_G\Subset T_F$, by \eqref{FCG}.  This shows that $\mathsf{T}^S$ is core compact and hence stably locally compact, as we already showed that $\mathsf{T}^S$ is ultrasober.

For \eqref{FDGT}, say $F\Vdash G$ and $T\in\mathsf{T}_F$, i.e. $F\subseteq T$.  As $T$ is round, we have $H\in\mathsf{B}T$ with $H\vdash\mathsf{1}F$ so $H\vdash G$ and hence $T\cap G\neq\emptyset$, as $T$ is prime, i.e. $T\in\mathsf{T}^G$, proving $\mathsf{T}_F\subseteq\mathsf{T}^G$.  Conversely, if $F\nVdash G$ then we have $H\in\mathsf{B}S$ with $H\vdash\mathsf{1}F$ but $H\nvdash G$ and hence $\mathsf{T}_F\nsubseteq\mathsf{T}^G$, by \eqref{KeyLemma}, thus completing the proof of \eqref{FDGT}.

For the converse of \eqref{FCGT}, assume $\vdash$ is a cover relation.  Say $F\not\vdash G$ and consider
\[G^{{}_1\!\dashv}=\{H\in\mathsf{F} S :\exists g\in G\ (H\vdash g)\}.\]
As each $T\in\mathsf{T}^G$ is round, $(\mathsf{T}_H)_{H\in G^{{}_1\!\dashv}}$ covers $\mathsf{T}^G$.  Now take any finite $\mathcal{H}\subseteq G^{{}_1\!\dashv}$, so $\mathcal{H}\vdash_{1\exists}G$  and hence $\mathcal{H}\vdash G$.  Note we must have $I\in\mathsf{B}S$ with $I\vdash\mathsf{1}F$ and $I\not\vdash J$, for some $J\in\mathcal{H}_\between$ \textendash\, otherwise $F\Vdash\mathcal{H}_\between$, by the definition of $\Vdash$, and hence $F\vdash G$, as $\vdash$ is auxiliary to $\Vdash$, a contradiction.  Then $\mathsf{T}_F\nsubseteq\mathsf{T}^J\supseteq\bigcup_{H\in\mathcal{H}}\mathsf{T}_H$, by  \eqref{KeyLemma}.  This shows that $(\mathsf{T}_H)_{H\in G^{{}_1\!\dashv}}$ has no finite subcover of $\mathsf{T}_F$ and hence $\mathsf{T}_F\not\Subset\mathsf{T}^G$, as required.
\end{proof}

So the above result says that we can represent any cover relation $\vdash$ on a set $S$ as the compact cover relation $\vdash_\Subset$ (see \autoref{SubsetCex}) on a subbasis $(\mathsf{T}_s)_{s\in S}$ of some stably locally compact space, namely the tight spectrum $\mathsf{T}^S$.  The finite subsets of $S$ then get represented as a basis $(\mathsf{T}_F)_{F\in\mathsf{F}S}$ which is also closed under finite intersections.  Specifically, these correspond unions in $\mathsf{F}S$, i.e. for all $F,G\in\mathsf{F}S$,
\[\mathsf{T}_F\cap\mathsf{T}_G=\mathsf{T}_{F\cup G}.\]
On this basis, $\subseteq$ and $\Subset$ correspond to $\Vdash_{1\forall}$ and $\vdash_{1\forall}$, thanks to \eqref{FDGT} and \eqref{FCGT}, i.e.
\begin{align*}
F\Vdash_{1\forall}G\qquad&\Leftrightarrow\qquad\mathsf{T}_F\subseteq\mathsf{T}_G\text{ and}\\
F\vdash_{1\forall}G\qquad&\Leftrightarrow\qquad\mathsf{T}_F\Subset\mathsf{T}_G,
\end{align*}
at least as long as $F\in\mathsf{B}S$ (which is automatic if $G\neq\emptyset$).  Note $(\mathsf{T}_F)_{F\in\mathsf{B}S}$ is still a basis for $\mathsf{T}^S$, one consisting entirely of sets contained in compact subsets.

Families $\mathcal{P}\subseteq\mathsf{B}S$ then correspond to open subsets of the spectrum defined by
\[\mathsf{T}_\mathcal{P}=\bigcup_{F\in\mathcal{P}}\mathsf{T}_F.\]
Finite intersections and arbitrary unions now correspond to the $\wedge$ and $\bigcup$ operations on $\mathsf{PB}S$, i.e. for any $\mathcal{Q},\mathcal{R}\subseteq\mathsf{B}S$ and any $\mathscr{P}\subseteq\mathsf{PB}S$,
\[\mathsf{T}_\mathcal{Q}\cap\mathsf{T}_\mathcal{R}=\mathsf{T}_{\mathcal{Q}\wedge\mathcal{R}}\qquad\text{and}\qquad\bigcup_{\mathcal{P}\in\mathscr{P}}\mathsf{T}_\mathcal{P}=\mathsf{T}_{\bigcup\mathscr{P}}.\]
Again thanks to \eqref{FDGT} and \eqref{FCGT}, $\subseteq$ and $\Subset$ can now be expressed in terms of $\vdash$ by
\begin{align}
\label{TQSub}\mathsf{T}_\mathcal{Q}\Subset\mathsf{T}_\mathcal{R}\qquad&\Leftrightarrow\qquad\exists\mathcal{G}\in\mathsf{F}\mathcal{R}\ (\mathcal{Q}\vdash\mathcal{G}_\between).\\
\label{TQsub}\mathsf{T}_\mathcal{Q}\subseteq\mathsf{T}_\mathcal{R}\qquad&\Leftrightarrow\qquad\mathcal{Q}{\downarrow}\subseteq\mathcal{R}{\downarrow},
\end{align}
where
\[\mathcal{Q}{\downarrow}=\mathsf{F}\mathcal{Q}^{\bowtie\circ{}_\forall\!\dashv}=\{F\in\mathsf{B}S:\exists\mathcal{G}\in\mathsf{F}\mathcal{Q}\ (F\vdash\mathcal{G}_\between)\}.\]
Indeed, open sets of $\mathsf{T}^S$ correspond to `quasi-ideals', i.e. $\mathcal{Q}\subseteq\mathsf{B}S$ with $\mathcal{Q}=\mathcal{Q}{\downarrow}$, as we show next in \autoref{QuasiIdealsOpenSubsets} below.

\subsection{Quasi-Ideals}\label{QI}

\begin{dfn}
Given a relation $\vdash$ on $\mathsf{F}S$, we call $\mathcal{Q}\subseteq\mathsf{B}S$ a \emph{quasi-ideal} if
\[\tag{Quasi-Ideal}F\in\mathcal{Q}\qquad\Leftrightarrow\qquad\exists\mathcal{G}\in\mathsf{F}\mathcal{Q}\ (F\vdash\mathcal{G}_\between).\]
We denote the set of all quasi-ideals by $\mathsf{Q}S=\{\mathcal{Q}\subseteq\mathsf{B}S:\mathcal{Q}=\mathcal{Q}{\downarrow}\}$.
\end{dfn}

The correspondence mentioned above, between quasi-ideals and open subsets of the tight spectrum, actually remains valid even when $\vdash$ is only a strong idempotent.

\begin{prp}\label{QuasiIdealsOpenSubsets}
If $\vdash$ is a strong idempotent on $\mathsf{F}S$ then $\mathcal{Q}\mapsto\mathsf{T}_\mathcal{Q}$ is an order isomorphism from the quasi-ideals $\mathsf{Q}S$ onto the open subsets of the spectrum $\mathsf{OT}^S$.
\end{prp}

\begin{proof}
As noted above, all opens of $\mathsf{T}^S$ are of the form $\mathsf{T}_\mathcal{P}$, for some $\mathcal{P}\subseteq\mathsf{B}S$.  Now note $\mathsf{T}_\mathcal{P}\subseteq\mathsf{T}_{\mathcal{P}\hspace{-1pt}\downarrow}$, as each $T\in\mathsf{T}^S$ is round, and $\mathsf{T}_\mathcal{P}\supseteq\mathsf{T}_{\mathcal{P}\hspace{-1pt}\downarrow}$, as each $T\in\mathsf{T}^S$ is prime.  As $\vdash\ =\ \vdash\bullet\vdash$, $\mathcal{P}\hspace{-3pt}\downarrow$ is a quasi-ideal.  This shows $\mathcal{Q}\mapsto\mathsf{T}_\mathcal{Q}$ maps $\mathsf{Q}S$ onto $\mathsf{OT}^S$.

If $\mathcal{P}\subseteq\mathcal{Q}$ then certainly $\mathsf{T}_\mathcal{P}\subseteq\mathsf{T}_\mathcal{Q}$.  Conversely, say we have $\mathcal{P},\mathcal{Q}\in\mathsf{Q}S$ with $\mathcal{P}\nsubseteq\mathcal{Q}$.  Take any $F\in\mathcal{P}\setminus\mathcal{Q}$.  As $\mathcal{P}$ is a quasi-ideal, we have $\mathcal{G}\in\mathsf{F}\mathcal{P}$ with $F\vdash\mathcal{G}_\between$.  We can further assume that $\emptyset\notin\mathcal{G}$.  Indeed, if $\emptyset\in\mathcal{G}\subseteq\mathcal{P}\subseteq\mathsf{B}S$ then $\mathcal{P}=\mathsf{B}S=\mathsf{F}S$, in which case we replace $\mathcal{G}$ with $\mathsf{1}H$, for any $H\in\mathsf{F}S$ with $F\vdash H$.

Now argue like in the proof of \eqref{KeyLemma}.  Specifically, by \eqref{BetweenInterpolation}, we have $(\mathcal{H}_n)\subseteq\mathsf{FF}S$ such that $\emptyset\notin\mathcal{H}_0=\mathcal{G}$ and $F\vdash_\between\mathcal{H}_{n+1}\vdash_*\mathcal{H}_n$, for all $n\geq0$.  If we had $\mathcal{H}_n\subseteq\mathcal{Q}$, for some $n\geq0$, then $F\in\mathcal{Q}$, as $\mathcal{Q}$ is a quasi-ideal, contradicting our choice of $F$.  Thus we can view $(\mathcal{H}_n\setminus\mathcal{Q})_{n\in\mathbb{N}}$ as the finite levels of an infinite tree.  K\"onig's lemma then yields an infinite branch, i.e. we have $I_n\in\mathcal{H}_n\setminus\mathcal{Q}$ with $I_{n+1}\vdash_{1\forall}I_n$, for all $n\geq0$.  Then we get round $R=\bigcup_{n\geq0}I_n\neq\emptyset$ such that, for all $G\in\mathsf{F}R$ and $\mathcal{J}\in\mathsf{F}\mathcal{Q}$, we have $J\in\mathcal{J}_\between$ with $G\nvdash J$ -- otherwise we could take $I_n\in\mathcal{H}_n\setminus\mathcal{Q}$ with $I_n\vdash\mathsf{1}G\bowtie\{G\}\vdash\mathcal{J}_\between$ and hence $I_n\in\mathcal{Q}$, as $\mathcal{Q}$ is a quasi-ideal, a contradiction.  \autoref{BirkhoffStone2} then yields tight $T\supseteq R$ with $H\nsubseteq T$, for all $H\in\mathcal{Q}$, which means that $T\in\mathsf{T}_{I_0}\setminus\mathsf{T}_\mathcal{Q}\subseteq\mathsf{T}_\mathcal{P}\setminus\mathsf{T}_\mathcal{Q}$ and hence $\mathsf{T}_\mathcal{P}\nsubseteq\mathsf{T}_\mathcal{Q}$.  This shows that $\mathcal{Q}\mapsto\mathsf{T}_\mathcal{Q}$ is indeed an order isomorphism.
\end{proof}

In particular, when $\vdash$ is a strong idempotent, the quasi-ideals form an arithmetic lattice, i.e. a continuous lattice where finite meets distribute over arbitrary joins and where the way-below relation $\ll$ respects pairwise meets, i.e.
\[p\ll q,r\qquad\Rightarrow\qquad p\ll q\wedge r.\]
This even remains valid for more general monotone cut-idempotents.  When $\vdash$ is also a strong idempotent or cover relation, we can also rephrase and reprove \autoref{TopRep} in more frame theoretic terms as a result about quasi-ideals.  In this context, subbases correspond to `generators' -- within a frame $L$, a subset $S$ is said to \emph{generate} $L$ when every element of $L$ can be obtained as some (potentially infinite) join of finite meets of elements in $S$, i.e.
\[\tag{Generating Subset}L=\{\bigvee_{F\in\mathcal{Q}}\bigwedge F:\mathcal{Q}\subseteq\mathsf{F}S\}.\]
While we are primarily interested concrete topological spaces, this frame theoretic approach is more in line with the original work of Vickers.  Indeed, the following result is partly analogous to \cite[Proposition 40]{Vickers2004}.

Recall that the sets $G^\dashv=\mathsf{1}G{\downarrow}$ and $g^\dashv=\{g\}^\dashv=\{\{g\}\}{\downarrow}$ appearing below are quasi-ideals, for all $G\in\mathsf{F}S$ and $g\in S$.

\begin{thm}\label{MCI}
If $\vdash$ is a monotone cut-idempotent on $\mathsf{F}S$ then the quasi-ideals $\mathsf{Q}S$ form an arithmetic lattice such that, for all $\mathcal{P},\mathcal{Q}\in\mathsf{Q}S$,
\begin{equation}\label{WayBelow}
\mathcal{P}\ll\mathcal{Q}\qquad\Leftrightarrow\qquad\exists\mathcal{G}\in\mathsf{F}\mathcal{Q}\ (\mathcal{P}\vdash\mathcal{G}_\between).
\end{equation}
Moreover, $\vdash$ is divisible and thus a strong idempotent precisely when, for all $G\in\mathsf{F}S$,
\begin{equation}\label{Gdashv}
G^\dashv=\bigvee_{g\in G}g^\dashv.
\end{equation}
Then $\mathsf{Q}S$ is generated by the principal quasi-ideals $(s^\dashv)_{s\in S}$ and, for all $F,G\in\mathsf{F}S$,
\begin{equation}\label{Vdashsubseteq}
F\Vdash G\qquad\Leftrightarrow\qquad\bigwedge_{f\in F}f^\dashv\subseteq\bigvee_{g\in G}g^\dashv.
\end{equation}
Also $\vdash$ is auxiliary to $\Vdash$ and thus a cover relation precisely when, for all $F,G\in\mathsf{F}S$,
\begin{equation}\label{vdashll}
F\vdash G\qquad\Leftrightarrow\qquad\bigwedge_{f\in F}f^\dashv\ll\bigvee_{g\in G}g^\dashv.
\end{equation}
\end{thm}

\begin{proof}
As $\vdash$ is cut-idempotent, $\mathcal{P}\hspace{-3pt}\downarrow$ is a quasi-ideal, for any $\mathcal{P}\subseteq\mathsf{B}S$.  The quasi-ideals thus form a complete lattice where the join of $\mathscr{P}\subseteq\mathsf{Q}S$ is given by
\[\bigvee\mathscr{P}=(\bigcup\mathscr{P}){\downarrow}.\]
If $\mathscr{D}\subseteq\mathsf{Q}S$ is directed then $\bigcup\mathscr{D}$ is already a quasi-ideal and hence $\bigvee\mathscr{D}=\bigcup\mathscr{D}$.  In particular, for any $\mathcal{Q}\in\mathsf{Q}S$, we see that
\[\mathcal{Q}=\bigcup_{\mathcal{G}\in\mathsf{F}\mathcal{Q}}\mathcal{G}\hspace{-3pt}\downarrow\ =\bigvee_{\mathcal{G}\in\mathsf{F}\mathcal{Q}}\mathcal{G}\hspace{-3pt}\downarrow.\]
If we have $\mathcal{P}\in\mathsf{Q}S$ with $\mathcal{P}\ll\mathcal{Q}$ then we must have $\mathcal{G}\in\mathsf{F}\mathcal{Q}$ with $\mathcal{P}\subseteq\mathcal{G}\hspace{-3pt}\downarrow$, i.e. $\mathcal{P}\vdash\mathcal{G}_\between$.  Conversely, if we have $\mathcal{G}\in\mathsf{F}\mathcal{Q}$ with $\mathcal{P}\vdash\mathcal{G}_\between$ then, whenever $\mathcal{Q}\subseteq\bigvee\mathscr{D}$, for some directed $\mathscr{D}\subseteq\mathsf{Q}S$, we can find $\mathcal{R}\in\mathscr{D}$ with $\mathcal{G}\subseteq\mathcal{R}$ and hence $\mathcal{P}\subseteq\mathcal{R}$, as $\mathcal{R}$ is a quasi-ideal, showing that $\mathcal{P}\ll\mathcal{Q}$.  This proves \eqref{WayBelow}.  In particular, $\mathcal{G}\hspace{-3pt}\downarrow\ \ll\mathcal{Q}$, whenever $\mathcal{G}\in\mathsf{F}\mathcal{Q}$ and $\mathcal{Q}\in\mathsf{Q}S$.  For any $\mathcal{Q}\in\mathsf{Q}S$, it follows that $\mathcal{Q}=\bigvee_{\mathcal{G}\in\mathsf{F}\mathcal{Q}}\mathcal{G}\hspace{-3pt}\downarrow\ =\bigvee_{\mathcal{P}\ll\mathcal{Q}}\mathcal{P}$, showing that $\mathsf{Q}S$ is a continuous lattice.

Finite meets, on the other hand, are given by intersections or, alternatively, by the $\wedge$ operation restricted $\mathsf{Q}S$.  Indeed, for any $\mathcal{P},\mathcal{Q}\in\mathsf{Q}S$, we immediately see that $\mathcal{P}\cap\mathcal{Q}\subseteq\mathcal{P}\wedge\mathcal{Q}$.  Also, if $\mathcal{P}$ and $\mathcal{Q}$ quasi-ideals and $F\vdash\mathcal{G}_\between$, for some $\mathcal{G}\in\mathsf{F}(\mathcal{P}\cap\mathcal{Q})=\mathsf{F}\mathcal{P}\cap\mathsf{F}\mathcal{Q}$, then $F\in\mathcal{P}\cap\mathcal{Q}$, i.e. $(\mathcal{P}\cap\mathcal{Q})\hspace{-3pt}\downarrow\ \subseteq\mathcal{P}\cap\mathcal{Q}$.  Conversely, for any $F\in\mathcal{P}$ and $G\in\mathcal{Q}$, we have $\mathcal{F}\in\mathsf{F}\mathcal{P}$ and $\mathcal{G}\in\mathsf{F}\mathcal{Q}$ with $F\vdash\mathcal{F}_\between$ and $G\vdash\mathcal{G}_\between$.  As $\vdash$ is a lower relation, it follows that $F\cup G\vdash\mathcal{F}_\between\cup\mathcal{G}_\between=(\mathcal{F}\wedge\mathcal{G})_\between$ and hence $F\cup G\in\mathcal{P}\cap\mathcal{Q}$, as $\mathcal{P}$ and $\mathcal{Q}$ are quasi-ideals.  This shows that $\mathcal{P}\wedge\mathcal{Q}\subseteq\mathcal{P}\cap\mathcal{Q}$ and $\mathcal{P}\wedge\mathcal{Q}\subseteq(\mathcal{P}\wedge\mathcal{Q})\hspace{-3pt}\downarrow$.  Thus $\mathcal{P}\cap\mathcal{Q}=\mathcal{P}\wedge\mathcal{Q}$ is a quasi-ideal.  Now if $\mathcal{R}\ll\mathcal{P}$ and $\mathcal{R}\ll\mathcal{Q}$ then we have $\mathcal{F}\in\mathsf{F}\mathcal{P}$ and $\mathcal{G}\in\mathsf{F}\mathcal{Q}$ with $\mathcal{R}\vdash\mathcal{F}_\between\cup\mathcal{G}_\between=(\mathcal{F}\wedge\mathcal{G})_\between$ and hence $\mathcal{R}\ll\mathcal{P}\wedge\mathcal{Q}$, showing that $\mathsf{Q}S$ is an arithmetic lattice.

To see that $\mathsf{Q}S$ is a frame, i.e. that finite meets distribute over infinite joins, take any $\mathscr{P}\subseteq\mathsf{Q}S$ and $\mathcal{Q}\in\mathsf{Q}S$.  It suffices to show that $(\bigvee\mathscr{P})\wedge\mathcal{Q}\subseteq\bigvee_{\mathcal{P}\in\mathscr{P}}(\mathcal{P}\wedge\mathcal{Q})$.  Accordingly, take any $F\in\bigvee\mathscr{P}$ and $G\in\mathcal{Q}$, so we have $\mathcal{H}\in\mathsf{F}(\bigcup\mathscr{P})$ such that $F\vdash\mathcal{H}_\between$.  As $\mathcal{Q}$ is a quasi-ideal, we also have $\mathcal{I}\in\mathsf{F}\mathcal{Q}$ with $G\vdash\mathcal{I}_\between$ and hence
\[F\cup G\vdash\mathcal{H}_\between\cup\mathcal{I}_\between=(\mathcal{H}\wedge\mathcal{I})_\between.\]
As $\mathcal{H}\wedge\mathcal{I}\subseteq\mathcal(\bigcup\mathscr{P})\wedge\mathcal{Q}=\bigcup_{\mathcal{P}\in\mathscr{P}}(\mathcal{P}\wedge\mathcal{Q})$, this means that $F\cup G\in\bigvee_{\mathcal{P}\in\mathscr{P}}(\mathcal{P}\wedge\mathcal{Q})$.  This shows that $\mathcal{Q}\wedge(\bigvee\mathscr{P})\subseteq\bigvee_{\mathcal{P}\in\mathscr{P}}(\mathcal{P}\wedge\mathcal{Q})$, as required.

This completes the proof of the first part, i.e. we have shown that $\mathsf{Q}S$ is an arithmetic lattice whenever $\vdash$ is a monotone cut-idempotent.  We certainly always have $g^\dashv\subseteq G^\dashv$, whenever $g\in G\in\mathsf{F}S$, and hence $\bigvee_{g\in G}g^\dashv\subseteq G^\dashv$.  The reverse inclusion also holds precisely when $F\vdash G$ always implies $F\in\bigvee_{g\in G}g^\dashv$.  But $F\in\bigvee_{g\in G}g^\dashv$ is saying that we have $\mathcal{G}\in\mathsf{F}(\bigcup_{g\in G}g^\dashv)$, i.e. $\mathcal{G}\vdash_{1\exists}G$, such that $F\vdash\mathcal{G}_\between$ and hence $F\vdash\bullet\vdash_{1\exists}G$.  Thus $\vdash\ \subseteq\ \vdash\bullet\vdash_{1\exists}$, i.e. $\vdash$ is divisible, precisely when \eqref{Gdashv} holds, i.e. $G^\dashv=\bigvee_{g\in G}g^\dashv$, for all $G\in\mathsf{F}S$.

In particular, if $\vdash$ is divisible then, for all $F,G\in\mathsf{F}S$,
\[\bigwedge_{f\in F}f^\dashv\subseteq\bigvee_{g\in G}g^\dashv\ \ \Leftrightarrow\ \ \{F\}\hspace{-3pt}\downarrow\ \vdash G\ \ \Leftrightarrow\ \ \forall H\in\mathsf{B}S\ (H\vdash\mathsf{1}F\ \Rightarrow\ H\vdash G)\ \ \Leftrightarrow\ \  F\Vdash G.\]
This proves and \eqref{Vdashsubseteq} and it follows from this and \eqref{WayBelow} that
\[\bigwedge_{f\in F}f^\dashv\ll\bigvee_{g\in G}g^\dashv\qquad\Leftrightarrow\qquad\exists\mathcal{H}\in\mathsf{F}(G^\dashv)\ (\{F\}\hspace{-3pt}\downarrow\ \vdash\mathcal{H}_\between)\qquad\Leftrightarrow\qquad F\Vdash\bullet\vdash G.\]
Thus $\vdash\ =\ \Vdash\bullet\vdash$, i.e. $\vdash$ is auxiliary to $\Vdash$, precisely when \eqref{vdashll} holds.

If $\vdash$ is divisible then we further claim that, for all $\mathcal{F},\mathcal{G}\in\mathsf{FB}S$,
\begin{equation}\label{FunionG}
\mathcal{F}\hspace{-3pt}\downarrow\vee\ \mathcal{G}\hspace{-3pt}\downarrow\ =(\mathcal{F}\cup\mathcal{G}){\downarrow}.
\end{equation}
Again, $\mathcal{F}\hspace{-3pt}\downarrow\vee\ \mathcal{G}\hspace{-3pt}\downarrow\ \subseteq(\mathcal{F}\cup\mathcal{G}){\downarrow}$ is immediate.  Conversely, $F\in(\mathcal{F}\cup\mathcal{G}){\downarrow}$ means that $F\vdash(\mathcal{F}\cup\mathcal{G})_\between$ and hence $F\vdash_\between\mathcal{H}\vdash_*\mathcal{F}\cup\mathcal{G}$, for some $\mathcal{H}\in\mathsf{FF}S$, by \eqref{BetweenInterpolation}.  Note $\mathcal{H}\vdash_*\mathcal{F}\cup\mathcal{G}$ implies $\mathcal{H}\subseteq\mathcal{F}\hspace{-3pt}\downarrow\cup\ \mathcal{G}\hspace{-3pt}\downarrow$.  As $F\vdash\mathcal{H}_\between$, this shows that $F\in(\mathcal{F}\hspace{-3pt}\downarrow\cup\ \mathcal{G}\hspace{-3pt}\downarrow){\downarrow}=\mathcal{F}\hspace{-3pt}\downarrow\vee\ \mathcal{G}\hspace{-3pt}\downarrow$, which proves the claim.  For all $\mathcal{G}\in\mathsf{FB}S$, it follows that
\[\mathcal{G}\hspace{-3pt}\downarrow\ =\big(\bigcup_{G\in\mathcal{G}}\{G\}\big)\hspace{-3pt}\downarrow\ =\bigvee_{G\in\mathcal{G}}\{G\}\hspace{-3pt}\downarrow\ =\bigvee_{G\in\mathcal{G}}\bigcap_{g\in G}g^\dashv\ =\bigvee_{G\in\mathcal{G}}\bigwedge_{g\in G}g^\dashv.\]
For any $\mathcal{Q}\in\mathsf{Q}S$, it follows that
\[\mathcal{Q}=\bigcup_{\mathcal{G}\in\mathsf{F}\mathcal{Q}}\mathcal{G}\hspace{-3pt}\downarrow\ =\bigvee_{\mathcal{G}\in\mathsf{F}\mathcal{Q}}\mathcal{G}\hspace{-3pt}\downarrow\ =\bigvee_{\mathcal{G}\in\mathsf{F}\mathcal{Q}}\bigvee_{G\in\mathcal{G}}\bigwedge_{g\in G}g^\dashv\ =\bigvee_{G\in\mathcal{Q}}\bigwedge_{g\in G}g^\dashv.\]
This shows that $\mathsf{Q}S$ is indeed generated by the principal quasi-ideals $(s^\dashv)_{s\in S}$.
\end{proof}

For general monotone cut-idempotents, there are thus two competing notions of what a `finite' quasi-ideal could be.  On the one hand, we have the \emph{finitely generated} quasi-ideals $(\bigvee_{F\in\mathcal{F}}\{F\}\hspace{-3pt}\downarrow)_{\mathcal{F}\in\mathsf{FB}S}$, which form a sublattice of $\mathsf{Q}S$ and also a distributive proximity lattice w.r.t. $\ll$.

On the other hand, we have the \emph{finitely principal} quasi-ideals $(\mathcal{F}\hspace{-3pt}\downarrow)_{\mathcal{F}\in\mathsf{FB}S}$, which form a basis for $\mathsf{Q}S$, i.e. every quasi-ideal is a (infinite) join of finitely principal quasi-ideals.  However, they do not form a lattice, as they only closed under finite meets, not finite joins -- while they do have a join-like operation $(\mathcal{F}\hspace{-3pt}\downarrow,\mathcal{G}\hspace{-3pt}\downarrow)\mapsto(\mathcal{F}\cup\mathcal{G})\hspace{-3pt}\downarrow$, these are not joins in $\mathsf{Q}S$ unless $\vdash$ is divisible (see \eqref{FunionG} above).  Thus it is only for strong idempotents that these two classes of `finite' quasi-ideals coincide and then also satisfy \eqref{veeInterpolation}.

To illustrate some of these deficiencies, we give another simple example.

\begin{xpl}\label{DyadicXpl}
Say we are given a Scott relation $\vDash$ on $\mathsf{F}S$ and a stronger monotone cut-idempotent $\vdash\ \subseteq\ \vDash$ which is also auxiliary to $\vDash$.  We claim that their cut-composition $\vdash\bullet\vDash$ is also a monotone cut-idempotent.  Indeed, monotonicity immediately passes to the composition, while for cut-idempotence just note that
\[\vdash\bullet\vDash\ \ =\ \ \vdash\bullet\vdash\bullet\vDash\ \ =\ \ \vdash\bullet\vDash\bullet\vdash\bullet\vDash.\]

However, $\vdash\bullet\vDash$ need not be divisible, even if $\vdash$ is.  For example, $\vDash$ and $\vdash$ could be the relations from \autoref{DenseCovers} and \autoref{SubsetCex} respectively defined on finite families $\mathsf{FO}X$ of open subsets of some topological space $X$, i.e.
\begin{align*}
F\vDash G\qquad&\Leftrightarrow\qquad\bigcap F\subseteq\mathrm{cl}(\bigcup G).\\
F\vdash G\qquad&\Leftrightarrow\qquad\bigcap F\Subset\bigcup G.
\end{align*}
In this case, $\vdash\bullet\vDash$ can be characterised as follows -- for all $F,G\in\mathsf{FO}X$,
\[F\vdash\bullet\vDash G\qquad\Leftrightarrow\qquad\bigcap F\Subset\mathrm{int}(\mathrm{cl}(\bigcup G)).\]
For example, let $X=(0,1)$ in its usual topology and let $p=(0,\frac{1}{2})$ and $q=(\frac{1}{2},1)$.  Then $X\vdash\{p,q\}$ even though $X\not\vdash\{r,s\}$ for any $r\vdash p$ and $s\vdash q$.  It follows that $\vdash$ is not divisible and joins do not correspond to unions in $\mathsf{F}S$, e.g.
\[(\{p\}\cup\{q\})^\dashv=\{p,q\}^\dashv=X^\dashv\neq(p\cup q)^\dashv=p^\dashv\vee q^\dashv.\]

Now let us restrict to the family $S$ of all open subintervals of $(0,1)$ formed from consecutive dyadic rationals, i.e.
\[S=\{((k-1)/2^n,k/2^n):1\leq k\leq 2^n\}.\]
The cut-composition of the restriction of the relations above to $\mathsf{F}S$ does not have the same characterisation, however it can be verified directly that we still obtain monotone cut-idempotent $\vdash$ on $\mathsf{F}S$ when we define
\[F\vdash G\qquad\Leftrightarrow\qquad\bigcap F\Subset\mathrm{int}(\mathrm{cl}(\bigcup G)).\]
In this case, the principal quasi-ideals do not even generate $\mathsf{Q}S$, for example $\{s\in S:s\Subset(0,\frac{3}{4})\}$ is a quasi-ideal which is not generated by principal quasi-ideals (as $(0,\frac{3}{4})$ is not a union of elements of $S$, only the smaller subset $(0,\frac{1}{2})\cup(\frac{1}{2},\frac{3}{4})$ is).  Moreover, the tight spectrum of $S$ consists of subsets $S_x=\{s\in S:x\in s\}$ where $x\in(0,1)$ is not a dyadic rational (roundness fails when $x$ is dyadic).  We can thus identify $\mathsf{T}^S$ with the non-dyadic points of $(0,1)$ in the usual subspace topology, which is not even core compact.
\end{xpl}

\subsection{Recovery}
Our next goal is to show that all stably locally compact spaces arise via \autoref{TopRep}.  This follows from the result below, which says that any stably locally compact space can be recovered from the spectrum of any subbasis under the cover relation $\vdash_\Subset$ arising from compact containment.  We can even say something about more general core compact spaces.

First we denote the \emph{specialisation preorder} on $X$ by $\rightarrow$, i.e. for any subbasis $S$,
\[x\rightarrow y\qquad\Leftrightarrow\qquad\forall p\in S\ (y\in p\ \Rightarrow\ x\in p).\]
So $x\rightarrow y$ iff $x$ converges to $y$ in the topology generated by $S$, viewing $x$ as a singleton net or a constant sequence.  Further denote the \emph{saturation} of $Y\subseteq X$ by
\[\mathrm{sat}(Y)=Y^\leftarrow=\{x\in X:\exists y\in Y\ (x\rightarrow y)\}=\bigcap\{O\in\mathsf{O}X:Y\subseteq O\}.\]

The \emph{Alexandroff topology} on $X$ is generated by point-closures $\mathrm{cl}\{x\}=x^\rightarrow$ in the original topology.  It then follows that arbitrary closed subsets in the original topology are open in the Alexandroff topology.  Moreover, saturation in the original topology just becomes the closure in the Alexandroff topology.

The \emph{strong patch topology} is that generated by the original topology and the Alexandroff topology, i.e. by both open and closed subsets.  We call a subspace $D\subseteq X$ \emph{strongly dense} if it is dense in the strong patch topology.

\begin{prp}\label{StronglyDense}
A subset $Y$ is strongly dense in $X$ if and only if, for all $O\subseteq X$,
\[O\text{ is open}\qquad\Rightarrow\qquad\mathrm{sat}(Y\cap O)=O.\]
\end{prp}

\begin{proof}
Say $Y$ is not strongly dense in $X$, so we have open $O$ and closed $C$ in the original topology of $X$ where $O\cap C\neq\emptyset$ is disjoint from $Y$.  But then $O\setminus C$ is an open subset of the original topology which contains $Y\cap O$, showing that $\mathrm{sat}(Y\cap O)\neq O$.  Conversely, if we have open $O$ with $\mathrm{sat}(Y\cap O)\neq O$ then, taking $x\in O\setminus \mathrm{sat}(Y\cap O)$, we see that $O\cap x^\rightarrow$ is disjoint from $Y$, contradicting strong density.
\end{proof}

So if $Y$ is strongly dense in $X$ then the usual map $O\mapsto O\cap Y$ from $\mathsf{O}X$ onto $\mathsf{O}Y$ is actually an order isomorphism with inverse map $N\mapsto\mathrm{sat}(N)$.  In particular, any properties of $X$ that depend only on the order structure of its open subsets immediately pass to strongly dense subspaces, e.g. any strongly dense subspace of a core compact or core coherent space is again core compact or core coherent.

The \emph{patch topology} on $X$ is generated by open sets and complements of compact saturated sets.  We let $X_\square$ denote $X$ with this topology and call it the \emph{patch space}.  We call $D\subseteq X$ \emph{very dense} if $D$ is dense in both the patch topology and Alexandroff topology, i.e. $X$ is both the patch-closure and saturation of $D$,
\[\tag{Very Dense}X=\mathrm{sat}(D)=\mathrm{cl}_\square(D).\]

For example, if $X=\mathbb{R}$ in the lower topology generated by $x^>$, for $x\in\mathbb{R}$, then the patch topology is the usual topology generated by intervals $(x,y)=x^<\cap y^>$, for $x,y\in\mathbb{R}$, while the strong patch topology yields the Sorgenfrey line, i.e. the topology generated by intervals $[x,y)=x^\leq\cap y^>$, for $x,y\in\mathbb{R}$.

We follow the convention that a subbasis must cover the space.

\begin{thm}\label{Recovery}
If $S$ is a subbasis of $T_0$ core compact $X$ and $\vdash\ =\ \vdash_\Subset$ then
\[x\mapsto S_x=\{p\in S:x\in p\}\]
is a homeomorphism from $X$ onto a very dense subspace of $\mathsf{T}^S$.  Moreover,
\begin{enumerate}
\item If $X$ is core coherent then $(S_x)_{x\in X}$ is even strongly dense in $\mathsf{T}^S$.
\item if $X$ is also sober then $(S_x)_{x\in X}$ is actually the whole of $\mathsf{T}^S$.
\end{enumerate}
\end{thm}

\begin{proof}
If $x\in\bigcap F\Subset\bigcup G$ then certainly $S_x\cap G\neq\emptyset$.  Conversely, if $x\in g\in S$ then core compactness yields $F\in\mathsf{F}S$ with $x\in\bigcap F\Subset g$ and hence $F\subseteq S_x$.  This shows that $S_x$ is tight so $x\mapsto S_x$ is indeed a map from $X$ to $\mathsf{T}^S$.  As $X$ is $T_0$ and $S$ is a subbasis of $X$, the map is injective.  As $(\mathsf{T}_p)_{p\in S}$ is a subbasis of $\mathsf{T}^S$ and
\begin{equation}\label{xinp}
x\in p\qquad\Leftrightarrow\qquad p\in S_x\qquad\Leftrightarrow\qquad S_x\in\mathsf{T}_p,
\end{equation}
for all $p\in S$, the map is also a homeomorphism onto its range.

To see that $(S_x)_{x\in X}$ is patch-dense in $\mathsf{T}^S$, say we have open $O$ and compact saturated $C$ in $\mathsf{T}^S$ such that $O\setminus C\neq\emptyset$.  Taking $T\in O\setminus C$, we must have $F,G\in\mathsf{F} S $ with $T\in\mathsf{T}_F\subseteq O$ and $T\notin\mathsf{T}^G\supseteq C$, as $C$ is compact saturated.  As $F\subseteq T$, $G\cap T=\emptyset$ and $T$ is tight, it follows that $F\not\vdash G$, i.e. $\bigcap F\not\Subset\bigcup G$.  As $\mathsf{T}^S$ is core compact, we have $\mathcal{H}\in\mathsf{F}\mathsf{F} S  $ with $C\subseteq\bigcup_{H\in\mathcal{H}}\mathsf{T}_H\Subset\mathsf{T}^G$ and hence $\mathcal{H}\vdash G$, i.e. $\bigcup_{H\in\mathcal{H}}\bigcap H\Subset\bigcup G$.  This means $\bigcap F\nsubseteq\bigcup_{H\in\mathcal{H}}\bigcap H$, as $\bigcap F\not\Subset\bigcup G$, so we have $x\in\bigcap F\setminus\bigcup_{H\in\mathcal{H}}\bigcap H$ and hence $S_x\in\mathsf{T}_F\setminus\bigcup_{H\in\mathcal{H}}\mathsf{T}_H\subseteq O\setminus C$, showing that $(S_x)_{x\in X}$ is patch-dense in $\mathsf{T}^S$.

To see that $\mathsf{T}^S$ is the saturation of $(S_x)_{x\in X}$, take $T\in\mathsf{T}^S$.  If we had $X=\bigcup(S\setminus T)$ then, taking any $t\in T$ and $F\in\mathsf{F}T$ with $\bigcap F\Subset t$, we would have $\bigcap F\Subset\bigcup G$, for some finite $G\subseteq S\setminus T$, contradicting the tightness of $T$.  Taking $x\in X\setminus\bigcup(S\setminus T)$ we see that $S_x\subseteq T$ and hence $T\rightarrow S_x$.  Thus $(S_x)_{x\in X}$ is very dense in $\mathsf{T}^S$.

This completes the proof of the first part of the theorem, while the remaining two parts, when $X$ is also core coherent and sober, are proved as follows.

\begin{enumerate}
\item Assume $X$ is also core coherent.  To see that $(S_x)_{x\in X}$ is strongly dense in $\mathsf{T}^S$ it suffices to show that, whenever $F\in\mathsf{F}S$ and $T\in\mathsf{T}_F$, we have some $x\in S$ such that $S_x\in\mathsf{T}_F\cap T^\rightarrow$.  Arguing as above note that, for each $f\in F$, we have $G_f\in\mathsf{F}T$ with $G_f\Subset f$ and hence $\bigcap_{f\in F}G_f\Subset\bigcap F$, by core coherence.  Then $\bigcap F\subseteq\bigcup(S\setminus T)$ would imply $\bigcap F\Subset\bigcup G$, for some finite $G\subseteq S\setminus T$, contradicting the tightness of $T$.  Thus we have some $x\in\bigcap F\setminus\bigcup(S\setminus T)$, which implies that $S_x\in\mathsf{T}_F$ and $S_x\subseteq T$, i.e. $S_x\in\mathsf{T}_F\cap T^\rightarrow$, as required.

\item Now assume $X$ is stably locally compact and take $T\in\mathsf{T}^S$.  We claim that $C=X\setminus\bigcup(S\setminus T)$ is irreducible.  This amounts to showing that, for all finite $F$ in the subbasis, $\bigcap F$ intersects $C$ whenever each $f\in F$ intersects $C$.  As $C=X\setminus\bigcup(S\setminus T)$, the only $p\in S$ intersecting $C$ are the elements of $T$, i.e. it suffices to show $\bigcap F\cap C\neq\emptyset$, for all finite $F\subseteq T$.  To see this, note $F\cap C=\emptyset$ would mean $F\subseteq\bigcup S\setminus T$.  As $T$ is round, we would then have finite $G\subseteq T$ with $G\Subset f$, for all $f\in F$.  As $X$ is core coherent, this means $\bigcap G\Subset\bigcap F$.  By core compactness, we can take $D$ with $\bigcap G\Subset D\Subset\bigcap F\subseteq\bigcup S\setminus T$ and then get finite $H\subseteq S\setminus T$ with $\bigcap G\Subset D\subseteq\bigcup H$ so $T\supseteq G\vdash H\subseteq S\setminus T$, contradicting the tightness of $T$.

Thus $C$ is irreducible and hence $C=\mathrm{cl}\{x\}$, for some $x\in X$, as $X$ is also sober.  If we had $t\in T\setminus S_x$, it would follow that $t\cap C=\emptyset$ and hence $t\subseteq\bigcup(S\setminus T)$, but this would contradict the tightness of $T$ as above.  Thus we must have $T\subseteq S_x$, while $S_x\subseteq T$ is immediate from the fact that $x\notin\bigcup(S\setminus T)$.  Thus $T=S_x$, showing that the map $x\mapsto S_x$ is surjective. \qedhere
\end{enumerate}
\end{proof}

This immediately allows us to characterise core compact core coherent spaces as certain subspaces of stably locally compact spaces.

\begin{cor}
Core compact core coherent topological spaces are precisely the strongly dense subspaces of stably locally compact spaces.
\end{cor}

\begin{proof}
Any strongly dense subspace of a stably locally compact space is still core compact and core coherent, as noted after \autoref{StronglyDense}.  Conversely, any core compact core coherent space $X$ can be identified with a strongly dense subspace of the tight spectrum $\mathsf{T}^S$, where $S$ is any subbasis of $X$, by \autoref{Recovery}.
\end{proof}

Incidentally, this could also be achieved by the using the usual sobrification of a space (see \cite[\S8.2.3]{Goubault2013}) rather than the tight spectrum considered here.  One just needs to note that any $T_0$ space is strongly dense in its sobrification.

Consider the following classes.
\begin{align*}
\mathbf{Cov}&=\{(S,\vdash):\text{$\vdash$ is a cover relation on }\mathsf{F}S\}.\\
\mathbf{ACov}&=\{(S,\vdash):\text{$\vdash$ is a cover relation on }\mathsf{F}S\text{ and $\Vdash$ is $1$-antisymmetric}\}.\\
\mathbf{Sub}&=\{(S,X):S\text{ is a subbasis of a core compact $T_0$ space }X\}.\\
\mathbf{SSub}&=\{(S,X):S\text{ is a subbasis of a stably locally compact space }X\}.
\end{align*}
(here $1$-antisymmetric means ${}_1{\Vdash_1}$ is antisymmetric, i.e. $p\Vdash q\Vdash p$ implies $p=q$).  We refer to the elements of $\mathbf{Cov}$ as \emph{cover systems}, while those in $\mathbf{ACov}$ are also said to be \emph{antisymmetric}.  We can summarise our main results thus far as follows.
\begin{thm}\label{ACovSSub}
$\mathbf{ACov}$ is dual to $\mathbf{SSub}$.
\end{thm}

More precisely, every subbasis $S$ of a stably locally compact space $X$ forms an antisymmetric cover system under the compact cover relation $\vdash_\Subset$, as we saw in \autoref{SubsetCex}.  Moreover, the underlying space $X$ can be recovered via the spectrum, by \autoref{Recovery}.  Conversely, any antisymmetric cover system $(S,\vdash)$ can be faithfully represented as the subbasis $(\mathsf{T}_p)_{p\in S}$ of the stably locally compact space $\mathsf{T}^S$, by \autoref{TopRep}.  Moreover, $\vdash$ can be recovered as the compact cover relation on this subbasis, by \eqref{FCGT}.

More formally, what we have here is an equivalence of categories but where we only consider isomorphisms in both $\mathbf{ACov}$ and $\mathbf{SSub}$ (in the next section we will consider more general morphisms arising from continuous maps).  So, at the isomorphism level at least, this duality between $\mathbf{ACov}$ and $\mathbf{SSub}$ extends the classic Stone duality between distributive lattices (considered as cover systems via \eqref{MeetJoinCover}) and compact open lattice bases of sober spaces.

For the corresponding Priestley duality extension, we consider patch topologies and `pseudosubbases'.

\begin{dfn}
We call $S\subseteq\mathsf{O}X$ a \emph{pseudosubbasis} of $X$ if $S$ is $\cap$-round and
\[\tag{$T_0$}x,y\in X\qquad\Rightarrow\qquad\exists p\in S\ (|p\cap\{x,y\}|=1).\]
\end{dfn}

For any open subsets $O$ and $N$ of Hausdorff $X$, note $O\Subset N$ means $\mathrm{cl}(O)$ is compact and contained in $N$.  If $P$ is pseudosubbasis of $X$ then, for any $x\in X$, we have $p\in P$ with $x\in p$ and then $\cap$-roundness yields $F\in\mathsf{F}P$ with $x\in\bigcap F\Subset s$.  In particular, this shows that each $x\in X$ has a compact neighbourhood.  Thus any Hausdorff space with a pseudosubbasis is automatically locally compact.

Let
\[\mathbf{PSub}=\{(P,X):P\text{ is a pseudosubbasis of a Hausdorff space }X\}.\]
In \cite{BiceStarling2020HTight}, we showed $\mathbf{PSub}$ is the same as $\mathbf{SSub}$.  Specifically, if $(P,X)\in\mathbf{PSub}$ then $P$ generates a stably locally compact topology whose patch topology recovers the original Hausdorff topology.  Conversely, if $(P,X)\in\mathbf{SSub}$ then $P$ is still a pseudosubbasis in the finer patch topology.  We can thus rephrase \autoref{ACovSSub} as
\begin{thm}\label{ACovPSub}
$\mathbf{ACov}$ is dual to $\mathbf{PSub}$.
\end{thm}

As any pseudosubbasis induces a closed partial order, namely the specialisation preorder w.r.t. the topology it generates, pseudosubbases can always be viewed as up-subbases, i.e. subbases for the up-sets in the corresponding pospace.  Thus the duality between $\mathbf{ACov}$ and $\mathbf{PSub}$ extends the classic Priestley duality between distributive lattices and compact open lattice up-bases of Hausdorff pospaces.

We can also prove the Priestley analog of \autoref{Recovery} more directly.

\begin{thm}\label{HausdorffRecovery}
If $X$ is Hausdorff, $S$ is a pseudosubbasis for $X$ and $\vdash\ =\ \vdash_\Subset$ as in \eqref{SubsetC} then we have a homeomorphism from $X$ onto $\mathcal{T}^S_\square$ given by
\[x\mapsto S_x=\{p\in S:x\in p\}.\]
\end{thm}

\begin{proof}
First note that we can essentially repeat the first part of the proof of \autoref{Recovery}.  Specifically, if $x\in\bigcap F\Subset\bigcup G$ then certainly $S_x\cap G\neq\emptyset$.  Conversely, if $x\in g\in G\in\mathsf{F} S $ then, as $S$ is $\cap$-round, we have $F\in\mathsf{F} S $ with $x\in\bigcap F\Subset g$ and hence $F\subseteq S_x$.  This shows that $S_x$ is tight so $x\mapsto S_x$ is indeed a map from $X$ to $\mathcal{T}^S$.  As $S$ is $T_0$, the map is injective.

To see that it is also surjective, take any tight $T\subseteq S$.  As $T$ is round and $X$ is Hausdorff, $\bigcap T=\bigcap_{t\in T}\mathrm{cl}(t)$.  Also note that $\bigcap F\setminus\bigcup G\neq\emptyset$ for any finite $F\subseteq T$ and $G\subseteq S\setminus T$ \textendash\, otherwise $\bigcap F\subseteq\bigcup G$ and the roundness of $T$ would yield finite $H\subseteq T$ with $\bigcap H\Subset\bigcap F\subseteq\bigcup G$, i.e. $H\vdash G$, contradicting the tightness of $T$.  Thus
\[\bigcap T\setminus\bigcup(S\setminus T)=\bigcap_{t\in T}\mathrm{cl}(t)\cap\bigcap_{s\notin T}X\setminus s\]
is an intersection of compact sets with the finite intersection property and is thus non-empty.  Taking $x\in\bigcap T\setminus\bigcup(S\setminus T)$, we immediately see that $S_x=T$.

As in \eqref{xinp}, we then see that $x\mapsto S_x$ is a homeomorphism from $X$ in the topology generated by $S$ to $\mathsf{T}^S$ (in the topology generated by $(\mathsf{T}_s)_{s\in S}$).  It is thus also a homeomorphism with respect to the corresponding patch topologies which, in the case of $X$, coincides with the original topology, by \cite[Proposition 2.37]{BiceStarling2020HTight}.
\end{proof}

\subsection{Stabilisation}

Even for a core compact $T_0$ space $X$, \autoref{Recovery} still yields a \emph{stabilisation} of $X$, i.e. a stably locally compact space $X'$ containing $X$ as a very dense subspace.  In fact, we can show that $\mathsf{T}^S$ above is the minimal stabilisation w.r.t. \emph{proximal} $\phi:Y\rightarrow X$, meaning that $\phi$ is continuous and, for all $O,N\in\mathsf{O}X$,
\[\tag{Proximal Map}O\Subset N\qquad\Rightarrow\qquad\phi^{-1}[O]\Subset\phi^{-1}[N].\]
Note that, for locally compact sober spaces, proximal just means \emph{proper}, i.e. preimages of compact saturated sets are again compact (see \cite[Lemma V-5.19]{GierzHofmannKeimelLawsonMisloveScott2003}).

Recall that a proper map $\psi:Z\rightarrow K$ with dense image $\psi[Z]$ has to be surjective if $K$ is locally compact and Hausdorff.  Indeed then any $k\in K$ has a relatively compact open neighbourhood $O$ so $\psi^{-1}[\mathrm{cl}(O)]$ is compact, by properness.  Then $\psi[\psi^{-1}[\mathrm{cl}(O)]]$ compact, by continuity, and hence closed, as $L$ is Hausdorff.  The density of $\psi[Z]$ then implies $\mathrm{cl}(O)=\psi[\psi^{-1}[\mathrm{cl}(O)]]$ and, in particular, $k\in\psi[Z]$.

As usual, we denote surjective maps by a double-headed arrow $\twoheadrightarrow$.  Also note below that we are not placing any additional assumptions on $Y$ beyond it being a very dense subspace of a stably locally compact space $Z$.

\begin{thm}\label{MinimalStabilisation}
If $X$ is core compact $T_0$ with subbasis $S$ and $Z$ is a stabilisation of $Y$, any proximal $\phi:Y\rightarrow X$ has a unique proximal extension $\psi:Z\twoheadrightarrow\mathsf{T}^S$, i.e.
\begin{equation}\label{psiy}
\psi(y)=S_{\phi(y)},\quad\text{for all }y\in Y.
\end{equation}
\end{thm}

\begin{proof}
Let $R=\mathsf{O}Z$.  For each $z\in Z$, define
\begin{align*}
\psi(z)&=\{s\in S:\exists r\in R_z\ (\phi[r\cap Y]\Subset s)\}.
\end{align*}
First we prove \eqref{psiy}.  If $s\in\psi(y)$ then we have $r\in R_y$ with $\phi[r\cap Y]\Subset s$ so, in particular, $\phi(y)\in s$, showing that $\psi(y)\subseteq S_{\phi(y)}$.  Conversely, if $\phi(y)\in s\in S$ then core compactness yields $O\in\mathsf{O}X$ with $\phi(y)\in O\Subset s$.  As $\phi$ is continuous, $\phi^{-1}[O]$ is open in $Y$ so we have $r\in R$ with $r\cap Y=\phi^{-1}[O]$.  It follows that $y\in r\cap Y$ and $\phi[r\cap Y]\subseteq O\Subset s$ and hence $s\in\psi(y)$, showing that $S_{\phi(y)}\subseteq\psi(y)$, as required.

Next we claim $\psi(z)\in\mathsf{T}^S$, for all $z\in Z$.  To prove \eqref{Prime}, fix $z\in Z$ and take $F\in\mathsf{F}(\psi(z))$ and $G\in\mathsf{F}S$ with $\bigcap F\Subset\bigcup G$.  This means, for each $f\in F$, we have $r_f\in R_z$ with $\phi[r_f\cap Y]\Subset f$.  Setting $r=\bigcap_{f\in F}r_f\in R$ and $O=\bigcap F\in\mathsf{O}X$, it follows that $z\in r$ and $\phi[r\cap Y]\subseteq O\Subset\bigcup G$.  Note that $r\cap Y\neq\emptyset$, as $Y$ is (patch-)dense in $Z$, so $O\neq\emptyset$ and hence $G\neq\emptyset$.

Take $\mathcal{H}\in\mathsf{F}\mathsf{F}S$ with $O\Subset\bigcup_{H\in\mathcal{H}}\bigcap H$ and $\mathcal{H}\vdash_{1\exists}G$, i.e. for all $H\in\mathcal{H}$, we have $g\in G$ with $\bigcap H\Subset g$.  For every $I\in\mathcal{H}_\between$, it follows that $O\Subset\bigcup I$ and hence $r\cap Y\subseteq\phi^{-1}[O]\Subset\phi^{-1}[\bigcup I]$, as $\phi$ is proximal.  For each $i\in I$, note $\phi^{-1}[i]$ is open in $Y$, as $\phi$ is continuous, so we have $q_i\in R$ with $q_i\cap Y=\phi^{-1}[i]$ and hence $r\cap Y\Subset\bigcup_{i\in I}q_i$.  As $Y$ is patch-dense in $Z$, it follows that $z\in r\subseteq\mathrm{cl}_\square(r\cap Y)\subseteq\bigcup_{i\in I}q_i$, so we must have $i\in I$ with $z\in q_i$ and $\phi[q_i\cap Y]=\phi[\phi^{-1}[i]]\subseteq i$.  As this holds for all $I\in\mathcal{H}_\between$, we must have $H\in\mathcal{H}$ such that, for every $h\in H$, we have $q_h\in R_z$ with $\phi[q_h\cap Y]\subseteq h$.  Setting $q=\bigcap_{h\in H}q_h\in R_z$, we see that $\phi[q\cap Y]\subseteq\bigcap H\Subset g$, for some $g\in G$, i.e. $g\in\phi(z)$.  This shows $\phi(z)$ is prime.

Now say $s\in\phi(z)$, so we have $r\in R_z$ with $\phi[r\cap Y]\Subset s$.  As $X$ is core compact, we have $\mathcal{F}\in\mathsf{FF}S$ with $\phi[r\cap Y]\Subset\bigcup_{F\in\mathcal{F}}\bigcap F\Subset s$.  The argument above then shows that $F\subseteq\phi(z)$, for some $F\in\mathcal{F}$.  As $\bigcap F\Subset s$, this shows that $\phi(z)$ is round and hence tight.  As $Z=Y^\leftarrow$, we have $y\in Y$ with $z\rightarrow y$ and hence $\psi(z)\supseteq\psi(y)=S_{\phi(y)}\neq\emptyset$.  Thus $\phi(z)\neq\emptyset$ and hence $\phi(z)\in\mathsf{T}^S$.

For continuity just note that if $\psi(z)\in\mathsf{T}_s$, i.e. $s\in\psi(z)$, then we have $r\in R_z$ with $\phi[r\cap Y]\Subset s$ and hence $z\in r\subseteq\psi^{-1}[s]$.  Next note that it suffices to verify proximality on a basis.  Accordingly, say $\mathsf{T}_F\Subset\mathsf{T}_G$, i.e. $\bigcap F\Subset g$, for all $g\in G$.  As $\phi$ is proximal, $\phi^{-1}[\bigcap F]\Subset\phi^{-1}[g]\subseteq\psi^{-1}[\mathsf{T}_g]$ and, again by patch-density, $\psi^{-1}[\mathsf{T}_F]\subseteq\mathrm{cl}_\square\phi^{-1}[\bigcap F]\Subset\psi^{-1}[\mathsf{T}_g]$.  As $Z$ is core coherent, it follows that $\psi^{-1}[\mathsf{T}_F]\Subset\bigcap_{g\in G}\psi^{-1}[\mathsf{T}_g]=\psi^{-1}[\mathsf{T}_G]$, showing that $\psi$ is proximal.  Further note that there can be no other proximal extension, as proximal maps between stably compact spaces are proper and hence patch-continuous, which means they are uniquely defined by their values on any patch-dense subspace.  As noted above, this also implies that $\psi$ is surjective, as $\mathsf{T}^S$ in its patch topology is locally compact and Hausdorff with (patch-)dense subspace $(S_x)_{x\in X}$.
\end{proof}

It follows that $\mathsf{T}^S$ is the unique minimal stabilisation of $X$, up to homeomorphism (and, in particular, does not depend on the chosen subbasis $S$).  Indeed, if $X'$ is another stabilisation that is minimal (in that it is universal for proximal maps as above) then we would have proximal maps from $X'$ to $\mathsf{T}^S$ and vice versa which reduce to the identity on $X$ and are thus homeomorphisms.

\begin{xpl}
Let $X=\mathbb{Z}\cup\{\omega,\omega'\}$ with the subbasis
\[S=\{\{n\}:n\in\mathbb{Z}\}\cup\{X\setminus\{x\}:x\in X\}.\]
So $X$ is the `two-point compactification' of $\mathbb{N}$, which is compact, locally compact, locally Hausdorff (in particular, sober) but not stable, as $X\setminus\{\omega\}$ and $X\setminus\{\omega'\}$ are compact open sets with non-compact intersection $\mathbb{N}=X\setminus\{\omega\}\cap X\setminus\{\omega'\}$.  One can check that the spectrum $\mathsf{T}^S$ consists of the points $S_x$ coming from $x\in X$ together with one additional tight subset, namely
\[T=S_{\omega}\cup S_{\omega'}=\{X\setminus\{x\}:x\in X\},\]
which is just enough to make the resulting space stably locally compact.

For an alternative `two-point stabilisation' of $X$, we can consider the larger space $X'=X\cup\{+\infty,-\infty\}$ with the subbasis
\[S'=\{\{n\}:n\in\mathbb{Z}\}\cup\{X'\setminus\{x\}:x\in X\}\cup\{\mathbb{Z}_+\cup\{+\infty\}\}\cup\{\mathbb{Z}_-\cup\{-\infty\}\}.\]
By \autoref{MinimalStabilisation}, we have a unique proximal map $\phi':X'\rightarrow\mathsf{T}^S$ with $\phi'(x)=S_x$, for all $x\in X$.  One can check that $\phi'$ maps both $+\infty$ and $-\infty$ to $T=S_{\omega}\cup S_{\omega'}$.
\end{xpl}

\section{Categorical Duality}\label{Functoriality}

\subsection{The Categories}

Now we want to make the duality in \autoref{ACovSSub} functorial with respect to appropriate morphisms.  On the topological side, we consider partial continuous maps with open domains, i.e. given $(R,Y),(S,X)\in\mathbf{Sub}$, we let
\[\mathbf{Sub}((R,Y),(S,X))=\{\phi\subseteq Y\times X:\phi\text{ is a continuous function on }\mathrm{dom}(\phi)\in\mathsf{O}X\}\]
(as mentioned at the beginning of \autoref{Preliminaries}, we take the domain to correspond to the right side of the Cartesian product, while the range corresponds to the left side -- this ensures that composition of functions in the usual order is given by $\circ$).  More explicitly, a function $\phi\subseteq Y\times X$ is a member of $\mathbf{Sub}((R,Y),(S,X))$ when
\[\phi(x)\in r\in R\qquad\Rightarrow\qquad\exists F\in\mathsf{F}S\ (x\in\bigcap F\subseteq\phi^{-1}[r]).\]
To simplify notation, we will freely abbreviate $\mathbf{Sub}((R,Y),(S,X))$ to $\mathbf{Sub}(R,S)$ or $\mathbf{Sub}(Y,X)$ whenever there is no risk of confusion.

If $\phi$ and $\pi$ are maps defined on $\mathrm{dom}(\phi)$ and $\mathrm{dom}(\pi)$ respectively then their composition $\phi\circ\pi$ is a map defined to be $\phi(\pi(x))$ wherever possible, i.e. for all
\[x\in\mathrm{dom}(\phi\circ\pi)=\{x\in\mathrm{dom}(\pi):\pi(x)\in\mathrm{dom}(\phi)\}.\]
It follows $\phi\circ\pi$ is continuous if $\phi$ and $\pi$ are, i.e. for any $(Q,Z),(R,Y),(S,X)\in\mathbf{Sub}$,
\[\mathbf{Sub}(Q,R)\circ\mathbf{Sub}(R,S)\subseteq\mathbf{Sub}(Q,S).\]
Thus these form the morphisms of a category which we again denote by $\mathbf{Sub}$ (where the identities in the category are the identity maps on the spaces).  Moreover, any $\phi\in\mathbf{Sub}(R,S)$ yields a relation $\sqsubset_\phi\ \subseteq\mathsf{F}S\times\mathsf{F}R$ given by
\[F\sqsubset_\phi G\qquad\Leftrightarrow\qquad\bigcap F\Subset\phi^{-1}[\bigcup G].\]
To get appropriate morphisms for $\mathbf{Cov}$, we need to axiomatise such relations.

\begin{dfn}
For any $(R,\vDash),(S,\vdash)\in\mathbf{Cov}$, let
\[\mathbf{Cov}((S,\vdash),(R,\vDash))=\{\sqsubset\ \subseteq\mathsf{F}S\times\mathsf{F}R:\ \sqsubset\ =\ \sqsubset\bullet\vDash\ =\ \vdash\bullet\sqsubset_{1\exists}\}.\]
We call $\sqsubset\ \in\mathbf{Cov}((S,\vdash),(R,\vDash))$ a \emph{cover morphism} from $(R,\vDash)$ to $(S,\vdash)$.
\end{dfn}

Again we will freely write $\mathbf{Cov}((S,\vdash),(R,\vDash))$ as $\mathbf{Cov}(S,R)$ or $\mathbf{Cov}(\vdash,\vDash)$ whenever there is no risk of confusion and refer to any $\sqsubset\ \in\mathbf{Cov}((S,\vdash),(R,\vDash))$ as a cover morphism `from $R$ to $S$' or `from $\vDash$ to $\vdash$'.  Note that every cover morphism $\sqsubset$ is monotone, indeed $\sqsubset\ =\ \vdash\bullet\sqsubset_{1\exists}$ is lower because $\vdash$ is lower and upper because $\sqsubset_{1\exists}$ is upper.  Also, as $\vdash$ and $\vDash$ are strong idempotents, we see that
\begin{align*}
\sqsubset\ &=\ \ \sqsubset\bullet\vDash\ \ =\ \ \sqsubset\bullet\vDash\bullet\vDash_{1\exists}\ \ =\ \ \sqsubset\bullet\vDash_{1\exists}\\
&=\ \ \vdash\bullet\sqsubset_{1\exists}\ \ =\ \ \vdash\bullet\vdash\bullet\sqsubset_{1\exists}\ \ =\ \ \vdash\bullet\sqsubset.
\end{align*}
In particular, every cover morphism is a \emph{Karoubi morphism}, i.e.
\[\tag{Karoubi Morphism}\label{Karoubi}\sqsubset\ =\ \sqsubset\bullet\vDash\ =\ \vdash\bullet\sqsubset.\]

\begin{rmk}\label{VickersKaroubi}
Instead of cover morphisms, Vickers considered Karoubi morphisms.  However these only correspond to preframe homomorphisms, and hence to certain closed relations (at least in the compact case -- see \cite{JungKegelmannMoshier2001}).  To get bona fide continuous maps, as in \autoref{Sp} below, we really do need these more restrictive cover morphisms, as also noted in \cite[Theorem 42]{Vickers2004} and \cite[Theorem 5.17]{Kawai2020}.
\end{rmk}

\begin{prp}
Cover morphisms are closed under cut-composition.
\end{prp}

\begin{proof}
Say we have cover systems $(S,\vdash),(R,\vDash),(Q,\vDdash)\in\mathbf{Cov}$ and cover morphisms $\sqsubset\ \in\mathbf{Cov}(S,R)$ and $\sqin\ \in\mathbf{Cov}(R,Q)$.  As $\sqin\ =\ \sqin\bullet\vDdash$, we immediately see that $\sqsubset\bullet\sqin\ =\ \sqsubset\bullet\sqin\bullet\vDdash$.  On the other hand, \autoref{1existsdiamond} yields
\[\sqsubset\bullet\sqin\ \ =\ \ \sqsubset\bullet\vDash\bullet\sqin_{1\exists}\ \ =\ \ \vdash\bullet\sqsubset_{1\exists}\bullet\sqin_{1\exists}\ \ \subseteq\ \ \vdash\bullet\mathrel{(\sqsubset\bullet\sqin)_{1\exists}}\ \ \subseteq\ \ \vdash\bullet\sqsubset\bullet\sqin\ \ \subseteq\ \ \sqsubset\bullet\sqin.\]
This shows that $\sqsubset\bullet\sqin\ \in\mathbf{Cov}(S,Q)$.
\end{proof}

Thus cover morphisms under cut-composition turn the the cover systems into a category, which we again denote by $\mathbf{Cov}$ (where the cover relations themselves are the identity morphisms).

\subsection{The Functors}

We consider $\mathbf{SSub}$ as a full subcategory of $\mathbf{Sub}$.

\begin{thm}\label{Ab}
We have a contravariant functor $\mathsf{Ab}:\mathbf{SSub}\rightarrow\mathbf{Cov}$ given by
\[\mathsf{Ab}(S,X)=(S,\vdash_\Subset)\qquad\text{and}\qquad\mathsf{Ab}(\phi)=\ \sqsubset_\phi.\]
\end{thm}

\begin{proof}
We already saw in \autoref{SubsetCex} that $(S,\vdash_\Subset)\in\mathbf{Cov}$ when $(S,X)\in\mathbf{Sub}$.  We also immediately see that $\sqsubset_\mathrm{id}\ =\ \vdash_\Subset$, i.e.  $\mathsf{Ab}$ preserves identity morphisms.  Next, given $\phi\in\mathbf{SSub}((S,X),(R,Y))$, we must show that $\sqsubset_\phi\ \in\mathbf{Cov}(R,S)$.

If $F\vdash_\Subset\mathcal{H}\bowtie\mathcal{I}\sqsubset_{\phi1\exists}G$ or even $\mathcal{I}\sqsubset_\phi G$ then
\[\bigcap F\subseteq\bigcap_{H\in\mathcal{H}}\bigcup H\subseteq\bigcup_{I\in\mathcal{I}}\bigcap I\Subset\phi^{-1}[\bigcup G]\]
and so $F\sqsubset_\phi G$.  Conversely, if $F\sqsubset_\phi G$ then, as $R$ is a subbasis of the core compact space $Y$, we can cover $\phi^{-1}[\bigcup G]$ with sets $\bigcap I_H$ where $\bigcap I_H\Subset\bigcap H\Subset\phi^{-1}[g]$, for some $g\in G$, and $H,I_H\in\mathsf{F}R$.  We then have a finite subcover of $\bigcap F$, i.e. we have $\mathcal{H}\in\mathsf{FF}R$ such that
\[\bigcap F\subseteq\bigcup_{H\in\mathcal{H}}\bigcap I_H\Subset\bigcup_{H\in\mathcal{H}}\bigcap H=\bigcap_{J\in\mathcal{H}_\between}\bigcup J,\]
and, for all $H\in\mathcal{H}$, we have $g\in G$ with $\bigcap H\Subset\phi^{-1}[g]$, i.e. $F\vdash_{\Subset\between}\mathcal{H}\sqsubset_{\phi1\exists}G$, showing that $\sqsubset_\phi\ =\ \vdash_\Subset\bullet\sqsubset_{\phi1\exists}$.

On the other hand, if $F\sqsubset_\phi\mathcal{H}\bowtie\mathcal{I}\vdash_\Subset G$ then, by core coherence,
\begin{equation}\label{RightAux}
\bigcap F\Subset\bigcap_{H\in\mathcal{H}}\phi^{-1}[\bigcup H]\subseteq\phi^{-1}[\bigcup_{I\in\mathcal{I}}\bigcap I]\subseteq\phi^{-1}[\bigcup G]
\end{equation}
and so $F\sqsubset_\phi G$.  Conversely, if $F\sqsubset_\phi G$ then, as $Y$ is core compact, we can take open $O\subseteq Y$ with $\bigcap F\Subset O\Subset\phi^{-1}[\bigcup G]$.  As $\phi$ is continuous and hence respects $\Subset$ on $\mathrm{dom}(\phi)$, it follows that $\phi[O]\Subset\phi[\phi^{-1}[\bigcup G]]\subseteq\bigcup G$.  As $S$ is a subbasis of the core compact space $X$, as above we have $\mathcal{H}\in\mathsf{F}\mathsf{F}S$ such that
\[\phi[O]\Subset\bigcup_{H\in\mathcal{H}}\bigcap H=\bigcap_{I\in\mathcal{H}_\between}\bigcup I\]
and, for all $H\in\mathcal{H}$, we have $g\in G$ with $\bigcap H\Subset g$ so $\mathcal{H}\vdash_\Subset G$.  Thus
\[\bigcap F\Subset O\subseteq\phi^{-1}[\phi[O]]\subseteq\phi^{-1}[\bigcap_{I\in\mathcal{H}_\between}\bigcup I]=\bigcap_{I\in\mathcal{H}_\between}\phi^{-1}[\bigcup I],\]
i.e. $F\sqsubset_\phi\mathcal{H}_\between$, showing that $\sqsubset_\phi\ =\ \sqsubset_\phi\bullet\vdash_\Subset$ and thus $\sqsubset_\phi\ \in\mathbf{Cov}(R,S)$.

Now we just have to show that $\mathsf{Ab}$ turns composition of continuous partial functions into cut-composition of the corresponding cover morphisms.  Accordingly, say $\phi\in\mathbf{SSub}((S,X),(R,Y))$ and $\psi\in\mathbf{Sub}((R,Y),(Q,Z))$.  If $F\sqsubset_{\psi\circ\phi}H$ then
\[\bigcap F\Subset(\psi\circ\phi)^{-1}[\bigcup H]=\phi^{-1}[\psi^{-1}[\bigcup H]].\]
Taking open $O\subseteq Z$ with $\bigcap F\Subset O\Subset\phi^{-1}[\psi^{-1}[\bigcup H]]$, we see that $\phi[O]\Subset\psi^{-1}[\bigcup H]$.  We then get $\mathcal{G}\in\mathsf{F}\mathsf{F}R$ with $\phi[O]\subseteq\bigcup_{G\in\mathcal{G}}\bigcap G\Subset\psi^{-1}[\bigcup H]$ so $F\sqsubset_{\phi\between}\mathcal{G}\sqsubset_\psi H$ and hence $F\sqsubset_\phi\bullet\sqsubset_\psi H$.

Conversely, say $F\sqsubset_\phi\bullet\sqsubset_{\psi}H$, so we have $\mathcal{G}\in\mathsf{FF}R$ with $F\sqsubset_{\phi\between}\mathcal{G}\sqsubset_{\psi}H$.  The first expression means $\bigcap F\Subset\phi^{-1}[\bigcup G]$, for all $G\in\mathcal{G}_\between$ so
\begin{equation}\label{circPreserving}
\bigcap F\Subset\bigcap_{G\in\mathcal{G}_\between}\phi^{-1}[\bigcup G]=\phi^{-1}[\bigcup_{G\in\mathcal{G}}\bigcap G]\subseteq\phi^{-1}[\psi^{-1}[\bigcup H]],
\end{equation}
again by core coherence.  Thus $F\sqsubset_{\psi\circ\phi}H$, showing that $\sqsubset_{\psi\circ\phi}\ =\ \sqsubset_\phi\bullet\sqsubset_{\psi}$.
\end{proof}

We can extend the above functor to core compact spaces if we restrict to proximal maps.  Accordingly, let $\mathbf{Sub_P}$ denote the wide subcategory of $\mathbf{Sub}$ consisting of proximal maps.  Given $\sqsubset\ \in\mathbf{Cov}((S,\vdash),(R,\vDash))$, let us also define $\sqsubseteq\ \subseteq\mathsf{F}R\times\mathsf{F}S$ by
\[F\sqsubseteq G\qquad\Leftrightarrow\qquad\forall H\in\mathsf{F}S\ (H\sqsubset\mathsf{1}F\ \Rightarrow\ H\vdash G).\]
We will show that proximal maps correspond to `proper' cover morphisms.

\begin{dfn}
We call $\sqsubset\ \in\mathbf{Cov}((S,\vdash),(R,\vDash))$ \emph{proper} if
\[\tag{Proper}\vDash\ \ \subseteq\ \ \sqsubseteq\bullet\sqsubset.\]
\end{dfn}

Let $\mathbf{Cov_P}$ denote the wide subcategory of $\mathbf{Cov}$ consisting of proper morphisms.

\begin{thm}
We have a contravariant functor $\mathsf{Ab_P}:\mathbf{Sub_P}\rightarrow\mathbf{Cov_P}$ given by
\[\mathsf{Ab_P}(S,X)=(S,\vdash_\Subset)\qquad\text{and}\qquad\mathsf{Ab_P}(\phi)=\ \sqsubset_\phi.\]
\end{thm}

\begin{proof}
The proof is mostly the same as the proof of \autoref{Ab} except we use the proximality of $\phi$ to obtain $\Subset$ in at the end rather than the start in \eqref{RightAux}, i.e.
\[\bigcap F\subseteq\bigcap_{H\in\mathcal{H}}\phi^{-1}[\bigcup H]\subseteq\phi^{-1}[\bigcup_{I\in\mathcal{I}}\bigcap I]\Subset\phi^{-1}[\bigcup G]\]
as well as in \eqref{circPreserving}, i.e.
\[\bigcap F\subseteq\bigcap_{G\in\mathcal{G}_\between}\phi^{-1}[\bigcup G]=\phi^{-1}[\bigcup_{G\in\mathcal{G}}\bigcap G]\Subset\phi^{-1}[\psi^{-1}[\bigcup H]].\]

The only extra thing we have to show is that $\sqsubset_\phi$ is proper whenever $\phi$ is proximal.  To see this, say $\phi\in\mathbf{Sub_P}((S,X),(R,Y))$ and $F,G\in\mathsf{F}S$ satisfy $F\vdash_\Subset G$, i.e. $\bigcap F\Subset\bigcup G$.  As $\phi$ is proximal, $\phi^{-1}[\bigcap F]\Subset\phi^{-1}[\bigcup G]$ so we have $\mathcal{H}\in\mathsf{FF}R$ with
\[\phi^{-1}[\bigcap F]\Subset\bigcap_{I\in\mathcal{H}_\between}\bigcup I=\bigcup_{H\in\mathcal{H}}\bigcap H\Subset\phi^{-1}[\bigcup G].\]
Thus $\mathcal{H}\sqsubset_\phi G$ and if $E\sqsubset_\phi\mathsf{1}F$ then $\bigcap E\subseteq\phi^{-1}[\bigcap F]\Subset\bigcap_{I\in\mathcal{H}_\between}\bigcup I$, showing that $F\sqsubseteq_\phi\mathcal{H}_\between$ and hence $\vdash_\Subset\ \subseteq\ \sqsubseteq_\phi\bullet\sqsubset_\phi$, i.e. $\sqsubset_\phi$ is proper.
\end{proof}

We call $\mathsf{Ab}$ the \emph{abstraction} functor, as opposed to \emph{spatialisation} or \emph{spectral} functor $\mathsf{Sp}$ we define next.  First, for any $\sqsubset\ \in\mathbf{Cov}((S,\vdash),(R,\vDash))$, define a function $\phi_\sqsubset$ on $\mathrm{dom}(\phi_\sqsubset)=\{T\in\mathsf{T}^S:\mathsf{F}T^{\sqsubset_1}\neq\emptyset\}$ by
\[\phi_\sqsubset(T)=\mathsf{F}T^{\sqsubset_1}=\{r\in R:\exists F\in\mathsf{F}T\ (F\sqsubset r)\}.\]

\begin{thm}
For any $\sqsubset\ \in\mathbf{Cov}((S,\vdash),(R,\vDash))$, $F\in\mathsf{F}S$ and $G\in\mathsf{F}R$,
\begin{equation}\label{FsqF'}
F\sqsubset G\qquad\Leftrightarrow\qquad\mathsf{T}_F\Subset\phi_\sqsubset^{-1}[\mathsf{T}^{G}],
\end{equation}
\end{thm}

\begin{proof}
Take $T\in\mathsf{T}^S$ with $\mathsf{F}T^{\sqsubset_1}\neq\emptyset$.  If $r\in\mathsf{F}T^{\sqsubset_1}$ then we have $F\in\mathsf{F}T$ with $F\sqsubset r$.  As $\sqsubset\ =\ \sqsubset\bullet\vDash\ =\ \vdash\bullet\sqsubset_{1\exists}$, we have $\mathcal{G}\in\mathsf{FF}S$ and $\mathcal{H}\in\mathsf{FF}R$ with $F\vdash_\between\mathcal{G}\sqsubset_{1\exists\between}\mathcal{H}\vDash r$.  As $T$ is prime, we then have $G\in\mathcal{G}\cap\mathsf{F}T$.  By \eqref{existsbetween}, we then have $H\in\mathcal{H}$ with $G\sqsubset_{1\forall}H\vDash r$ and hence $H\subseteq\mathsf{F}T^{\sqsubset_1}$, showing that $\mathsf{F}T^{\sqsubset_1}$ is round.

On the other hand, if $\mathsf{F}T^{\sqsubset_1}\supseteq F\vDash G$ then we have $H\in\mathsf{F}T$ with $H\sqsubset\mathsf{1}F\bowtie\{F\}\vDash G$ and hence $H\sqsubset G$, as $\sqsubset\bullet\vDash\ \subseteq\ \sqsubset$.  Then again we have $\mathcal{I}\in\mathsf{FF}S$ with $H\vdash_\between\mathcal{I}\sqsubset_{1\exists}G$.  As $T$ is prime, we again have $I\in\mathcal{I}\cap\mathsf{F}T$.  Taking $g\in G$ with $I\sqsubset g$, it follows that $g\in\mathsf{F}T^{\sqsubset_1}$, showing that $\mathsf{F}T^{\sqsubset_1}$ is also prime and hence tight.

The above argument showed that if $F\sqsubset G$ and $T\in\mathsf{T}_F$ then $\mathsf{F}T^{\sqsubset_1}\in\mathsf{T}^G$ and hence $\mathsf{T}_F\subseteq\phi_\sqsubset^{-1}[\mathsf{T}^G]$.  But if $F\sqsubset G$ then we have $\mathcal{H}\in\mathsf{FF}S$ with $F\vdash_\between\mathcal{H}\sqsubset G$ and hence $\mathsf{T}_F\Subset\mathsf{T}_\mathcal{H}\subseteq\phi_\sqsubset^{-1}[\mathsf{T}^G]$, by \eqref{TQSub}, which proves the $\Rightarrow$ part of \eqref{FsqF'}.

Conversely, say $\mathsf{T}_F\Subset\phi_\sqsubset^{-1}[\mathsf{T}^G]$.  By the definition of $\phi_\sqsubset$, we know that $(\mathsf{T}_H)_{H\sqsubset G}$ covers
$\phi_\sqsubset^{-1}[\mathsf{T}^G]$ and hence has a finite subcover of $\mathsf{T}_F$, i.e. $\mathsf{T}_F\subseteq\mathsf{T}_\mathcal{H}$, for some $\mathcal{H}\in\mathsf{FF}S$ with $\mathcal{H}\sqsubset G$.  By \eqref{FDGT}, this implies $F\Vdash_\between\mathcal{H}\sqsubset G$.  As
\[\Vdash\bullet\sqsubset\ \ =\ \ \Vdash\bullet\vdash\bullet\sqsubset\ \ =\ \ \vdash\bullet\sqsubset\ \ =\ \ \sqsubset,\]
this implies that $F\sqsubset G$, thus proving the $\Leftarrow$ part of \eqref{FsqF'}.
\end{proof}

Note that the identity morphism on a cover system $(S,\vdash)$ is just $\vdash$ itself.

\begin{thm}\label{Sp}
We have a contravariant functor $\mathsf{Sp}:\mathbf{Cov}\rightarrow\mathbf{SSub}$ given by
\[\mathsf{Sp}(S,\vdash)=((\mathsf{T}_p)_{p\in S},\mathsf{T}^S)\qquad\text{and}\qquad\mathsf{Sp}(\sqsubset)=\phi_\sqsubset.\]
Restricting to proper morphisms yields a functor $\mathsf{Sp_P}:\mathbf{Cov_P}\rightarrow\mathbf{SSub_P}$.
\end{thm}

\begin{proof}
By \autoref{TopRep}, $((\mathsf{T}_p)_{p\in S},\mathsf{T}^S)\in\mathbf{SSub}$ whenever $(S,\vdash)\in\mathbf{Cov}$.  Moreover, if $\sqsubset\ \in\mathbf{Cov}(S,R)$ and $\phi_\sqsubset(T)\in\mathsf{T}_r$, for some $r\in R$, then we have $F\in\mathsf{F} T $ with $F\sqsubset r$ and hence $T\in\mathsf{T}_F\Subset\phi_\sqsubset^{-1}[\mathsf{T}_r]$, by \eqref{FsqF'}.  This shows that $\phi_\sqsubset$ is a continuous partial map with open domain, i.e. $\phi_\sqsubset\in\mathbf{SSub}(\mathsf{T}^R,\mathsf{T}^S)$.

Now say $\sqin\ \in\mathbf{Cov}(R,Q)$.  If $q\in\phi_{\sqsubset\bullet\sqin}(T)$ then we have finite $F\subseteq T$ with $F\sqsubset\bullet\sqin q$ and hence we have $\mathcal{G}\in\mathsf{FF}R$ with $F\sqsubset_\between\mathcal{G}\sqin q$.  By \eqref{FsqF'}, $\phi_\sqsubset(T)\in\mathsf{T}^G$, for all $G\in\mathcal{G}_\between$, and hence we have $G\in\mathcal{G}$ with $\phi_\sqsubset(T)\in\mathsf{T}_G$.  Then \eqref{FsqF'} again yields $q\in\phi_{\sqin}(\phi_\sqsubset(T))$.  Conversely, if $q\in\phi_{\sqin}(\phi_\sqsubset(T))$ then we have $G\in\mathsf{F}\phi_\sqsubset(T)$ such that $G\sqin q$.  As $\phi_\sqsubset(T)\neq\emptyset$, we may assume $G\neq\emptyset$ and then we have $F\in\mathsf{F}T$ with $F\sqsubset\mathsf{1}G$ so $F\sqsubset\bullet\sqin q$ and hence $q\in\phi_{\sqsubset\bullet\sqin}(T)$.  This shows that $\phi_{\sqsubset\bullet\sqin}(T)=\phi_{\sqin}(\phi_\sqsubset(T))$ and hence $\phi_{\sqsubset\bullet\sqin}=\phi_{\sqin}\circ\phi_\sqsubset$.

The definition of tightness shows that $\phi_\vdash(T)=T$, for any $T\in\mathsf{T}^S$, where $\vdash$ is the cover relation on $S$, i.e. $\phi_\vdash$ is an identity morphism in $\mathbf{SSub}$.  Thus $\mathsf{Sp}$ is indeed a functor from $\mathbf{Cov}$ to $\mathbf{SSub}$.

If $\sqsubset\ \in\mathbf{Cov}_\mathbf{P}(S,R)$ then, for all $F,G\in\mathsf{F}R$ with $F\vdash\mathsf{1}G$ and all $g\in G$, we have $\mathcal{H}_g\in\mathsf{FF}S$ with $F\sqsubseteq_\between\mathcal{H}_g\sqsubset g$.  For any $T\in\phi^{-1}[\mathsf{T}_F]$, we have $I\in\mathsf{F}T$ with $I\sqsubset\mathsf{1}F$ so $I\vdash\mathcal{H}_{g\between}$ and hence $H\subseteq T$, for some $H\in\mathcal{H}_g$, i.e. $T\in\mathsf{T}_{\mathcal{H}_g}$.  This shows that $\phi^{-1}[\mathsf{T}_F]\subseteq\mathsf{T}_{\mathcal{H}_g}\Subset\phi^{-1}[\mathsf{T}_g]$.  Letting $\mathcal{H}=\bigwedge_{g\in G}\mathcal{H}_g$, the core coherence of $\mathsf{T}^S$ then yields $\phi^{-1}[\mathsf{T}_F]\subseteq\mathsf{T}_{\mathcal{H}}\Subset\phi^{-1}[\mathsf{T}_G]$.  This verifies proximality on a basis, which suffices because $\mathsf{T}^{R}$ is core compact.  Thus $\phi_\sqsubset\in\mathbf{SSub_P}(\mathsf{T}^R,\mathsf{T}^S)$, i.e. restricting yields a functor $\mathsf{Sp_P}:\mathbf{Cov_P}\rightarrow\mathbf{SSub_P}$.
\end{proof}

\subsection{The Natural Isomorphisms}

To show that the functors $\mathsf{Ab}$ and $\mathsf{Sp}$ witness a duality of categories, it only remains to define appropriate natural isomorphisms.

\begin{prp}
We have a natural isomorphism $\lambda:1_\mathbf{SSub}\rightarrow\mathsf{Sp}\mathsf{Ab}$ given by
\[\lambda_{(S,X)}(x)=S_x=\{p\in S:x\in p\}.\]
\end{prp}

\begin{proof}
By \eqref{Recovery}, $\lambda_{(S,X)}$ is an isomorphism in $\mathbf{SSub}(\mathsf{Sp}(\mathsf{Ab}(S,X)),(S,X))$, for all $(S,X)\in\mathbf{SSub}$.  We just have to show $\lambda_{(R,Y)}\circ\phi=\mathsf{Sp}(\mathsf{Ab}(\phi))\circ\lambda_{(S,X)}$, whenever $\phi\in\mathbf{SSub}(R,S)$, i.e. for all $x\in X$,
\[R_{\phi(x)}=\phi_{\sqsubset_\phi}(S_x).\]
But expanding definitions yields
\begin{align*}
r\in\phi_{\sqsubset_\phi}(S_x)\quad&\Leftrightarrow\quad\exists F\in\mathsf{F} S_x \ (F\sqsubset_\phi r)\\
&\Leftrightarrow\quad\exists F\in\mathsf{F} S \ (x\in\bigcap F\Subset\phi^{-1}[r])\\
&\Leftrightarrow\quad\phi(x)\in r,\quad\text{ as $\phi$ is continuous and $X$ is core compact},\\
&\Leftrightarrow\quad r\in R_{\phi(x)}.\qedhere
\end{align*}
\end{proof}

\begin{prp}
We have a natural isomorphism $\blacktriangleleft\ :1_\mathbf{Cov}\rightarrow\mathsf{Ab}\mathsf{Sp}$ given by
\[F_\mathsf{T}\blacktriangleleft_{(S,\vdash)}G\qquad\Leftrightarrow\qquad F\vdash G,\]
where $F_\mathsf{T}=\{\mathsf{T}_f:f\in F\}$.
\end{prp}

\begin{proof}
If $(S,X)\in\mathbf{Cov}$, $F,G,H\in\mathsf{F}S$ and $F_\mathsf{T}=G_\mathsf{T}$ then
\[F\vdash H\qquad\Leftrightarrow\qquad\mathsf{T}_F=\mathsf{T}_G\Subset\mathsf{T}^H\qquad\Leftrightarrow\qquad G\vdash H,\]
by \eqref{FCGT}.  Thus $\blacktriangleleft_{(S,\vdash)}$ is well-defined and $\blacktriangleleft_{(S,\vdash)}\ \in\mathbf{Cov}((S,\vdash),\mathsf{Ab}(\mathsf{Sp}(S,\vdash)))$, as $\vdash$ is cover relation.  We also see that $\blacktriangleleft_{(S,\vdash)}$ is an isomorphism with inverse given by
\[F\blacktriangleleft_{(S,\vdash)}^{-1}\{\mathsf{T}_g:g\in G\}\qquad\Leftrightarrow\qquad F\vdash G.\]
Given $\mathcal{F}\in\mathsf{FF}S$, let $\mathcal{F}_\mathsf{T}=\{F_\mathsf{T}:F\in\mathcal{F}\}$.  To show that $\blacktriangleleft_{(S,\vdash)}\bullet\sqsubset\ =\ \sqsubset_{\phi_\sqsubset}\bullet\blacktriangleleft_{(R,\vDash)}$ whenever $\sqsubset\ \in\mathbf{Cov}((S,\vdash),(R,\vDash))$, we again just expand the definitions, i.e.
\begin{align*}
F_\mathsf{T}\sqsubset_{\phi_\sqsubset}\bullet\blacktriangleleft_{(R,\vDash)}G\qquad\Leftrightarrow\qquad&\exists\mathcal{H}\in\mathsf{FF}R\ (F_\mathsf{T}\sqsubset_{\phi_\sqsubset\between}\mathcal{H}_\mathsf{T}\blacktriangleleft_{(R,\vDash)}G)\\
\Leftrightarrow\qquad&\exists\mathcal{H}\in\mathsf{FF}R\ (\mathsf{T}_F\Subset\bigcap_{I\in\mathcal{H}_\between}\phi_\sqsubset^{-1}[\mathsf{T}^I]\text{ and }\mathcal{H}\vdash G)\\
\Leftrightarrow\qquad&\exists\mathcal{H}\in\mathsf{FF}R\ (F\sqsubset_\between\mathcal{H}\vdash G),\quad\text{ by \eqref{FsqF'}},\\
\Leftrightarrow\qquad&F\sqsubset G\\
\Leftrightarrow\qquad&\exists\mathcal{H}\in\mathsf{F}\mathsf{F} S  \ (F\vdash_\between\mathcal{H}\sqsubset G)\\
\Leftrightarrow\qquad&\exists\mathcal{H}\in\mathsf{F}\mathsf{F} S  \ (F_\mathsf{T}\blacktriangleleft_{(S,\vdash)\between}\mathcal{H}\sqsubset G)\\
\Leftrightarrow\qquad&F_\mathsf{T}\blacktriangleleft_{(S,\vdash)}\bullet\sqsubset G.\qedhere
\end{align*}
\end{proof}

\begin{thm}\label{CategoricalDuality}
We have a categorical duality between $\mathbf{SSub}$ and $\mathbf{Cov}$ as witnessed by the adjoint functors $\mathsf{Ab}$ and $\mathsf{Sp}$ and natural isomorphisms $\lambda$ and $\blacktriangleleft$.
\end{thm}

\begin{proof}
It only remains to verify the zigzag identities, i.e. $\mathsf{Ab}\lambda\ \bullet\blacktriangleleft\!\mathsf{Ab}=1_\mathsf{Ab}$ and $\mathsf{Sp}\!\blacktriangleleft\ \circ\ \lambda\mathsf{Sp}=1_\mathsf{Sp}$.  The first equation means that $\sqsubset_{\lambda_{(S,X)}}\bullet\blacktriangleleft_{(S,\vdash_\Subset)}\ =\ \vdash_\Subset$, for all $(S,X)\in\mathbf{SSub}$.  Noting that
\[F\sqsubset_{\lambda_{(S,X)}}G_\mathsf{T}\quad\Leftrightarrow\quad\bigcap F\Subset\lambda_{(S,X)}^{-1}[\mathsf{T}^G]=\bigcup G\quad\Leftrightarrow\quad F\vdash_\Subset G,\]
we do indeed see that
\begin{align*}
F\sqsubset_{\lambda_{(S,X)}}\bullet\blacktriangleleft_{(S,\vdash_\Subset)}G\quad&\Leftrightarrow\quad\exists\mathcal{H}\in\mathsf{FF}S\ (F\sqsubset_{\lambda_{(S,X)}\between}\mathcal{H}_\mathsf{T}\blacktriangleleft_{(S,\vdash_\Subset)}G)\\
&\Leftrightarrow\quad\exists\mathcal{H}\in\mathsf{FF}S\ (F\vdash_{\Subset\between}\mathcal{H}\vdash_\Subset G)\\
&\Leftrightarrow\quad F\vdash G.
\end{align*}
The second equation means that $\phi_{\blacktriangleleft_{(S,\vdash)}}\circ\lambda_{(S_\mathsf{T},\mathsf{T}^S)}$ is the identity map on $\mathsf{T}^S$, for all $(S,\vdash)\in\mathbf{Cov}$, i.e. $\phi_{\blacktriangleleft_{(S,\vdash)}}(\lambda_{(S_\mathsf{T},\mathsf{T}^S)}(T))=T$, for all $T\in\mathsf{T}^S$.  Noting that
\[\lambda_{(S_\mathsf{T},\mathsf{T}^S)}(T)=\{\mathsf{T}_p:T\in\mathsf{T}_p\}=\{\mathsf{T}_p:p\in T\}=T_\mathsf{T},\]
we do indeed see again that
\begin{align*}
p\in\phi_{\blacktriangleleft_{(S,\vdash)}}(\lambda_{(S_\mathsf{T},\mathsf{T}^S)}(T))\quad&\Leftrightarrow\quad\exists F\in\mathsf{F}T\ (F_\mathsf{T}\blacktriangleleft_{(S,\vdash)}p)\\
&\Leftrightarrow\quad \exists F\in\mathsf{F}T\ (F\vdash p)\\
&\Leftrightarrow\quad p\in T.\qedhere
\end{align*}
\end{proof}

Restricting $\mathsf{Sp}$ to $\mathbf{ACov}$ immediately yields a duality between $\mathbf{SSub}$ and $\mathbf{ACov}$, giving us the functorialisation of \autoref{ACovSSub} that we set out to achieve.

\subsection{Karoubi Morphisms}

Let us finish with a brief note about Vickers original category.  As mentioned in \autoref{VickersKaroubi}, Vickers considered relations merely satisfying \eqref{Karoubi}.  Let us call relations $\vdash$ on $\mathsf{F}S$ and $\vDash$ on $\mathsf{F}R$ \emph{Karoubi isomorphic} if there are Karoubi morphisms $\sqsubset\ \subseteq\mathsf{F}R\times\mathsf{F}S$ and $\sqin\ \subseteq\mathsf{F}S\times\mathsf{F}R$ with
\[\sqin\bullet\sqsubset\ =\ \vdash\qquad\text{and}\qquad\sqsubset\bullet\sqin\ =\ \vDash.\]
While cover relations are more restrictive than Vickers' monotone cut-idempotents, the following result shows that they have the same Karoubi isomorphism classes (although to construct $\vDash$ below we need to move beyond finite subsets and consider infinite quasi-ideals).  For some similar results in the context of proximity lattices and other related categories, see \cite[Theorems 5.8 and 5.29]{Kawai2020}.

\begin{thm}\label{Karoubiisomorphic}
Any monotone cut-idempotent $\vdash$ on $\mathsf{F}S$ is Karoubi isomorphic to the cover relation $\vDash$ on $\mathsf{FQ}S$ given by
\[\mathscr{F}\vDash\mathscr{G}\qquad\Leftrightarrow\qquad\bigwedge\mathscr{F}\ll\bigvee\mathscr{G}.\]
\end{thm}

\begin{proof}
First note that $\vDash$ above is indeed a cover relation just like in \autoref{SubsetCex}.  Now define $\sqsubset\ \subseteq\mathsf{FQ}S\times\mathsf{F}S$ and $\sqin\ \subseteq\mathsf{F}S\times\mathsf{FQ}S$ by
\begin{align*}
\mathscr{F}\sqsubset G\qquad&\Leftrightarrow\qquad\bigwedge\mathscr{F}\ll G^\dashv\\
F\sqin\mathscr{G}\qquad&\Leftrightarrow\qquad F\in\bigvee\mathscr{G}.
\end{align*}

Say $\mathscr{F}\sqsubset\bullet\vdash G$, so we have $\mathcal{H}\in\mathsf{FF}S$ such that $\mathscr{F}\sqsubset_\between\mathcal{H}\vdash G$.  For all $I\in\mathcal{H}_\between$, this means $\bigwedge\mathscr{F}\ll I^\dashv$ and, in particular, $\bigwedge\mathscr{F}\subseteq I^\dashv$.  In other words, $\bigwedge\mathscr{F}\vdash\mathcal{H}_\between$ and hence $\bigwedge\mathscr{F}\ll G^\dashv$, by \eqref{WayBelow}, i.e. $\mathscr{F}\sqsubset G$.  Conversely, say $\mathscr{F}\sqsubset G$, i.e. $\bigwedge\mathscr{F}\ll G^\dashv$.  By \eqref{WayBelow}, we have $\mathcal{H}\in\mathsf{FF}S$ with $\bigwedge\mathscr{F}\vdash_\between\mathcal{H}\vdash G$.  As $\vdash\ \subseteq\ \vdash\bullet\vdash$ and hence ${}_\forall\hspace{-3pt}\vdash\ \subseteq{}_\forall\hspace{-3pt}\vdash\bullet\vdash$, by \eqref{foralldiamond}, we have $\mathcal{I}\in\mathsf{FF}S$ with $\mathcal{H}\vdash_\between\mathcal{I}\vdash G$ so $\bigwedge\mathscr{F}\ll J^\dashv$, for all $J\in\mathcal{I}_\between$, by \eqref{WayBelow}.  This means $\mathscr{F}\sqsubset\mathcal{I}_\between$ so $\mathscr{F}\sqsubset\bullet\vdash G$.  This shows that $\sqsubset\ =\ \sqsubset\bullet\vdash$.

Now if $\mathscr{F}\vDash\bullet\sqsubset G$ then we have $\mathfrak{F},\mathfrak{G}\in\mathsf{FFQ}S$ with $\mathscr{F}\vDash\mathfrak{F}\bowtie\mathfrak{G}\sqsubset G$ and hence
\[\bigwedge\mathscr{F}\subseteq\bigwedge_{\mathscr{H}\in\mathfrak{F}}\bigvee\mathscr{H}\subseteq\bigvee_{\mathscr{G}\in\mathfrak{G}}\bigwedge\mathscr{G}\ll G^\dashv.\]
Thus $\mathscr{F}\sqsubset G$.  Conversely, if $\mathscr{F}\sqsubset G$ then, taking $\mathcal{H},\mathcal{I}\in\mathsf{FF}S$ as above, we see that $\mathscr{F}\vDash\{\mathcal{I}\hspace{-3pt}\downarrow\}\bowtie\{\mathcal{I}\hspace{-3pt}\downarrow\}\sqsubset G$.  Thus $\sqsubset\ =\ \vDash\bullet\sqsubset$ is a Karoubi morphism from $S$ to $\mathsf{Q}S$.

Similarly, if $F\sqin\bullet\vDash\mathscr{G}$, then we have $\mathfrak{F},\mathfrak{G}\in\mathsf{FFQ}S$ with $F\sqin\mathfrak{F}\bowtie\mathfrak{G}\vDash\mathscr{G}$ so
\[F\in\bigcap_{\mathscr{F}\in\mathfrak{F}}\bigvee\mathscr{F}=\bigwedge_{\mathscr{F}\in\mathfrak{F}}\bigvee\mathscr{F}\subseteq\bigvee_{\mathscr{H}\in\mathfrak{G}}\bigwedge\mathscr{H}\ll\bigvee\mathscr{G}\]
and hence $F\sqin\mathscr{G}$.  Conversely, if $F\sqin\mathscr{G}$ then, as $\bigvee\mathscr{G}$ is a quasi-ideal, we have $\mathcal{H}\in\mathsf{F}(\bigvee\mathscr{G})$ with $F\vdash\mathcal{H}_\between$ so $F\sqin\{\mathcal{H}\hspace{-3pt}\downarrow\}\bowtie\{\mathcal{H}\hspace{-3pt}\downarrow\}\vDash\mathscr{G}$, showing that $\sqin\bullet\vDash\ =\ \sqin$.

On the other hand, if $F\vdash\bullet\sqin\mathscr{G}$, then we have $\mathcal{H}\in\mathsf{F}(\bigvee\mathscr{G})$ with $F\vdash\mathcal{H}_\between$ and hence $F\in\bigvee\mathscr{G}$, as $\bigvee\mathscr{G}$ is a quasi-ideal, so $F\sqin\mathscr{G}$.  Conversely, if $F\sqin\mathscr{G}$ then, taking $\mathcal{H}$ as above, we see that $F\vdash_\between\mathcal{H}\sqin\mathscr{G}$.  This shows that $\vdash\bullet\sqin\ =\ \sqin$ so $\sqin$ is a Karoubi morphism from $\mathsf{Q}S$ to $S$.

Now say $F\sqin\bullet\sqsubset G$, so we have $\mathfrak{F},\mathfrak{G}\in\mathsf{FFQ}S$ with $F\sqin\mathfrak{F}\bowtie\mathfrak{G}\sqsubset G$ and hence
\[F\in\bigcap_{\mathscr{F}\in\mathfrak{F}}\bigvee\mathscr{F}=\bigwedge_{\mathscr{F}\in\mathfrak{F}}\bigvee\mathscr{F}\subseteq\bigvee_{\mathscr{G}\in\mathfrak{G}}\bigwedge\mathscr{G}\ll G^\dashv.\]
Thus $F\in G^\dashv$, i.e. $F\vdash G$.  Conversely, if $F\vdash G$ then, as $\vdash\ \subseteq\ \vdash\bullet\vdash$, we have $\mathcal{H}\in\mathsf{FF}S$ with $F\vdash_\between\mathcal{H}\vdash G$ so $F\in\mathcal{H}\hspace{-3pt}\downarrow\ \ll G^\dashv$ and hence $F\sqin\{\mathcal{H}\hspace{-3pt}\downarrow\}\bowtie\{\mathcal{H}\hspace{-3pt}\downarrow\}\sqsubset G$.  This shows that $\vdash\ =\ \sqin\bullet\sqsubset$.

Finally, say $\mathscr{F}\sqsubset\bullet\sqin\mathscr{G}$, so we have $\mathcal{H}\in\mathsf{FF}S$ with $\mathcal{H}\sqin\mathscr{G}$ and $\mathscr{F}\sqsubset\mathcal{H}_\between$, i.e. $\bigwedge\mathscr{F}\ll I$, for all $I\in\mathcal{H}_\between$.  Thus $\mathcal{H}\subseteq\bigvee\mathscr{G}$ and $\bigwedge\mathscr{F}\vdash\mathcal{H}_\between$ so $\bigwedge\mathscr{F}\ll\bigvee\mathscr{G}$, i.e. $\mathscr{F}\vDash\mathscr{G}$.  Conversely, if $\mathscr{F}\vDash\mathscr{G}$ then we have $\mathcal{H}\in\mathsf{F}(\bigvee\mathscr{G})$ with $\bigwedge\mathscr{F}\vdash\mathcal{H}_\between$.  As $\bigvee\mathscr{G}$ is a quasi-ideal, we also have $\mathcal{I}\in\mathsf{F}(\bigvee\mathscr{G})$ with $\mathcal{H}\vdash\mathcal{I}_\between$ and hence $\mathscr{F}\sqsubset_\between\mathcal{I}\sqin\mathscr{G}$, showing that $\vDash\ =\ \sqsubset\bullet\sqin$.  Thus $S$ and $\mathsf{Q}S$ are Karoubi isomorphic.
\end{proof}

Another somewhat subtle difference between Vickers' and our work is that we interpret cover relations in general stably locally compact spaces, whereas Vickers only works with stably compact spaces (or, for the most part, their corresponding stably continuous frames).  Thus it can happen that a single cover relation can represent two different spaces in these different contexts.

For example, consider a cover relation like $\vdash$ defined on $\mathsf{F}\mathbb{N}$ by
\[F\vdash G\qquad\Leftrightarrow\qquad\mathrm{min}(F)\leq\mathrm{max}(G)\]
(where $\mathbb{N}=\{1,2,\ldots\}$ and we take $\min(\emptyset)=\infty$ and $\max(\emptyset)=0$, so that $\emptyset\not\vdash F$ and $F\not\vdash\emptyset$, for all $F\in\mathsf{F}\mathbb{N}$).  Note this corresponds to the canonical cover relation on the sets $N_k=\{1,\ldots,k\}$, for $k\in\mathbb{N}$, i.e.
\[F\vdash G\qquad\Leftrightarrow\qquad\bigcap_{k\in F}N_k\subseteq\bigcup_{k\in G}N_k.\]
On the one hand, $(N_k)_{k\in\mathbb{N}}$ is subbasis of $\mathbb{N}$ considered as a locally compact space in its lower topology.  On the other hand, $(N_k)_{k\in\mathbb{N}}$ is a more general kind of generating family (one which does not cover the whole space) of $\mathbb{N}\cup\{\infty\}$ considered as a compact space again in its lower topology.  In our setup, this relation $\vdash$ would correspond to the former situation, while in Vickers' setup it would correspond to the latter.  Indeed, if we wanted to be more faithful to Vickers' work then this could be achieved by not requiring subbases to cover the space in question and also allowing the empty set to be a part of our tight spectrum.

\bibliography{Maths}{}

\newcommand{\etalchar}[1]{$^{#1}$}
\begin{thebibliography}{GHK{\etalchar{+}}03}

\bibitem[BdR63]{BaayendeRijk1963}
P.C. Baayen and R.P.G. de~Rijk.
\newblock Compingent lattices and algebras.
\newblock {\em Indagationes Mathematicae (Proceedings)}, 66(Supplement C):591
  -- 602, 1963.
\newblock \href {http://dx.doi.org/10.1016/S1385-7258(63)50058-4}
  {\path{doi:10.1016/S1385-7258(63)50058-4}}.

\bibitem[BH14]{BezhanishviliHarding2014}
Guram Bezhanishvili and John Harding.
\newblock Stable compactifications of frames.
\newblock {\em Cah. Topol. G\'{e}om. Diff\'{e}r. Cat\'{e}g.}, 55(1):37--65,
  2014.

\bibitem[Bic21]{Bice2021GHLJS}
Tristan Bice.
\newblock Gr\"{a}tzer-{H}ofmann-{L}awson-{J}ung-{S}\"{u}nderhauf duality.
\newblock {\em Algebra Universalis}, 82(2):Paper No. 35, 13, 2021.
\newblock \href {http://dx.doi.org/10.1007/s00012-021-00729-2}
  {\path{doi:10.1007/s00012-021-00729-2}}.

\bibitem[BJ11]{BezhanishviliJansana2011}
Guram Bezhanishvili and Ramon Jansana.
\newblock Priestley style duality for distributive meet-semilattices.
\newblock {\em Studia Logica}, 98(1-2):83--122, 2011.
\newblock \href {http://dx.doi.org/10.1007/s11225-011-9323-5}
  {\path{doi:10.1007/s11225-011-9323-5}}.

\bibitem[BS18]{BiceStarling2016}
Tristan Bice and Charles Starling.
\newblock Locally compact {S}tone duality.
\newblock {\em J. Log. Anal.}, 10, 2018.
\newblock \href {http://dx.doi.org/10.4115/jla.2018.10.2}
  {\path{doi:10.4115/jla.2018.10.2}}.

\bibitem[BS20]{BiceStarling2020HTight}
Tristan Bice and Charles Starling.
\newblock Hausdorff tight groupoids generalised.
\newblock {\em Semigroup Forum}, 100(2):399--438, 2020.
\newblock \href {http://dx.doi.org/10.1007/s00233-019-10027-y}
  {\path{doi:10.1007/s00233-019-10027-y}}.

\bibitem[CC00]{CederquistCoquand2000}
Jan Cederquist and Thierry Coquand.
\newblock Entailment relations and distributive lattices.
\newblock In {\em Logic {C}olloquium '98 ({P}rague)}, volume~13 of {\em Lect.
  Notes Log.}, pages 127--139. Assoc. Symbol. Logic, Urbana, IL, 2000.

\bibitem[CG20]{CelaniGonzalez2020}
Sergio~A. Celani and Luciano~J. Gonz{\'a}lez.
\newblock A {C}ategorical {D}uality for {S}emilattices and {L}attices.
\newblock {\em Appl. Categ. Structures}, 28(5):853--875, 2020.
\newblock \href {http://dx.doi.org/10.1007/s10485-020-09600-2}
  {\path{doi:10.1007/s10485-020-09600-2}}.

\bibitem[CZ03]{CoquandZhang2003}
Thierry Coquand and Guo-Qiang Zhang.
\newblock A representation of stably compact spaces, and patch topology.
\newblock volume 305, pages 77--84. 2003.
\newblock Topology in computer science (Schlo\ss Dagstuhl, 2000).
\newblock \href {http://dx.doi.org/10.1016/S0304-3975(02)00695-3}
  {\path{doi:10.1016/S0304-3975(02)00695-3}}.

\bibitem[Exe08]{Exel2008}
Ruy Exel.
\newblock Inverse semigroups and combinatorial {$C^\ast$}-algebras.
\newblock {\em Bull. Braz. Math. Soc. (N.S.)}, 39(2):191--313, 2008.
\newblock \href {http://dx.doi.org/10.1007/s00574-008-0080-7}
  {\path{doi:10.1007/s00574-008-0080-7}}.

\bibitem[GHK{\etalchar{+}}03]{GierzHofmannKeimelLawsonMisloveScott2003}
G.~Gierz, K.~H. Hofmann, K.~Keimel, J.~D. Lawson, M.~Mislove, and D.~S. Scott.
\newblock {\em Continuous lattices and domains}, volume~93 of {\em Encyclopedia
  of Mathematics and its Applications}.
\newblock Cambridge University Press, Cambridge, 2003.
\newblock \href {http://dx.doi.org/10.1017/CBO9780511542725}
  {\path{doi:10.1017/CBO9780511542725}}.

\bibitem[GL13]{Goubault2013}
Jean Goubault-Larrecq.
\newblock {\em Non-{H}ausdorff topology and domain theory}, volume~22 of {\em
  New Mathematical Monographs}.
\newblock Cambridge University Press, Cambridge, 2013.
\newblock [On the cover: Selected topics in point-set topology].
\newblock \href {http://dx.doi.org/10.1017/CBO9781139524438}
  {\path{doi:10.1017/CBO9781139524438}}.

\bibitem[Gr{\"a}78]{Gratzer1978}
George Gr{\"a}tzer.
\newblock {\em General lattice theory}, volume~75 of {\em Pure and Applied
  Mathematics}.
\newblock Academic Press, Inc. [Harcourt Brace Jovanovich, Publishers], New
  York-London, 1978.

\bibitem[HL78]{HofmannLawson1978}
Karl~H Hofmann and Jimmie~D Lawson.
\newblock The spectral theory of distributive continuous lattices.
\newblock {\em Transactions of the American Mathematical Society},
  246:285--310, 1978.

\bibitem[HP08]{HansoulPoussart2008}
G.~Hansoul and C.~Poussart.
\newblock Priestley duality for distributive semilattices.
\newblock {\em Bull. Soc. Roy. Sci. Li\`ege}, 77:104--119, 2008.

\bibitem[JKM99]{JungKegelmannMoshier1999}
Achim Jung, Mathias Kegelmann, and M.~Andrew Moshier.
\newblock Multilingual sequent calculus and coherent spaces.
\newblock {\em Fund. Inform.}, 37(4):369--412, 1999.
\newblock \href {http://dx.doi.org/10.3233/FI-1999-37403}
  {\path{doi:10.3233/FI-1999-37403}}.

\bibitem[JKM01]{JungKegelmannMoshier2001}
Achim Jung, Mathias Kegelmann, and M.Andrew Moshier.
\newblock Stably compact spaces and closed relations.
\newblock {\em Electronic Notes in Theoretical Computer Science}, 45:209--231,
  2001.
\newblock MFPS 2001,Seventeenth Conference on the Mathematical Foundations of
  Programming Semantics.
\newblock \href {http://dx.doi.org/10.1016/S1571-0661(04)80964-2}
  {\path{doi:10.1016/S1571-0661(04)80964-2}}.

\bibitem[Joh86]{Johnstone1986}
Peter~T. Johnstone.
\newblock {\em Stone spaces}, volume~3 of {\em Cambridge Studies in Advanced
  Mathematics}.
\newblock Cambridge University Press, Cambridge, 1986.
\newblock Reprint of the 1982 edition.

\bibitem[JS96]{JungSunderhauf1995}
Achim Jung and Philipp S\"{u}nderhauf.
\newblock On the duality of compact vs. open.
\newblock In {\em Papers on general topology and applications ({G}orham, {ME},
  1995)}, volume 806 of {\em Ann. New York Acad. Sci.}, pages 214--230. New
  York Acad. Sci., New York, 1996.
\newblock \href {http://dx.doi.org/10.1111/j.1749-6632.1996.tb49171.x}
  {\path{doi:10.1111/j.1749-6632.1996.tb49171.x}}.

\bibitem[Kaw20]{Kawai2020}
Tatsuji Kawai.
\newblock Presenting de {G}root duality of stably compact spaces.
\newblock {\em Theoret. Comput. Sci.}, 823:44--68, 2020.
\newblock \href {http://dx.doi.org/10.1016/j.tcs.2020.03.002}
  {\path{doi:10.1016/j.tcs.2020.03.002}}.

\bibitem[Kaw21]{Kawai2021}
Tatsuji Kawai.
\newblock Predicative theories of continuous lattices.
\newblock {\em Log. Methods Comput. Sci.}, 17(2):Paper No. 22, 38, 2021.
\newblock \href {http://dx.doi.org/10.23638/LMCS-17(2:22)2021}
  {\path{doi:10.23638/LMCS-17(2:22)2021}}.

\bibitem[Kub02]{Kubis2002}
Wies{\l}aw Kubi\'{s}.
\newblock Separation properties of convexity spaces.
\newblock {\em J. Geom.}, 74(1-2):110--119, 2002.
\newblock \href {http://dx.doi.org/10.1007/PL00012529}
  {\path{doi:10.1007/PL00012529}}.

\bibitem[Law91]{Lawson1991}
J.~D. Lawson.
\newblock Order and strongly sober compactifications.
\newblock In {\em Topology and category theory in computer science ({O}xford,
  1989)}, Oxford Sci. Publ., pages 179--205. Oxford Univ. Press, New York,
  1991.

\bibitem[Len08]{Lenz2008}
Daniel~H. Lenz.
\newblock On an order-based construction of a topological groupoid from an
  inverse semigroup.
\newblock {\em Proc. Edinb. Math. Soc. (2)}, 51(2):387--406, 2008.
\newblock \href {http://dx.doi.org/10.1017/S0013091506000083}
  {\path{doi:10.1017/S0013091506000083}}.

\bibitem[PP12]{PicadoPultr2012}
Jorge Picado and Ale{\v{s}} Pultr.
\newblock {\em Frames and locales: Topology without points}.
\newblock Frontiers in Mathematics. Birkh\"auser/Springer Basel AG, Basel,
  2012.
\newblock \href {http://dx.doi.org/10.1007/978-3-0348-0154-6}
  {\path{doi:10.1007/978-3-0348-0154-6}}.

\bibitem[Pri70]{Priestley1970}
H.~A. Priestley.
\newblock Representation of distributive lattices by means of ordered stone
  spaces.
\newblock {\em Bull. London Math. Soc.}, 2:186--190, 1970.
\newblock \href {http://dx.doi.org/10.1112/blms/2.2.186}
  {\path{doi:10.1112/blms/2.2.186}}.

\bibitem[Sco74]{Scott1974}
Dana Scott.
\newblock Completeness and axiomatizability in many-valued logic.
\newblock In {\em Proceedings of the {T}arski {S}ymposium ({P}roc. {S}ympos.
  {P}ure {M}ath., {V}ol. {XXV}, {U}niv. {C}alifornia, {B}erkeley, {C}alif.,
  1971)}, pages 411--435, 1974.

\bibitem[{Shi}52]{Shirota1952}
Taira {Shirota}.
\newblock {A generalization of a theorem of I. Kaplansky.}
\newblock {\em {Osaka Math. J.}}, 4:121--132, 1952.
\newblock URL: \url{http://projecteuclid.org/euclid.ojm/1200687806}.

\bibitem[Smy86]{Smyth1986}
M.~B. Smyth.
\newblock Finite approximation of spaces (extended abstract).
\newblock In {\em Category theory and computer programming ({G}uildford,
  1985)}, volume 240 of {\em Lecture Notes in Comput. Sci.}, pages 225--241.
  Springer, Berlin, 1986.
\newblock \href {http://dx.doi.org/10.1007/3-540-17162-2\_125}
  {\path{doi:10.1007/3-540-17162-2\_125}}.

\bibitem[Smy92]{Smyth1992}
M.~B. Smyth.
\newblock Stable compactification. {I}.
\newblock {\em J. London Math. Soc. (2)}, 45(2):321--340, 1992.
\newblock \href {http://dx.doi.org/10.1112/jlms/s2-45.2.321}
  {\path{doi:10.1112/jlms/s2-45.2.321}}.

\bibitem[Sto36]{Stone1936}
M.~H. Stone.
\newblock The theory of representations for {B}oolean algebras.
\newblock {\em Trans. Amer. Math. Soc.}, 40(1):37--111, 1936.
\newblock \href {http://dx.doi.org/10.2307/1989664}
  {\path{doi:10.2307/1989664}}.

\bibitem[Sto38]{Stone1938}
M.~H. Stone.
\newblock Topological representations of distributive lattices and {B}rouwerian
  logics.
\newblock {\em Cat. Mat. Fys.}, 67(1):1--25, 1938.
\newblock URL: \url{http://hdl.handle.net/10338.dmlcz/124080}.

\bibitem[vdV93]{VanDeVel1993}
M.~L.~J. van~de Vel.
\newblock {\em Theory of convex structures}, volume~50 of {\em North-Holland
  Mathematical Library}.
\newblock North-Holland Publishing Co., Amsterdam, 1993.

\bibitem[vG12]{vanGool2012}
Sam~J. van Gool.
\newblock Duality and canonical extensions for stably compact spaces.
\newblock {\em Topology Appl.}, 159(1):341--359, 2012.
\newblock \href {http://dx.doi.org/10.1016/j.topol.2011.09.040}
  {\path{doi:10.1016/j.topol.2011.09.040}}.

\bibitem[Vic04]{Vickers2004}
Steven Vickers.
\newblock Entailment systems for stably locally compact locales.
\newblock {\em Theoret. Comput. Sci.}, 316(1-3):259--296, 2004.
\newblock \href {http://dx.doi.org/10.1016/j.tcs.2004.01.033}
  {\path{doi:10.1016/j.tcs.2004.01.033}}.

\bibitem[Vri62]{deVries1962}
H.~De Vries.
\newblock Compact spaces and compactifications: An algebraic approach.
\newblock Thesis Amsterdam, 1962.

\bibitem[Wal38]{Wallman1938}
Henry Wallman.
\newblock Lattices and topological spaces.
\newblock {\em Ann. of Math. (2)}, 39(1):112--126, 1938.
\newblock \href {http://dx.doi.org/10.2307/1968717}
  {\path{doi:10.2307/1968717}}.

\end{thebibliography}
\bibliographystyle{alphaurl}

\end{document}